\documentclass{amsart}

\usepackage{amsmath, amssymb, amsfonts, amsthm, mathtools, mathrsfs}
\usepackage[all,cmtip,2cell]{xy}
\UseAllTwocells
\usepackage{tikz-cd}
\usepackage{graphicx}
\usepackage{hyperref}

\counterwithin{equation}{section}

\newcommand{\pbcorner}[1][ul]{\save*!/#1+1.2pc/#1:(1,-1)@^{|-}\restore}

\newcommand{\map}[3]{#1\colon#2\to#3}

\DeclareMathOperator{\vdim}{vdim}

\theoremstyle{plain}
\newtheorem{thm}{Theorem}[section]
\newtheorem{lem}[thm]{Lemma}
\newtheorem{prop}[thm]{Proposition}
\newtheorem{cor}[thm]{Corollary}
\newtheorem{conj}[thm]{Conjecture}

\newtheorem*{thm*}{Theorem}
\newtheorem*{lem*}{Lemma}
\newtheorem*{prop*}{Proposition}
\newtheorem*{cor*}{Corollary}
\newtheorem*{conj*}{Conjecture}

\theoremstyle{definition}
\newtheorem{dfn}[thm]{Definition}
\newtheorem{ex}[thm]{Example}
\newtheorem{rem}[thm]{Remark}

\newtheorem*{dfn*}{Definition}
\newtheorem*{ex*}{Example}
\newtheorem*{rem*}{Remark}

\newcommand\ab{\allowbreak}
\newcommand\bs{\boldsymbol}
\newcommand\ds{\displaystyle}

\newcommand\id{\mathop{\mathrm{id}}\nolimits}

\newcommand\Coh{\mathrm{Coh}}

\newcommand\crit{\mathop{\mathrm{Crit}}\nolimits}
\newcommand\Crit{\mathop{\mathrm{Crit}}\nolimits}
\newcommand\dcrit{\mathop{\mathbf{Crit}}\nolimits}

\newcommand{\dg}{\mathrm{dg}}
\newcommand\dsch{\mathop{\mathbf{dSch}}\nolimits}
\newcommand\dst{\mathop{\mathbf{dSt}}\nolimits}
\newcommand\Spec{\mathop{\mathrm{Spec}}\nolimits}
\newcommand\dSpec{\mathop{\mathbf{Spec}}\nolimits}

\newcommand{\sHom}{\mathop{\mathcal{H}\! \mathit{om}}\nolimits}
\newcommand{\sRHom}{\mathop{\mathit{R}\mathcal{H}\! \mathit{om}}\nolimits}
\newcommand{\per}{\mathrm{per}}
\newcommand{\RHom}{\mathop{R\mathrm{Hom}}\nolimits}
\newcommand{\Qcoh}{\mathop{\mathrm{Qcoh}}}
\newcommand\Sym{\mathop{\mathrm{Sym}}\nolimits}
\newcommand\Tot{\mathop{\mathrm{Tot}}\nolimits}
\newcommand\op{\mathrm{op}}
\newcommand\pbc{\mathfrak{pbc}}
\newcommand\red{\mathrm{red}}

\renewcommand\ss{\normalfont{\text{-ss}}}
\newcommand\vir{\mathrm{vir}}

\newcommand\rot{\mathrm{rot}}
\newcommand\rank{\mathop{\mathrm{rank}}\nolimits}
\newcommand\Perf{\mathrm{Perf}}

\newcommand\lisan{\mathop{\normalfont{\text{Lis-An}}}\nolimits}

\renewcommand\L{\mathbb{L}}
\newcommand\T{\mathbb{T}}

\newcommand\D{{\mathrm d}}
\newcommand\mH{\mathrm{H}}
\newcommand\hi{\hat{\imath}}
\newcommand\ddr{d_{\mathrm{dR}}}
\newcommand\ddrc{\D_{\mathrm{dR}}} 

\renewcommand\det{\mathop{\mathrm{det}}\nolimits}
\newcommand\detr{\mathop{\widehat{\mathrm{det}}}\nolimits}
\newcommand{\simd}{\mathrel{\rotatebox[origin=c]{-90}{$\sim$}}}
\newcommand{\simu}{\mathrel{\rotatebox[origin=c]{90}{$\sim$}}}

\newcommand{\mfrakM}{\mathfrak{M}}

\newcommand{\mfrakW}{\mathfrak{W}}
\newcommand{\mfrakX}{\mathfrak{X}}
\newcommand{\mfrakY}{\mathfrak{Y}}
\newcommand{\mfrakZ}{\mathfrak{Z}}

\newcommand{\setC}{\mathbb{C}}
\newcommand{\setQ}{\mathbb{Q}}

\newcommand{\setZ}{\mathbb{Z}}

\newcommand{\bfT}{\mathbf{T}}

\newcommand{\ts}{\textstyle}

\title{Dimensional reduction in cohomological Donaldson--Thomas theory}

\author{Tasuki Kinjo}

\email{tasuki.kinjo@ipmu.jp}

\address{graduate school of mathematical science, the university of tokyo, 3-8-1 komaba, meguroku, tokyo 153-8914, japan.}

\begin{document}

\begin{abstract}

For oriented $-1$-shifted symplectic derived Artin stacks, Ben-Bassat--Brav--Bussi--Joyce introduced certain perverse sheaves on them
which can be regarded as sheaf theoretic categorifications of the Donaldson--Thomas invariants.
In this paper, we prove that the hypercohomology of the above perverse sheaf
on the $-1$-shifted cotangent stack over a quasi-smooth derived Artin stack
is isomorphic to the Borel--Moore homology of the base stack up to a certain shift of degree.
This is a global version of the dimensional reduction theorem due to Davison.

We give two applications of our main theorem.
Firstly, we apply it to the study of the cohomological Donaldson--Thomas invariants for local surfaces.
Secondly, regarding our main theorem as a version of Thom isomorphism theorem for dual obstruction cones,
we propose a sheaf theoretic construction of the virtual fundamental classes for quasi-smooth derived Artin stacks.
\end{abstract}

\maketitle

\setcounter{tocdepth}{1}
\tableofcontents

\section{Introduction}

\subsection{Motivations}
For a Calabi--Yau $3$-fold $X$, Thomas introduced enumerative invariants in \cite{Tho} which is now called the Donaldson--Thomas invariants (DT invariants for short)
which virtually count stable sheaves on $X$.
Later several variations and generalizations of DT invariants were introduced.
One such example is the DT invariants for quivers with potentials first introduced in \cite{Sze},
which is now understood as the local version of the original DT invariants.
Another example is cohomological Donaldson--Thomas theory (CoDT theory for short),
which studies the sheaf theoretic categorification of DT invariants.

CoDT theory was initiated by the work of Kontsevich--Soibelman \cite{KoSo} in the case of quivers with potentials.
It studies the vanishing cycle cohomologies of the moduli stacks of representations over Jacobi algebras associated with quivers with potentials.
Later \cite{BBBBJ, BBDJS, JU} opened the door to CoDT theory for CY 3-folds by defining a natural perverse sheaf $\varphi_{\mfrakM_X^{H\text{-ss}}}$ (resp. $\varphi_{\mathcal{M}_X^{H\text{-st}}}$) on the moduli stack $\mfrakM_X^{H\text{-ss}}$ of compactly supported $H$-semistable sheaves
(resp. the moduli scheme $\mathcal{M}_X^{H\text{-st}}$ of $H$-stable sheaves) on a CY $3$-fold $X$ with a fixed ample divisor $H$, which can be regarded as a categorification of the original DT invariant in the following sense:
for a compact component $\mathcal{N} \subset \mathcal{M}_X^{H\text{-st}}$ we have
\[
\int_{[\mathcal{N}]^{\vir}} 1 = \sum_i(-1)^i \dim \mH^i(\mathcal{N}; \varphi_{\mathcal{M}_X^{H\text{-st}}}|_{\mathcal{N}}).
\]
where $[\mathcal{N}]^{\vir}$ denotes the virtual fundamental class of $\mathcal{N}$.

CoDT theory for quivers with potentials is well-developed and shown to have a rich theory.
For example, Kontsevich--Soibelman in \cite{KoSo} constructed algebra structures called
\emph{critical cohomological Hall algebras} (critical CoHAs for short) on the CoDT invariants
for quivers with potentials.
Later Davison--Meinhardt in \cite{DM} proved the wall crossing formulas for CoDT invariants of quivers with potentials,
and realized it as natural maps induced by the CoHA multiplications.
In contrast to these developments, almost nothing is known concerning CoDT theory for CY $3$-folds
though it is expected that the local theory as above can be extended to the global settings.
The aim of this paper is to make the first step towards the development of CoDT theory for local surfaces (i.e. CY $3$-folds of the form $\Tot_S(\omega_S)$ where $S$ is a smooth surface) by proving a global version of the dimensional reduction theorem \cite[Theorem A.1]{Dav1}.

\subsection{Dimensional reduction}
In this paper, we always use the term ``dimensional reduction'' as a statement that relates three-dimensional things to
two-dimensional things.
A dimensional reduction in DT theory was first observed by the work of \cite{BBS} in the motivic context.
They computed the motivic refinement of the DT invariant of zero-dimensional closed subschemes of $\mathbb{C}^3$
by relating it with the motive of the moduli stack of zero-dimensional sheaves on $\mathbb{C}^2$.
Later Davison proved a categorified version of the dimensional reduction theorem in \cite{Dav1}, which we briefly recall now.
Let $U$ be a scheme, $E$ be a vector bundle on $U$, $s \in \Gamma(U, E)$ be a section,
and $\bar{s}$ be the regular function on $\Tot_U(E^\vee)$ corresponding to $s$.
Denote by $\pi \colon \Tot_U(E^\vee) \to U$ the projection.
Write $Z = s^{-1}(0)$. Then \cite[Theorem A.1]{Dav1} states that we have an isomorphism
\begin{align}\label{eq:locdimred}
 \pi_! \varphi^p_{\bar{s}}(\setQ_{\Tot_U(E^\vee)}) |_Z \cong \setQ_Z[-2\rank E]
\end{align}
where $\varphi^p_{\bar{s}}$ denotes the vanishing cycle functor.
This statement has many applications in CoDT theory for quiver with potentials.
For example, Davison in \cite{Dav2} computed the vanishing cycle cohomology for the Hilbert scheme of points in $\mathbb{C}^3$ by using the quiver description and applying \eqref{eq:locdimred}.
Another interesting application of the isomorphism \eqref{eq:locdimred} also discussed in \cite{Dav2}
is the study of compactly supported cohomologies of the moduli stacks of representations of preprojective algebras, e.g. the purity of their Hodge structures.
Therefore the isomorphism \eqref{eq:locdimred} not only allows us to compute the vanishing cycle cohomology but also can be used to import three-dimensional techniques to the study of two-dimensional objects.

The main theme of this paper is to globalize the isomorphism \eqref{eq:locdimred} so that it can be applied to the study of
CoDT theory for local surfaces.
Before stating the main theorem of this paper, we briefly recall the construction of the perverse sheaves introduced in \cite{BBBBJ, BBDJS}.
Let $(\bs{\mfrakX}, \omega)$ be a $-1$-shifted symplectic derived Artin stack.
By the work of \cite{PTVV, BD2}, the derived moduli stacks of coherent sheaves with proper supports on
smooth CY $3$-folds give such examples.
It is shown in \cite{BBBBJ} that $\bs{\mfrakX}$ is locally (in the smooth topology) written as a critical locus.
If we are given a square root of  the line bundle $\det(\L_{\bs{\mfrakX}} |_{\bs{\mfrakX} ^{\red}})$
\[
o \colon M^{\otimes^2} \cong \det(\L_{\bs{\mfrakX}}|_{\bs{\mfrakX} ^{\red}}) ,
\]
which is called an \emph{orientation} for $\bs{\mfrakX}$, we can construct a natural perverse sheaf
\[
\varphi^p_{\bs{\mfrakX}} = \varphi^p_{\bs{\mfrakX}, \omega, o}
\]
on $t_0(\bs{\mfrakX})$ locally isomorphic to the vanishing cycle complex twisted by a certain local system.
Our main theorem in this paper is as follows:

\begin{thm*}{\normalfont{(Theorem \ref{thm:dimred2})}}
Let $\bs{\mfrakY}$ be a quasi-smooth derived Artin stack,
and $\bfT^*[-1] \bs{\mfrakY} \coloneqq \ab \dSpec_{\bs{\mfrakY}}(\Sym(\mathbb{T}_{\bs{\mfrakY}}[1]))$
be the $-1$-shifted cotangent stack.
Equip $\bfT^*[-1] \bs{\mfrakY}$ with the canonical $-1$-shifted symplectic structure $\omega_{\bfT^*[-1] \bs{\mfrakY}}$ and the canonical orientation $o_{\bfT^*[-1] \bs{\mfrakY}}$ (see Example \ref{ex:cot} and Example \ref{ex:canori}).
If we write $\pi  \colon t_0( \bfT^*[-1] \bs{\mfrakY}) \to t_0(\bs{\mfrakY})$ the projection, we have a natural isomorphism
\begin{align}\label{eq:intdimred}
\pi_! \varphi^p_{\bfT^*[-1] \bs{\mfrakY}} \cong \setQ_{t_0(\bs{\mfrakY})}[\vdim \bs{\mfrakY}].
\end{align}
\end{thm*}

Now we return back to the story of the dimensional reduction in CoDT theory.
Consider a smooth surface $S$ and write $X = \Tot_S(\omega_S)$.
Denote by $\bs{\mfrakM}_{S}$ (resp. $\bs{\mfrakM}_{X}$) the derived moduli stack of compactly supported coherent sheaves on $S$ (resp. $X$), and we write $\mfrakM_{S} = t_0(\bs{\mfrakM}_{S})$ and $\mfrakM_{X} = t_0(\bs{\mfrakM}_{X})$.
By applying the work of \cite{BCS} and \cite{IQ},
we can show that there is a natural equivalence between $\bs{\mfrakM}_{X}$ and $\bfT^*[-1] \bs{\mfrakM}_{S}$ over $\bs{\mfrakM}_{S}$ preserving the $-1$-shifted symplectic structure.
Therefore we obtain the following corollary:

\begin{cor*}{\normalfont{(Corollary \ref{cor:dimred3})}}
  Let $\bs{\mfrakM}_{X}$ and $\bs{\mfrakM}_{S}$ be as above,
  and equip $\bs{\mfrakM}_{X}$ with the orientation induced by the canonical orientation on $\bfT^*[-1] \bs{\mfrakM}_{S}$
  \footnote{In \cite{JU}, natural orientation data for a wide class of CY $3$-folds including
  all projective ones and local surfaces are constructed using gauge theoretic techniques.
  In \cite[Remark 4.12]{JU} it is conjectured that our choice coincides with theirs for local surfaces.}.
  If we write $\pi \colon \mfrakM_{X} \to \mfrakM_{S}$ the canonical projection,
  we have an isomorphism
  \begin{align}\label{eq:intdimred2}
  \pi_! \varphi^p_{\bs{\mfrakM}_{X}} \cong \setQ_{\bs{\mfrakM}_{S}}[\vdim \bs{\mfrakM}_{S}].
\end{align}
\end{cor*}

By the Verdier self-duality of $\varphi^p_{\bs{\mfrakM}_{X}}$, the isomorphism \eqref{eq:intdimred2}
implies
\begin{align}\label{eq:intdimred3}
\mH^*(\mfrakM_{X}; \varphi^p_{\bs{\mfrakM}_{X}}) \cong \mH^{\mathrm{BM}}_{\vdim \bs{\mfrakM}_{S} -*}(\mfrakM_{S})
\end{align}
where $\mH^{\mathrm{BM}}$ denotes the Borel--Moore homology.
Since it is shown in \cite{KV} that $\mH^{\mathrm{BM}}_{*}(\bs{\mfrakM}_{S})$ carries a convolution product,
the isomorphism \eqref{eq:intdimred3} induces an algebra structure on $\mH^*(\bs{\mfrakM}_{X}; \varphi^p_{\bs{\mfrakM}_{X}})$.
We expect that this is isomorphic to the conjectural critical CoHA for $X$
and it is useful in the study of wall-crossing formulas for CoDT invariants of $X$ as in the local case.
Further, as the local dimensional reduction isomorphism \eqref{eq:locdimred} plays an important role in the cohomological study of moduli stacks of representations of preprojective algebras in \cite{Dav2},
we expect that its global variant \eqref{eq:intdimred2} is useful in the cohomological study of moduli stacks of coherent sheaves on K3 surfaces or Higgs sheaves on curves.
These directions will be pursued in future work.

\subsection{Thom isomorphism}

For a quasi-smooth derived scheme $\bs{Y}$,
the dimensional reduction isomorphism \eqref{eq:intdimred} has another interpretation: a version of the Thom isomorphism for the dual obstruction cone.
By imitating the construction of the Euler class,
we construct a class $e(\bfT^*[-1] \bs{Y}) \in \mH^{\mathrm{BM}}_{2\vdim \bs{Y}}(Y)$ where we write $Y = t_0(\bs{Y})$ as follows.
Consider the natural morphism
\begin{align}\label{eq:preEuler}
\pi_! \varphi_{\bfT^*[-1]\bs{Y}} \to \pi_* \varphi_{\bfT^*[-1]\bs{Y}}.
\end{align}
By taking the Verdier dual of the isomorphism \eqref{eq:intdimred}, we have
\[
\pi_* \varphi_{\bfT^*[-1]\bs{Y}} \cong \omega_{Y}[-\vdim \bs{Y}].
\]
Therefore the map \eqref{eq:preEuler} defines an element in $\mH^{\mathrm{BM}}_{2\vdim \bs{Y}}(Y)$, which we name $e(\bfT^*[-1] \bs{Y})$.
Since the virtural fundamental class is a generalization of the Euler class,
it is natural to compare $e(\bfT^*[-1] \bs{Y})$ with the virtual fundamental class $[\bs{Y}]^{\vir}$.
Concerning this we have obtained the following claim, which will be proved in a subsequent paper:
\begin{thm*}[\cite{Kin}]
Assume $Y$ is quasi-projective. Then we have
\[
  e(\bfT^*[-1] \bs{Y}) = (-1)^{\vdim \bs{Y} \cdot (\vdim \bs{Y} - 1) /2}[\bs{Y}]^{\vir}.
\]
\end{thm*}
In other words, this theorem gives a new construction of the virtual fundamental class (at least for quasi-projective cases).
It is an interesting problem to construct other enumerative invariants
(e.g. Donaldson--Thomas type invariants for Calabi--Yau $4$-folds constructed in \cite{CL, BJ, OT}) based on the isomorphism $\eqref{eq:intdimred}$ or its variant.
This direction will be investigated in future work.

\subsection{Plan of the paper}
This paper is organized as follows.
In Section \ref{section2} we recall some basic facts used in CoDT theory.
In Section \ref{section3} we prove the dimensional reduction theorem for quasi-smooth derived schemes,
by gluing the local dimensional reduction isomorphisms in \cite[Theorem A.1]{Dav1}.
In Section \ref{section4} we extend the dimensional reduction theorem to quasi-smooth derived Artin stacks.
The key point of the proof is the observation that the canonical $-1$-shifted symplectic structure and the canonical orientation for $-1$-shifted cotangent stacks are preserved by smooth base changes.
In Section \ref{section5} we discuss some applications of the dimensional reduction theorem.
In Appendix \ref{ap:A} we collect some basic facts on the determinant of perfect complexes.

\textbf{Acknowledgments}.

I am deeply grateful to my supervisor Yukinobu Toda
  for introducing cohomological Donaldson--Thomas theory to me, for sharing many ideas, and for useful comments on the manuscript of this paper.
I also thank Ben Davison for answering several questions related to this work.
I was partly supported by WINGS-FMSP program at the Graduate School of Mathematical Science,
the University of Tokyo.

\textbf{Notation and convention}.

\begin{itemize}
\item
All commutative dg algebras (cdga for short) sit in non-positive degree
with respect to the cohomological grading.

\item
All derived or underived Artin stacks are assumed to be $1$-Artin,
and to have quasi-compact and quasi-separated diagonals.

\item
All cdgas and derived Artin stacks are assumed to be locally of finite presentation
over the complex number field $\mathbb{C}$.

\item
Denote by $\mathbb{S}$, $\dsch$, and $\dst$ the $\infty$-categories of spaces, derived schemes, and
derived stacks respectively.

\item
For a derived scheme or a derived Artin stack $\bs{X}$,
$t_0(\bs{X})$ denotes the classical truncation.
Denote by $\bs{X}^\red = t_0(\bs{X})^\red$ the reduction of $\bs{X}$.

\item
For a morphism of derived Artin stacks $\bs{f} \colon \bs{X} \to \bs{Y}$,
$\mathbb{L}_{\bs{f}}$ denotes the relative cotangent complex.
We write $\mathbb{L}_{\bs{X}}$ for the absolute cotangent complex of $\bs{X}$,
and $\mathbb{T}_{\bs{X}} \coloneqq \mathbb{L}_{\bs{X}}^\vee$ for the tangent complex.

\item
For a derived Artin stack $\bs{\mathfrak{X}}$,
$\vdim \bs{\mathfrak{X}}$ denotes the locally constant function on $t_0(\bs{\mathfrak{X}})$
whose value at $p \in t_0(\bs{\mathfrak{X}})$ is $\sum (-1)^i \mathrm{H}^i(\mathbb{L}_{\bs{\mathfrak{X}}} |_p)$.
We define $\vdim \bs{f}$ for a morphism locally of finite presentation between derived Artin stacks $\bs{f}$ in a similar manner.

\item
For a derived Artin stack $\bs{\mfrakX}$ and a perfect complex $E$ on $\bs{\mfrakX}$,
we define $\detr(E) \coloneqq \det(E |_{\bs{\mfrakX}^\red})$.

\item
For a complex analytic space or a scheme $X$,
we will only consider (analytically) constructible sheaves or perverse sheaves which are of $\mathbb{Q}$-coefficients.
Denote by $D^b _c(X, \mathbb{Q})$ the full subcategory of the bounded derived category of sheaves of $\mathbb{Q}$-vector spaces $D^b(X, \mathbb{Q})$ spanned by the complexes with (analytically) constructible cohomology sheaves.

\item
Concerning sign conventions for derived categories, we always follow \cite[Tag 0FNG]{Stacks}.

\item
If there is no confusion, we use expressions such as $f_*$, $f_!$, and $\sHom$ for the derived functors
$Rf_*$, $Rf_!$, and $\sRHom$.

\end{itemize}

\section{Shifted symplectic structures and vanishing cycles}\label{section2}

In this section, we briefly recall the notion of shifted symplectic structures
introduced by \cite{PTVV}, and some facts about vanishing cycle functors which will be
needed later.

\subsection{Shifted symplectic geometry}

Here we briefly recall some notions in derived algebraic geometry and shifted symplectic geometry.

\begin{dfn}
    A derived Artin stack $\bs{\mathfrak{X}}$ is called \textit{quasi-smooth} if
   the cotangent complex $\mathbb{L}_{\bs{\mathfrak{X}}}$ is perfect of amplitude $[-1, 1]$.
\end{dfn}

Let $U$ be a smooth scheme, and $s \in \Gamma (U, E)$ be a section of a vector bundle $E$ on $U$.
Write $\bs{Z}(s)$ for the derived zero locus of $s$.
We have the following Cartesian diagram in $\dsch$
\[
  \xymatrix{
     \pbcorner \bs{Z}(s) \ar[r]^f\ar[d]_g & U \ar[d]^s\\
     U \ar[r]^0 & E.
    }
\]
Since $\mathbb{L}_g[-1]$ and $g^* \mathbb{L}_U$ are locally free sheaves concentrated in degree zero,
we conclude that $\bs{Z}(s)$ is quasi-smooth.

\begin{dfn}
    For a quasi-smooth derived scheme $\bs{X}$,
    a \textit{Kuranishi chart} is a tuple $(Z, U, E, s, \bs{\iota})$ where $Z$ is an open subscheme of $t_0(\bs{X})$, $U$ is a smooth scheme,
    $E$ is a vector bundle on $U$, $s$ is a section of $E$, and $\map{\bs{\iota}}{\bs{Z}(s)}{\bs{X}}$
    is an open immersion whose image is $Z$.
    A Kuranishi chart is said to be \textit{minimal} at $p = \bs{\iota}(q) \in Z$ if the differential
    $\map{(ds)_q}{T_q U}{E_q}$ is zero.
    A Kuranishi chart $(Z, U, E, s, \bs{\iota})$ is called \textit{good} if $U$ is affine and has global \'{e}tale coordinate
    i.e. regular functions $x_1, x_2, \ldots, x_n$ such that $\ddr x_1, \ddr x_2, \ldots, \ddr x_n$ form a basis of $\Omega_U$,
    and $E$ is a trivial vector bundle of a constant rank.
\end{dfn}

\begin{prop}\label{prop:qslocal}
    Let $\bs{X}$ be a quasi-smooth derived scheme, and $p \in \bs{X}$ be a point.
    \begin{itemize}
      \item[(i)] \cite[Theorem~4.1]{BBJ}
      There exists a Kuranishi chart $(Z, U, E, s, \bs{\iota})$ of $\bs{X}$ minimal at $p \in Z$.
      \item[(ii)] \cite[Theorem~4.2]{BBJ}
      For $i=1,2$, let $(Z_i, U_i, E_i, s_i, \bs{\iota}_i)$ be a Kuranishi chart on $\bs{X}$
      minimal at $p = \bs{\iota}_i(q_i)$.
      Then there exists a third Kuranishi chart $(Z', U', E', s', \bs{\iota}')$ of $\bs{X}$
      minimal at $p = \bs{\iota}'(q')$, \'{e}tale morphisms $\map{\eta_i}{U'}{U_i}$, and isomorphisms $\map{\tau_i}{E'}{\eta_i ^* E_i}$ with
      the following properties:
      \begin{itemize}
      \item $\tau_i (s') = \eta_i ^* s_i$.
      \item The composition $\bs{Z}(s') \to \bs{Z}(s_i) \xrightarrow{\bs{\iota}_i} \bs{X}$ is equivalent to $\bs{\iota}'$
       where the first map is induced by $\eta_i$ and $\tau_i$.
    \end{itemize}
  \end{itemize}
\end{prop}

\begin{proof}
  (i) is a direct consequence of \cite[Theorem~4.1]{BBJ}.

  (ii) follows from \cite[Theorem~4.2]{BBJ} except for $\eta_i$ being \'{e}tale and $\tau_i$ being isomorphism.
  Since $(Z_i, U_i, E_i, s_i, \bs{f}_i)$ is minimal at $q_i$, we have
  \[
  \mathrm{H}^0(\mathbb{L}_{\bs{Z}(s_i)} |_{q_i}) \cong \Omega_{U_i} |_{q_i}, \,
  \mathrm{H}^{-1}(\mathbb{L}_{\bs{Z}(s_i)} |_{q_i}) \cong E_i^\vee |_{q_i}.
  \]
  Similarly, we have
  \[\mathrm{H}^0(\mathbb{L}_{\bs{Z}(s')} |_{q'}) \cong \Omega_{U'} |_{q'}, \,
  \mathrm{H}^{-1}(\mathbb{L}_{\bs{Z}(s')} |_{q'}) \cong E'^\vee |_{q'}.
  \]
  Since the open immersion $\bs{Z}(s') \to \bs{Z}(s_i)$ induces a quasi-isomorphism
  $\mathbb{L}_{\bs{Z}(s_i)} |_{q_i} \simeq \mathbb{L}_{\bs{Z}(s')} |_{q'}$,
  we see that $\eta_i$ is \'etale at $q'$ and $\tau_i$ is isomorphic at $q'$.
  Thus by shrinking $U'$ around $q'$ if necessary, we obtain the required properties.
\end{proof}

Let $A$ be a cdga, and take a semi-free resolution $A' \to A$.
We define the space of $n$-shifted $p$-forms and the space of $n$-shifted closed $p$-forms by
\[
\mathcal{A}^p(A, n) \coloneqq \lvert (\wedge ^p \Omega_{A'} [n], d) \rvert,
\]
\[
 \mathcal{A}^{p,cl}(A, n) \coloneqq \lvert (\prod_{i \geq 0} \wedge^{p + i} \Omega_{A'} [-i + n], d + \ddr) \rvert
\]
respectively where $d$ is the internal differential induced by the differential of $\Omega_{A'}$, and
$\ddr$ is the de Rham differential.
Here for a dg module $E$, $\lvert E \rvert$ denotes the simplicial set corresponding to the truncation
$\tau^{\leq 0}(E)$ by the Dold--Kan correspondence.
These constructions can be made $\infty$-functorial, and they satisfy the sheaf condition with respect to the \'{e}tale topology
(see \cite[Proposition~1.11]{PTVV}). Therefore we obtain $\infty$-functors
\[
\map{\mathcal{A}^p(-, n), \mathcal{A}^{p,cl}(-, n)}{\dst^{\op}}{\mathbb{S}}.
\]
We write $\bs{f}^\star$ for $\mathcal{A}^p(\bs{f}, n)$ and also for $\mathcal{A}^{p,cl}(\bs{f}, n)$.
For a derived Artin stack $\bs{\mfrakX}$,
it is shown in \cite[Proposition~1.14]{PTVV} that we have an equivalence
\begin{align}\label{eq:spform}
\mathcal{A}^p(\bs{\mfrakX}, n) \simeq \mathrm{Map}(\mathcal{O}_{\bs{\mfrakX}}, \wedge^p \mathbb{L}_{\bs{\mfrakX}[n]}).
\end{align}
We have natural morphisms
\begin{align*}
\pi \colon& \mathcal{A}^{p,cl}(-, n) \to \mathcal{A}^{p}(-, n) \\
\ddrc \colon& \mathcal{A}^p(-, n)  \to \mathcal{A}^{p+1,cl}(-, n)
\end{align*}
where $\pi$ is induced by the projection $\prod_{i \geq 0} \wedge^{p + i} \Omega_{A'} [-i + n] \to \wedge ^p \Omega_{A'} [p]$,
and $\ddrc$ is induced by the map
\[
\mathcal{A}^p(A, n) \ni \omega^0 \mapsto (\ddr \omega^0,0,0,\ldots) \in \mathcal{A}^{p+1,cl}(A, n).
\]

\begin{dfn}
    Let $\bs{\mfrakX}$ be a derived Artin stack. A closed $n$-shifted $2$-form $\omega \in \mathcal{A}^{2,cl}(\bs{\mfrakX}, n)$ is called an \textit{$n$-shifted symplectic structure}
    if $\pi(\omega)$ is non-degenerate (i.e. the morphism $\mathbb{T}_{\bs{\mfrakX}} \to \mathbb{L}_{\bs{\mfrakX}}[n]$ induced by $\pi(\omega)$ using the identification \eqref{eq:spform} is an equivalence).
    An \textit{$n$-shifted symplectic derived Artin stack} is a derived Artin stack equipped with an $n$-shifted symplectic structure on it.
\end{dfn}

We say that two shifted symplectic derived Artin stacks $(\bs{\mfrakX}_1, \omega_{\bs{X}_1})$ and $(\bs{\mfrakX}_2, \omega_{\bs{\mfrakX}_2})$ are equivalent
if there exists an equivalence $\bs{\Phi} \colon \bs{\mfrakX}_1 \xrightarrow{\sim} \bs{\mfrakX}_2$ as derived stacks and an equivalence  $\bs{\Phi} ^\star \omega_{\bs{\mfrakX}_2} \sim \omega_{\bs{\mfrakX}_1}$.

\begin{ex}\label{ex:crit}
    Let $U = \Spec A$ be a smooth affine scheme and $\map{f}{U}{\mathbb{C}}$ be a regular function on it.
    Denote by $\bs{X}$ the derived critical locus $\dcrit(f)$.
    Since $\bs{X}$ is the derived zero locus of the section $\map{\ddr f}{U}{\Omega_U}$,
    the derived scheme $\bs{X}$ is equivalent to $\dSpec B$ where $B$ is a cdga defined by the Koszul complex
    \[
    B \coloneqq ( \cdots \to \wedge^2 \Omega_A ^{\vee} \xrightarrow{\cdot \ddr f} \Omega_A ^{\vee} \xrightarrow{\cdot \ddr f} A).
    \]
    Assume there exists a global \'{e}tale coordinate $(x_1, \ldots, x_n)$ on $U$.
    We write the dual basis of $\ddr x_1, \ldots, \ddr x_n$ by
    $\frac{\partial}{\partial x_1}, \ldots, \frac{\partial}{\partial x_n}$,
    and $y_i \in B^{-1}$ denotes the element of degree $-1$ corresponding to $\frac{\partial}{\partial x_i}$ for each $i=1,\ldots, n$.
    Although $B$ is not semi-free in general, one can see that $\Omega_B$ gives a model for $\L_B$.
    Define an element $\omega_{\bs{X}}' \in (\wedge ^2 \Omega_B)^{-1}$ of degree $-1$ by
    \[
    \omega_{\bs{X}}' \coloneqq \ddr x_1 \wedge \ddr y_1 + \cdots + \ddr x_n \wedge \ddr y_n.
    \]
    This defines a $-1$-shifted $2$-form, which is clearly non-degenerate.
    Since we also have $\ddr \omega_{\bs{X}}' =0$, the closed form $\omega_{\bs{X}} \coloneqq (\omega_{\bs{X}}', 0, \ldots)$ defines
    a $-1$-shifted symplectic structure on $\bs{X}$.
\end{ex}

\begin{ex}\label{ex:cot}
    Let $\bs{\mfrakY}$ be a derived Artin stack, and $n$ be an integer.
    Define the $n$-shifted cotangent stack of $\bs{\mfrakY}$ by $\bfT^*[n]\bs{\mfrakY} \coloneqq \dSpec_{\bs{\mfrakY}}(\Sym(\mathbb{T}_{\bs{\mfrakY}}[-n]))$.
    Let $\map{\bs{\pi}}{\bfT^*[n]\bs{\mfrakY}}{\bs{\mfrakY}}$ be the projection.
    We have a tautological $n$-shifted $1$-form $\lambda _{\bfT^*[n]\bs{\mfrakY}}$ on $\bfT^* [n]\bs{\mfrakY}$ defined by
    the image of the tautological section of $\bs{\pi}^* \mathbb{L}_{\bs{\mfrakY}}[n]$ under the canonical map
    $\bs{\pi}^* \mathbb{L}_{\bs{\mfrakY}}[n] \to \mathbb{L}_{\bfT^*[n]\bs{\mfrakY}}[n]$.
    In \cite[Theorem~2.2]{Cal} it is shown that $\omega_{\bfT^*[n]\bs{\mfrakY}} \coloneqq \ddrc{\lambda _{\bfT^*[n]\bs{\mfrakY}}}$ is non-degenerate,
     and we obtain the \emph{canonical $n$-shifted symplectic structure} on $\bfT^*[n]\bs{\mfrakY}$.
\end{ex}

It is proved in \cite[Theorem~5.18]{BBJ} that any $-1$-shifted symplectic derived scheme is
Zariski locally modeled on a derived critical locus.
For a $-1$-shifted symplectic derived scheme of the form $\bfT^*[-1]\bs{Y}$ for some quasi-smooth derived scheme $\bs{Y}$,
its local model as derived critical locus can be described by combining Proposition \ref{prop:qslocal} and the following lemma.

\begin{lem}\label{lem:cotcrit}
    Let $U = \Spec A$ be a smooth affine scheme admitting a global \'{e}tale coordinate,
    $E$ be a trivial vector bundle on $U$, and $s \in \Gamma(U, E)$ be a section.
    Denote by $\bar{s}$ the regular function on $\Tot_U(E ^{\vee})$ corresponding to $s$.
    Then we have an equivalence of $-1$-shifted symplectic derived schemes
    \begin{align}\label{eq:cotcrit}
     (\dcrit(\bar{s}), \omega_{\dcrit(\bar{s})}) \simeq (\bfT^*[-1]{\bs{Z}(s)}, \omega_{\bfT^*[-1]{\bs{Z}(s)}})
   \end{align}
    equipped with the $-1$-shifted symplectic structures constructed in Example \ref{ex:crit} and Example \ref{ex:cot} respectively.
\end{lem}

\begin{proof}
 Fix a global \'{e}tale coordinate $x_1, \ldots, x_n$ on $U$ and a basis $e_1, \ldots, e_m$ of $M \coloneqq \Gamma (U, E)$.
 Write $s = a_1 e_1 + \cdots + a_m e_m$.
 If we write $\alpha_i \coloneqq \ddr x_i$, then $\alpha_1, \ldots, \alpha_n$ defines a basis of $\Omega_U$.
Denote by $e_1^{\vee}, \ldots, e_m^{\vee}$ and $\alpha_1^{\vee}, \ldots, \alpha_n^{\vee}$ the dual bases of $e_1, \ldots, e_m$
 and $\alpha_1, \ldots, \alpha_n$ respectively.
 Define a cdga $C$ by the Koszul complex
 \[
  C \coloneqq (\cdots \to (\Omega_A ^{\vee} \oplus M^{\vee}) \otimes _A A[z_1, \ldots, z_m]
  \xrightarrow{\alpha_i^{\vee} \otimes 1 \mapsto \sum (\partial a_j / \partial x_i) z_j, e_j^{\vee} \otimes 1 \mapsto a_j} A[z_1, \ldots, z_m]).
 \]
 Now it is clear that $\dSpec C$ gives models for both $\dcrit(\bar{s})$ and $\bfT^*[-1]{\bs{Z}(s)}$.
 The tautologial $1$-form $\lambda _{\bfT^*[-1]{\bs{Z}(s)}}$ is represented by
 \[
 \sum _{i=1} ^n \alpha_i^\vee (\ddr x_i ) + \sum _{j=1} ^m z_j(\ddr e_j ^{\vee}) \in \mathcal{A}^{1}(\dSpec C, -1).
 \]
 A direct computation shows that $\ddrc \lambda _{\bfT^*[-1]{\bs{Z}(s)}} \sim \omega_{\dcrit(\bar{s})}$,
which implies the lemma.
\end{proof}

\subsection{D-critical schemes}\label{subsection:dcriticalsch}

In this section, we briefly recall the notion of d-critical structures introduced in \cite{Joy} which is a classical model for $-1$-shifted symplectic structures.
D-critical structures are easier to treat than $-1$-shifted symplectic structures and are enough to apply in cohomological Donaldson--Thomas theory.

For any complex analytic space $X$, Joyce in \cite[Theorem 2.1]{Joy} introduced a sheaf
\begin{align}\label{eq:shS}
  \mathcal{S}_X \in \mathrm{Mod}(\setC_{X})
\end{align}
 of $\mathbb{C}$-vector space on $X$ with the following property:
for any open subset $R \subset X$ and any closed embedding $i \colon R \hookrightarrow U$ where $U$ is a complex manifold,
we have an exact sequence of sheaves on $R$
\[
  0 \to \mathcal{S}_X \vert _R \to i^{-1} \mathcal{O}_U/I_{R, U} ^2 \xrightarrow{\ddr}
  i^{-1} \Omega_U / (I_{R, U} \cdot i^{-1} \Omega_U).
\]
Here $I_{R, U}$ is the ideal sheaf of $i^{-1} \mathcal{O}_U$ corresponding to $R$.
The following composition
\[
  \mathcal{S}_X \vert _R \to i^{-1} \mathcal{O_U}/I_{R, U} ^2 \to
  i^{-1}\mathcal{O_U}/I_{R, U} \cong \mathcal{O}_R
\]
glue to define a morphism $\map{\beta_X}{\mathcal{S}_X}{\mathcal{O}_X}$, and we define a subsheaf $\mathcal{S}_X ^0$ of $\mathcal{S}_X$
by the kernel of the composition
\[
\mathcal{S}_X \xrightarrow{\beta_X} \mathcal{O}_X \to \mathcal{O}_{X^{\red}}.
\]
It can be shown that we have a decomposition $\mathcal{S}_X = \mathbb{C}_X \oplus \mathcal{S}_X ^0$
where $\mathbb{C}_X$ is the constant sheaf and $\mathbb{C}_X \hookrightarrow \mathcal{S}_X$ is induced by
the inclusion $\mathbb{C}_U \hookrightarrow \mathcal{O}_U$ identifying $\mathbb{C}_U$ with the sheaf of locally constant functions on $U$.
If $X$ is the critical locus of some function $f$ on a complex manifold $U$ such that $f \vert_{X ^{\red}} = 0$,
then $f + I_{X, U}^2$ defines an element of $\Gamma(X, \mathcal{S}_X^0)$ since $\ddr f \vert _X = 0$.

\begin{dfn}\cite[Definition 2.5]{Joy}
  Let $X$ be a complex analytic space. A section $s \in \Gamma(X, \mathcal{S}_X ^0)$ is called an (analytic) \textit{d-critical structure} if for any closed point $x \in X$
  there exist an open neighborhood $x \in R \subset X$, a complex manifold $U$, a regular function $f$ on $U$ with $f \vert_{R^{\red}} = 0$,
  and a closed embedding $\map{i}{R}{U}$ such that $i(R) = \Crit(f)$ and $f + I_{R, U}^2 = s \vert_R$.
  The tuple $(R, U, f, i)$ as above is called a \textit{d-critical chart} for $(X, s)$.
  A \textit{d-critical scheme} is a scheme equipped with a d-critical structure on its analytification.
\end{dfn}

\begin{rem}
  Joyce \cite[Definition 2.5]{Joy} also introduced the algebraic version of the d-critical structure,
  and some authors define a d-critical scheme as a scheme equipped with an algebraic d-critical structure.
  We always work with analytic d-critical structures since they are enough for our purposes.
\end{rem}

For a d-critical chart $(R, U, f, i)$ of a d-critical scheme $(X, s)$ and an open subset $U' \subset U$,
define $R' = i^{-1}(U')$, $f' = f|_{U'}$, and $i' = i |_{R'} \colon R' \hookrightarrow U'$.
Then $(R', U', f', i')$ defines a d-critical chart on $(X, s)$.
We call $(R', U', f', i')$ an open subchart of $(R, U, f, i)$.

In order to compare two d-critical charts, Joyce introduced the notion of \textit{embedding}:
for a d-critical scheme $(X, s)$ and its d-critical charts $(R_1, U_1, f_1, i_1)$ and $(R_2, U_2, f_2, i_2)$
such that $R_1 \subset R_2$, an embedding ${(R_1, U_1, f_1, i_1)}\hookrightarrow{(R_2, U_2, f_2, i_2)}$ is defined by
a locally closed embedding $\Phi \colon {U_1}\hookrightarrow{U_2}$ such that $f_1 = f_2 \circ {\Phi}$ and
$\Phi \circ i_1 = i_2 \vert _{R_2}$.
The following theorem is useful when comparing two d-critical charts:

\begin{thm}\label{thm:dcrit}
  Let $(X, s)$ be a d-critical scheme.
  \begin{itemize}
    \item[(i)]\cite[Theorem 2.20]{Joy}
    Let $(R_1, U_1, f_1, i_1)$ and $(R_2, U_2, f_2, i_2)$ be d-critical charts, and $x \in R_1 \cap R_2$ be
    a point. Then by shrinking these d-critical charts around $x$ if necessary,
    we can find a third d-critical chart $(R_3, U_3, f_3, i_3)$ with $x \in R_3$ and embeddings
    ${(R_1, U_1, f_1, i_1)}\hookrightarrow{(R_3, U_3, f_3, i_3)}$ and ${(R_2, U_2, f_2, i_2)}\hookrightarrow{(R_3, U_3, f_3, i_3)}$.

    \item[(ii)]\label{thm:dcritemb}\cite[Theorem 2.22]{Joy}
    Let $\Phi \colon (R_1, U_1, f_1, i_1) \hookrightarrow (R_2, U_2, f_2, i_2)$ be an embedding of d-critical charts,
    and $x \in R_1$ be a point. Then by shrinking these d-critical charts around $x$ keeping $\Phi(U_1) \subset U_2$ if necessary and replacing $\Phi$ by its restriction,
    we can find holomorphic maps $\alpha \colon U_2 \to U_1$ and $\beta \colon U_2 \to \mathbb{C}^n$
    for $n = \dim U_2 - \dim U_1$,
    such that $(\alpha, \beta) \colon U_2 \to U_1 \times \mathbb{C}^n$ is biholomorphic onto its image,
    and we have $\alpha \circ \Phi =\id$, $\beta \circ \Phi = 0$, and
    $f_2 = f_1 \circ \alpha + (z_1^2+ \cdots +z_n^2)\circ \beta$ where $z_i$ is the $i$-th coordinate
    of $\mathbb{C}^n$.
  \end{itemize}
\end{thm}

For an embedding of d-critical charts $\Phi \colon (R_1, U_1, f_1, i_1) \hookrightarrow (R_2, U_2, f_2, i_2)$ of a d-critical scheme $(X, s)$,
Joyce defined in \cite[Definition 2.26]{Joy} a natural isomorphism
\[
J_{\Phi} \colon i_1^* (K_{U_1}^{\otimes^2})|_{R_1 ^\red} \cong i_2^* (K_{U_2}^{\otimes^2}) |_{R_1 ^\red}.
\]
If there exist $\alpha, \beta$ as in Theorem \ref{thm:dcritemb} (ii), $J_{\Phi}$ is defined as follows.
Firstly, we have isomorphisms
\[
  K_{U_2} \cong (\alpha, \beta)^*(K_{U_1 \times \mathbb{C}^n}) \cong \alpha^*K_{U_1} \otimes \beta^*K_{\mathbb{C}^n} \cong \alpha^*K_{U_1}
\]
where the final isomorphism is defined by the trivialization
\[
K_{\mathbb{C}^n} \cong \mathcal{O}_{\mathbb{C}^n} \cdot (\D z_1 \wedge \cdots \wedge \D z_n).
\]
Then by taking the square of this composition and pulling back to $R_1$, we obtain the desired isomorphism.

Using this preparation, we can construct a natural line bundle $K_{X, s}$ on $X^{\red}$, which is a d-critical version of the canonical line bundle as follows:

\begin{thm}\cite[Theorem 2.28]{Joy}
For a d-critical scheme $(X, s)$, one can define a line bundle $K_{X, s}$ on $X^{\red}$ which we call the virtual canonical bundle of $(X, s)$ characterized by the following properties:
\begin{itemize}
  \item[(i)] For a d-critical chart $\mathscr{R} = (R, U, f, i)$ of $(X, s)$, we have an isomorphism
  \[
  \iota_{\mathscr{R}} \colon K_{X, s} \vert_{R^{\red}}  \cong (i^* K_U) ^{\otimes^2} \vert_{R^{\red}}.
  \]
  \item[(ii)]
  For an embedding of d-critical charts
  \[
  \Phi \colon \mathscr{R}_1 = (R_1, U_1, f_1, i_1) \hookrightarrow \mathscr{R}_2 = (R_2, U_2, f_2, i_2),
  \]
  we have
  \[
  \iota_{\mathscr{R}_2} \vert_{R^{\red}_1} = J_{\Phi} \circ \iota_{\mathscr{R}_1}.
  \]
\end{itemize}
\end{thm}

\begin{dfn}\cite[Definition 2.31]{Joy}
An \textit{orientation} $o$ of a d-critical scheme $(X, s)$
is a choice of a line bundle $L$ on $X^{\red}$ and an isomorphism
  \[
    o \colon L^{\otimes^2} \xrightarrow{\cong} K_{X, s}.
  \]
\end{dfn}

As we have seen in the previous section, $-1$-shifted symplectic derived schemes are locally modeled on derived critical loci.
Therefore we can regard the notion of d-critical schemes as an underived version of the $-1$-shifted symplectic derived scheme.
The following theorem gives the rigorous statement of this fact:

\begin{thm}\cite[Theorem 6.6]{BBJ}\label{thm:shiftdcrit}
  Let $(\bs{X}, \omega_{\bs{X}})$ be a $-1$-shifted symplectic derived scheme.
  Then its underlying scheme $X = t_0 (\bs{X})$ carries a canonical d-critical structure $s_X$ with the following property:
  for any $-1$ shifted symplectic derived scheme $(\bs{R}, \omega_{\bs{R}})$ of the form $\bs{R} = \dcrit(f)$ where $f$ is a regular function on a
  smooth scheme $U$, $\omega_{\bs{R}}$ the $-1$-shifted symplectic form on $\bs{R}$ constructed in Example \ref{ex:crit},
  and an open inclusion $\bs{\iota} \colon \bs{R} \hookrightarrow \bs{X}$ such that $\bs{\iota} ^* \omega_{\bs{X}} \sim \omega_{\bs{R}}$,
  the tuple $(R, U, f, i)$ gives a d-critical chart for $(X, s_X)$ where we write $R = t_0(\bs{R})$ and $i \colon R \hookrightarrow U$ the natural closed embedding.
  Furthermore there exists a canonical isomorphism of line bundles
  \[ \Lambda_{\bs{X}} \colon \detr(\L_{\bs{X}}) \cong K_{X, s_X}.
  \]
  where $\detr(\L_{\bs{X}}) = \det(\L_{\bs{X}} |_{X^\red})$ by definition.
\end{thm}

We define the notion of orientations for $-1$-shifted symplectic derived schemes to be that of the underlying d-critical schemes.

For later use, we explain the construction of $\Lambda_{\bs{X}}$ in the above theorem
for $\bs{X} = \dcrit(f)$ where $f$ is a regular function on a smooth scheme $U$.
In this case, $\L_{\bs{X}}|_{X}$ is represented by the following two-term complex
\[
(T_U |_X \xrightarrow{\mathrm{Hess}(f)} \Omega_U |_X)
\]
where $\mathrm{Hess}(f)$ denotes the Hessian of $f$.
We define $\Lambda_{\bs{X}}$ by the following composition:
\begin{align}\label{eq:dcritcan}
  \begin{aligned}
\detr(\mathbb{L}_{\bs{X}}) \cong
\det(\Omega_U \vert _{X^{\red}}) \otimes \det(T_U \vert _{X^{\red}})^{-1}
 \\
\cong (i^* K_U) \vert _{X^{\red}} ^{\otimes^2}
 \xrightarrow{\cdot (\frac{1}{2})^{\dim U}} (i^* K_U) \vert _{X^{\red}} ^{\otimes^2} \cong  K_{X, s}.
\end{aligned}
\end{align}
where $\det(T_U)^{-1} \cong K_U$ is locally defined by
\begin{align}\label{eq:dualconv}
   (\partial/ \partial z_1 \wedge \cdots \wedge \partial / \partial z_n)^\vee \mapsto \D z_1 \wedge \cdots \wedge \D z_n.
\end{align}
The constant $(\frac{1}{2})^{\dim U}$ is just a convention, and it corresponds to
the fact that $\mathrm{Hess} (z^2) = 2(\D z)^{\otimes^2}$.

Now we discuss the canonical orientation for $-1$-shifted cotangent schemes.
To do this we recall basic facts on the determinant of perfect complexes, which are proved in Appendix \ref{ap:A}:
\begin{lem}\label{lem:det}
  Let $\bs{\mfrakX}$ be a derived Artin stack.
  \begin{itemize}
    \item[(i)] For a perfect complex $E$ on $\bs{\mfrakX}$ and an integer $m$,
    we have natural isomorphisms
    \begin{gather*}
      \hat{\eta}_E \colon \detr(E^\vee) \cong \detr(E)^{-1}, \, \hat{\chi}_E^{(m)} \colon \detr{(E[m])} \cong \detr(E)^{(-1)^m}.
    \end{gather*}
    \item[(ii)] For a distinguished triangle $\Delta \colon E \to F \to G \to E[1]$ of perfect complexes on $\bs{\mfrakX}$,
    we have a natural isomorphism
    \[
    \hi(\Delta) \colon  \detr(E) \otimes \detr(G) \cong \detr(F).
    \]
  \end{itemize}
\end{lem}
We write $\hat{\chi}_E = \hat{\chi}_E^{(1)}$.
The following example defines the canonical orientation for $-1$-shifted cotangent schemes:

\begin{ex}\cite[Lemma 4.3]{Tod}\label{ex:canori}
  Let $\bs{Y}$ be a quasi-smooth derived scheme, $\map{\bs{\pi}}{\bfT^*[-1]\bs{Y}}{\bs{Y}}$ be the projection and
  write $s_{\bfT^*[-1]\bs{Y}}$ the d-critical structure associated with the canonical $-1$-shifted symplectic form
  constructed in Example \ref{ex:cot}.
  We have a natural distinguished triangle
  \begin{align}\label{eq:cantri}
    \Delta_{\bs{\pi} } \colon {\bs{\pi} ^*} \mathbb{L}_{\bs{Y}} \to \mathbb{L}_{\bfT^*[-1]\bs{Y}} \to
     {\bs{\pi} ^*} \mathbb{T}_{\bs{Y}}[1] \to \mathbb{L}_{\bs{Y}}[1] .
  \end{align}
  Define an isomorphism
  \begin{align*}
    o_{\bfT^*[-1]\bs{Y}}' \colon (\pi^\red )^* \detr(\mathbb{L}_{\bs{Y}})^{\otimes^2}
    \cong \detr(\mathbb{L}_{\bfT^*[-1]\bs{Y}})
  \end{align*}
  by the composition of the isomorphisms
  \begin{align*}
     (\pi^\red )^* \detr(\mathbb{L}_{\bs{Y}})^{\otimes^2}
    &\xrightarrow{\id \otimes \hat{\eta}_{\bs{\pi}^* \mathbb{T}_{\bs{Y}}}} (\pi^\red )^* \detr(\mathbb{L}_{\bs{Y}}) \otimes
    (\pi^\red )^* \detr(\mathbb{T}_{\bs{Y}})^{-1} \\
    &\xrightarrow{\id \otimes \hat{\chi}_{\bs{\pi}^* \mathbb{T}_{\bs{Y}}}^{-1}} (\pi^\red )^* \detr(\mathbb{L}_{\bs{Y}}) \otimes
    (\pi^\red )^* \detr(\mathbb{T}_{\bs{Y}}[1]) \\
    &\xrightarrow{\hi(\Delta_{\bs{\pi}})} \detr(\mathbb{L}_{\bfT^*[-1]\bs{Y}})
  \end{align*}
  where we write $\pi  = t_0(\bs{\pi} )$.
  Define
  \begin{align}\label{eq:canori}
  o^{}_{\bfT^*[-1]\bs{Y}}
  \colon (\pi^\red )^* \detr(\mathbb{L}_{\bs{Y}})^{\otimes^2} \cong K_{\bfT^*[-1]\bs{Y}, s_{\bfT^*[-1]\bs{Y}}}
\end{align}
  by the composition $\Lambda_{\bfT^*[-1]\bs{Y}} \circ o_{\bfT^*[-1]\bs{Y}}'$ and
   we call this the \emph{canonical orientation} for $\bfT^*[-1]\bs{Y}$.
\end{ex}

\subsection{Vanishing cycle complexes}

Let $f$ be a holomorphic function on a complex manifold $U$, and set $U_0 = f^{-1}(0)$ and $U_{\leq 0} = f^{-1}(\{z\in \mathbb{C} \mid \mathrm{Re}(z) \leq 0\})$.
The (shifted) vanishing cycle functor $\map{\varphi_f ^p}{D^b _c(U, \setQ)}{D^b (U_0, \setQ)}$ is defined by the composition of the functors
\[
 \varphi_f ^p \coloneqq (U_0 \hookrightarrow U_{\leq 0})^* (U_{\leq 0} \hookrightarrow U)^!.
\]

The functor $ \varphi_f ^p$ preserves the constructibility (see e.g. \cite[Definition 4.2.4, Proposition 4.2.9]{Dim}).
The canonical morphism $(U_{\leq 0} \hookrightarrow U)^! \to (U_{\leq 0} \hookrightarrow U)^*$ induces a natural transform
\begin{align}
 \gamma_f \colon \varphi_f ^p \to (U_0 \hookrightarrow U)^*.
\end{align}

Here we list basic properties of the functor $\varphi_f ^p$ we use later:

\begin{prop}\label{prop:van}
  Let $U$ be a complex manifold and $f$ be a holomorphic function on it. Write $U_0 =f^{-1}(0)$.
  \begin{itemize}
    \item[(i)]
    If $F$ is a perverse sheaf on $U$, then $\varphi_f^p (F)$ is also a perverse sheaf.
    \item[(ii)]
    The support of $\varphi_f ^p (\mathbb{Q}_U)$ is contained in $\crit(f)$.
    \item[(iii)]\label{prop:van3}
    Let $\map{q}{V}{U}$ be a holomorphic map where $V$ is a complex manifold,
    and $\map{q_0}{V_0}{U_0}$ denotes the restriction of $q$ where $V_0 = (f \circ q)^{-1}(0)$.
    Then we have a canonical morphism $\map{\Theta_{q, f}}{q_0 ^* \varphi_f^p (F)}{\varphi_{f\circ q} ^p (q^*F)}$ for each $F \in D^b _c(U,\setQ)$, which is an isomorphism if $q$ is a submersion.
    Further, the following diagram commutes:
    \begin{align}\label{eq:vanpullback}
      \begin{aligned}
      \xymatrix{
      q_0^* \varphi_f^p(F)  \ar[rr]^-{\Theta_{q, f}} \ar[dr]_{q_0^* \gamma_f(F)} & & \varphi_{f\circ q} ^p (q^*F) \ar[dl]^-{\gamma_{f\circ q}(q^*F)} \\
            & q_0^*(F |_{U_0}). &
      }
    \end{aligned}
    \end{align}
    \item[(iv)](Thom--Sebastiani)
    Let $V$ be a complex manifold, $g$ be a holomorphic function on it, and
    $f \boxplus g$ be the function on $U \times V$ defined by $(f \boxplus g)(u, v) = f(u) + g(v)$.
    For $F \in D^b _c(U,\setQ)$ and $G \in D^b _c(V,\setQ)$, we have a canonical isomorphism
    \[
  \mathcal{TS}_{f, g, F, G} \colon \varphi_{f \boxplus g}^p(F \boxtimes G) \vert_{U_0 \times V_0} \cong \varphi_f^p (F) \boxtimes \varphi_g^p (G)
    \]
    where $V_0 = g^{-1}(0)$.
    Further, the following diagram commutes:
    \begin{align}\label{eq:tscomm}
      \begin{aligned}
      \xymatrix@C=36pt{
      \varphi_{f \boxplus g}^p(F \boxtimes G) \vert_{U_0 \times V_0} \ar[rr]^-{\gamma_{f \boxplus g}(F \boxtimes G)\vert_{U_0 \times V_0}} \ar[d]_{\mathcal{TS}_{f, g, F, G}}^-{\simd} & {}
      & (F \boxtimes G) \vert_{U_0 \times V_0} \ar[d]^-{\simd} \\
      \varphi_f^p (F) \boxtimes \varphi_g^p (G)\ar[rr]^-{\gamma_{f}(F) \boxtimes \gamma_{g}(G)} & {} & F |_{U_0} \boxtimes G |_{U_0}
      }
    \end{aligned}
    \end{align}
    \item[(v)](Verdier duality)
    For $F \in D^b_c(U,\setQ)$, there exists a canonical isomorphism
    \[
    \mathbb{D}_{U_0}(\varphi_f^p(F)) \cong  \varphi_f^p(\mathbb{D}_U(F))
    \]
    where $\mathbb{D}_{U_0}$ and $\mathbb{D}_U$ denote the Verdier duality functors on $U_0$ and $U$ respectively.
  \end{itemize}

\end{prop}

\begin{proof}
  (i) is proved in \cite[Corollary 10.3.13]{KaSc}. (ii) easily follows from the definition.
  (iii) follows from the smooth base change theorem.
 (iv) is proved in \cite[Corollary 1.3.4]{Sch}.
 (v) is proved in \cite{Mas}.
\end{proof}

By abuse of notation, we write $\varphi_f ^p = \varphi_f ^p (\mathbb{Q}_U[\dim U])$ if there is no confusion.
We identify $i$-th cohomology of the stalk of $\varphi_f ^p$ at some point with the $(i + \dim U)$-th relative cohomology of a ball modulo
the Milnor fiber at a small positive value.
We regard $\varphi_f ^p$ as a perverse sheaf on $U$, $U_0$, or $\crit(f)$ depending on each situation.

If we write $z \colon \mathbb{C} \to \mathbb{C}$ to be the identity map, we have natural isomorphisms
\begin{align}
  \begin{aligned}
    (\varphi _{z^2} ^p)_0 &\cong
     \mathrm{H}^1 (\mathbb{C}, \{z \in \mathbb{C} \mid \mathrm{Re}(z^2) >0 \}; \mathbb{Q}) \\
     &\cong \mathrm{H}^1 (\mathbb{R}, \mathbb{R} \setminus 0; \mathbb{Q}).
\end{aligned}
\end{align}
The latter isomorphism is induced by the inclusion
\[
(\mathbb{R}, \mathbb{R} \setminus 0) \hookrightarrow (\mathbb{C}, \{z \in \mathbb{C} \mid \mathrm{Re}(z^2) >0 \}).
\]
The orientation of $\mathbb{R}$ given by the positive direction defines a cohomology class
$a_+ \in \mathrm{H}^1 (\mathbb{R}, \mathbb{R} \setminus 0; \mathbb{Q})$, hence a trivialization
\begin{align}\label{eq:quadvan1}
 h_{1,z} \colon \varphi _{z^2} ^p \cong \mathbb{Q}_{0}.
\end{align}
Let $(z_1, \ldots, z_n)$ be the standard coordinate of $\mathbb{C}^n$.
Then Thom--Sebastiani theorem and \eqref{eq:quadvan1} gives an isomorphism
\begin{align}\label{eq:quadvann}
  h_{n,(z_1, \ldots, z_n)} \colon \varphi _{z_1^2+\cdots+z_n ^2} ^p |_{(0, \ldots, 0)}
   \cong  \mathbb{Q}_{(0,\ldots,0)}.
\end{align}

We now recall the construction of the globalization of the vanishing cycle complexes associated with oriented d-critical schemes
introduced in \cite{BBDJS}.
To do this we introduce the following notation.
For a scheme $X$, a principal $\mathbb{Z}/2\mathbb{Z}$-bundle $P$ on $X$,
and $F \in D^b _c (X, \setQ)$, one defines
\[
F \otimes _{\mathbb{Z}/2\mathbb{Z}} P \coloneqq F \otimes _{\mathbb{Q}_X} (\mathbb{Q}_X \otimes _{\mathbb{Z}_X/2\mathbb{Z}_X} P)
\]
where
$\mathbb{Z}_X/2\mathbb{Z}_X$-module structure on $\mathbb{Q}_X$ is defined by the multiplication by $-1$.
For a d-critical scheme $(X, s)$ with a fixed orientation
$o \colon {L^{\otimes^2}} \xrightarrow{\sim}{K_{X,s}}$
and its d-critical chart $\mathscr{R} = (R, U, f, i)$, we define a principal $\mathbb{Z}/2\mathbb{Z}$-bundle
\begin{align}\label{eq:priQ}
Q_{\mathscr{R}}^o
\end{align}
 over $R$ whose sections are local isomorphisms
\[
\map{a}{L}{(i^* K_U)\vert _{R^{\red}}}
\]
such that $a^{\otimes ^2} = \iota_{R, U, f, i} \circ o$.
For an embedding of d-critical charts
\[
\Phi \colon \mathscr{R}_1 = (R, U_1, f_1, i_1) \hookrightarrow \mathscr{R}_2 = (R, U_2, f_2, i_2)
\]
such that $\alpha \colon U_2 \to U_1$ and $\beta \colon U_1 \to \mathbb{C}^n$
 as in Theorem \ref{thm:dcritemb} (ii) exist,
\[
(Q_{\mathscr{R}_1}^o)^{-1} \otimes Q_{\mathscr{R}_2}^o
\]
parameterizes square roots of
\[
i_1^* \beta^* (\D z_1\wedge \cdots \wedge \D z_n) |_{R^\red}^{\otimes^2}.
\]
Thus the choice
$i_1^* \beta^* (\D z_1\wedge \cdots \wedge \D z_n) |_{R^\red}$ gives an isomorphism
\begin{align}\label{eq:Qisom}
Q_{\mathscr{R}_1}^o \cong Q^o_{\mathscr{R}_2}.
\end{align}
On the other hand, we have isomorphisms
\begin{align}\label{eq:Visom}
  i_1 ^* \varphi_{f_1} ^p
  \cong i_2^* \alpha^* \varphi_{f_1} ^p  \otimes \beta^* \varphi_{z_1^2 + \cdots + z_n^2} ^p
  \cong i_2^* \varphi_{f_2}^p
\end{align}
where the first one is defined using \eqref{eq:quadvan1} and the second one is Thom--Sebastiani isomorphism.

Under these notations, the globalized vanishing cycle complex is defined by the following theorem:

\begin{thm}\cite[Theorem 6.9]{BBDJS}\label{thm:defvan}
  Let $(X, s, o)$ be an oriented d-critical scheme,
  $\mathscr{R}_1 = (R_1, U_1,\ab f_1, i_1)$ and $\mathscr{R}_2 = (R_2, U_2, f_2, i_2)$ be any d-critical charts with $R_1 \subset R_2$.
  Then there exists a natural isomorphism
  \[
    \map{\Upsilon_{\mathscr{R}_2, \mathscr{R}_1}}
    {i_1^*\varphi_{f_1}^p\otimes _{\mathbb{Z}/2\mathbb{Z}}Q^o_{\mathscr{R}_2}}
    {i_2^*\varphi_{f_2}^p\otimes _{\mathbb{Z}/2\mathbb{Z}}Q^o_{\mathscr{R}_1} \vert _{R_1}}
  \]
  with the following properties:
  \begin{itemize}
    \item[(i)]
      If we are given another d-critical chart $\mathscr{R}_3 = (R_3, U_3, f_3, i_3)$ with $R_2 \subset R_3$,
      we have
      \begin{align*}
        \Upsilon_{\mathscr{R}_3, \mathscr{R}_1}
        =\Upsilon_{\mathscr{R}_3, \mathscr{R}_2} \vert _{R_1}
        \circ \Upsilon_{\mathscr{R}_2, \mathscr{R}_1}.
      \end{align*}

    \item[(ii)]
    If $\mathscr{R}_1$ is an open subchart of $\mathscr{R}_2$,
    then $\Upsilon_{\mathscr{R}_2, \mathscr{R}_1}$ is defined by the canonical isomorphisms
    \begin{gather*}
    \varphi_{f_1}^p \cong \varphi_{f_2}^p |_{U_1}, \quad
    Q^o_{\mathscr{R}_1} \cong Q^o_{\mathscr{R}_1} |_{U_1}.
  \end{gather*}

    \item[(iii)]
    For an embedding of d-critical charts
    \[
    \mathscr{R}_1 = (R, U_1, f_1, i_1) \hookrightarrow \mathscr{R}_2 = (R, U_2, f_2, i_2)
    \] such that
    $\alpha \colon U_2 \to U_1$ and $\beta \colon U_1 \to \mathbb{C}^n$ as in Theorem \ref{thm:dcritemb} (ii) exist,
    $\Upsilon_{\mathscr{R}_2, \mathscr{R}_1}$ is defined by isomorphisms \eqref{eq:Qisom} and \eqref{eq:Visom}.
  \end{itemize}
  Using  (i), we can define a perverse sheaf $\varphi _{X, s, o}^p$ on $X$ such that for a given d-critical chart
  $\mathscr{R} = (R, U, f, i)$ there exists a natural isomorphism
  \[
  \omega_{\mathscr{R}} \colon \varphi _{X, s, o}^p |_R \cong i^* \varphi_f^p \otimes _{\mathbb{Z}/2\mathbb{Z}} Q^o_{\mathscr{R}}.
  \]
  Moreover, there exists an isomorphism $\sigma_{X, s, o} \colon \mathbb{D}_X(\varphi_{X, s, o}^p) \cong \varphi_{X, s, o}^p$.
\end{thm}

For later use, we recall the construction of $\sigma_{X, s, o}$.
For a d-critical chart $\mathscr{R} = (R, U, f, i)$,
the Verdier self-duality of $\varphi_f^p$ induces an isomorphism
\begin{align}\label{eq:dualchoice}
\sigma'_{\mathscr{R}} \colon \mathbb{D}_X(\varphi_{X, s, o}^p)|_R \cong \varphi_{X, s, o}^p|_R.
\end{align}
If we define
\begin{align*}
\sigma_{\mathscr{R}} = (-1)^{\dim U \cdot (\dim U -1)/2} \sigma'_{\mathscr{R}},
\end{align*}
one can show that it glues to define an isomorphism $\sigma_{X, s, o}$
(the necessity of the sign intervention is due to the fact that the first diagram in \cite[Theorem 2.13]{BBDJS} commutes up to the
sign $(-1)^{\dim U \cdot \dim V}$).
If there is no confusion, we write $\sigma_{X} = \sigma_{X, s, o}$.

For an oriented $-1$-shifted symplectic derived scheme $(\bs{X}, \omega_{\bs{X}}, o)$,
define $\varphi_{\bs{X}, \omega_{\bs{X}}, o}$ to be the perverse sheaf $\varphi^p_{X, s_X, o}$
 on $X = t_0(\bs{X})$ where $s_X$ is the d-critical structure associated with $\omega_{\bs{X}}$.
If there is no confusion, we simply write $\varphi_{\bs{X}, o}$ or $\varphi_{\bs{X}}$ instead of $\varphi_{\bs{X}, \omega_{\bs{X}}, o}$.

\subsection{Dimensional reduction}

Let $U$ be a smooth variety of dimension $n$, and $s$ be a section of a trivial vector bundle $E$ of rank $r$ on $U$.
Denote $\bar{s} \colon \Tot_U(E^\vee) \to \mathbb{A}^1$ the regular function corresponding to $s$.
We have a canonical morphism
\begin{align}\label{map:va}
\gamma_{\bar{s}} (\mathbb{Q}_{\Tot_U(E^\vee)}[n + r]) \colon
\varphi_{\bar{s}}^p \to \mathbb{Q}_{{\bar{s}}^{-1}(0)}[n + r].
\end{align}
Define $Z \coloneqq Z(s)$ to be the zero locus of $s$,
and $\widetilde{Z} \coloneqq (\pi_{E^\vee})^{-1}(Z)$ where $\pi_{E^\vee} \colon \Tot_U(E^\vee) \to U$ is the projection.
By restricting \eqref{map:va} to $\widetilde{Z}$,
we obtain
\begin{align}
\gamma_{\bar{s}} \coloneqq \gamma_{\bar{s}} (\mathbb{Q}_{\Tot_U(E^\vee)}[n + r]) |_{\widetilde{Z}}
 \colon \varphi_{\bar{s}}^p \to \mathbb{Q}_{\widetilde{Z}}[n + r].
\end{align}
Here we identify $\varphi_{\bar{s}}^p$ and $\varphi_{\bar{s}}^p | _{\widetilde{Z}}$
since the support of $\varphi_{\bar{s}}^p$ is contained in $\widetilde{Z}$.

\begin{thm}\cite[Theorem A.1]{Dav1}\label{thm:locdim}
  The natural map
  \begin{align}\label{eq:locdim}
     \bar{\gamma}_{\bar{s}} \coloneqq (\pi_{E^\vee})_! \gamma_{\bar{s}} \colon (\pi_{E^\vee})_! \varphi_{\bar{s}}^p \to (\pi_{E^\vee})_!
   \mathbb{Q}_{\widetilde{Z}}[n + r] \cong \mathbb{Q}_{Z}[n - r]
 \end{align}
  is an isomorphism.
\end{thm}

We want to globalize this statement for the $-1$-shifted cotangent space
 using Lemma \ref{lem:cotcrit}.
To do this, we first need to prove the triviality of the $\setZ / 2\setZ$-bundle introduced in \eqref{eq:priQ} associated with the canonical orientations for shifted cotangent schemes
and certain d-critical charts:

\begin{lem}\label{lem:bdltriv}
  Let $U, s$ be as above, and $\bs{Z} \coloneqq \bs{Z}(s)$ be the derived zero locus of $s$.
  Assume $U$ is affine and carries a global \'{e}tale coordinate.
  Denote by $o^{} _{\bfT^*[-1]\bs{Z}}$ the canonical orientation constructed in Example \ref{ex:canori} and
  \[
  \widetilde{\mathscr{Z}} = (\Crit(\bar{s}), \Tot_U(E^\vee), \bar{s}, i)
  \]
   the d-critical chart induced by the equivalence
  \eqref{eq:cotcrit}.
  Then $Q^{o^{} _{\bfT^*[-1]\bs{Z}}}_{\widetilde{\mathscr{Z}}}$ is a trivial $\setZ / 2\setZ$-bundle.
\end{lem}

\begin{proof}
  The distinguished triangle \eqref{eq:cantri} for $\bs{Z}$ restricted to $\Crit(\bar{s})$ is represented by the following short exact sequence of two term complexes:
  \begin{align}\label{eq:shcot}
    \begin{aligned}
    \xymatrix{
    0 \ar[r]  & {\pi_{E^\vee}^*}E^\vee |_{\crit(\bar{s})} \ar[r] \ar[d]^{(\D s)^\vee}
    & T_{\Tot_U(E^\vee)}|_{\crit(\bar{s})} \ar[r] \ar[d]^{\mathrm{Hess(\bar{s})}}
    & {\pi_{E^\vee}^*}T_U  |_{\crit(\bar{s})} \ar[r] \ar[d]^{\D s}
    & 0 \\
    0 \ar[r] & {\pi_{E^\vee}^*}\Omega_U |_{\crit(\bar{s})} \ar[r]
    & \Omega_{\Tot_U(E^\vee)}|_{\crit(\bar{s})}  \ar[r]
    & {\pi_{E^\vee}^*}E |_{\crit(\bar{s})} \ar[r]
    & 0.
    }
  \end{aligned}
\end{align}

    Thus the canonical orientation \eqref{eq:canori} for $\dcrit(\bar{s})$ is identified with
    \begin{align}\label{eq:canoriex}
    \mu \colon (\detr(\pi_{E^\vee}^* \Omega_U)  \otimes \detr(\pi_{E^\vee}^* E^\vee)^{-1})^{\otimes^2}  \cong
    \detr(\Omega_{\Tot_U(E^\vee)})^{\otimes^2}.
    \end{align}
    where
    \[
    \detr(\pi_{E^\vee}^* E)  \otimes \detr(\pi_{E^\vee}^* T_U)^{-1} \cong
    \detr(\pi_{E^\vee}^* \Omega_U)  \otimes \detr(\pi_{E^\vee}^* E^\vee)^{-1}
    \]
    and
    \[
    \det(T_U)^{-1} \cong \det(\Omega_U)
    \]
    are defined in the same manner as \eqref{eq:dualconv}.
    On the other hand, we have an isomorphism
    \begin{align}\label{eq:natsq}
     \begin{aligned}
      \nu \colon \detr(\pi_{E^\vee}^* \Omega_U)  \otimes \detr(\pi_{E^\vee}^* E^{\vee})^{-1}
      &\cong
      \detr(\pi_{E^\vee}^* \Omega_U)  \otimes \detr(\pi_{E^\vee}^* E)  \\
      &\cong \detr(\Omega_{\Tot_U(E^\vee)}).
    \end{aligned}
  \end{align}
    where the first isomorphism is defined as \eqref{eq:dualconv} and the second isomorphism is induced by the lower short exact sequence in \eqref{eq:shcot}.
    By the definition of $o^{} _{\bfT^*[-1]\bs{Z}}$ and \eqref{eq:dcritcan} we see that
      \begin{align}\label{eq:munu}
        \mu = (-1)^{ (n+r)(n+r-1)/2 } (1/2)^{n + r} \cdot \nu ^{\otimes^2}
      \end{align}
      where we write $n = \dim U$ and $r = \rank E$.
      The appearance of the sign $(-1)^{ (n+r)(n+r-1)/2 } $ is caused by the difference of
      the maps \eqref{eq:dualconv} and \eqref{eq:dualconv2},
      and the difference of the symmetric monoidal structure for the category of graded line bundles \eqref{eq:symmon}
       and the standard symmetric monoidal structure for the category of ungraded line bundles.
    The equality \eqref{eq:munu} implies the triviality of $Q^{o^{} _{\bfT^*[-1]\bs{Z}}}_{\widetilde{\mathscr{Z}}}$.
\end{proof}

For later use, we explicitly choose a trivialization of $Q^{o^{} _{\bfT^*[-1]\bs{Z}}}_{\widetilde{\mathscr{Z}}}$.
For each $(a, b) \in \mathbb{Z}_{\geq 0}^2$, take $\epsilon_{a, b} \in \{1, -1, \sqrt{-1}, -\sqrt{-1}\}$ so that
\begin{itemize}
\item $\epsilon_{0, 0} = 1$.
\item $\epsilon_{a, b} = (-1)^b \sqrt{-1} \epsilon_{a-1, b-1}$
\item $\epsilon_{a+1, b} = (-\sqrt{-1})^{a - b}\epsilon_{a, b}$.
\end{itemize}
Then
\begin{align}\label{eq:rootchoice}
  \epsilon_{n, r} (1/\sqrt{2})^{n+r} \cdot \nu
\end{align}
gives a square root of $\mu$, hence a trivialization of $Q^{o^{} _{\bfT^*[-1]\bs{Z}}}_{\widetilde{\mathscr{Z}}}$.

\begin{cor}\label{cor:ls}
  Let $\bs{Y}$ be a quasi-smooth derived scheme
  , $\bs{\pi}_{\bs{Y}} \colon \bfT^*[-1]\bs{Y} \to \bs{Y}$ the projection and write $\pi_{\bs{Y}} = t_0(\bs{\pi}_{\bs{Y}})$.
  Then $({\pi_{\bs{Y}}})_! \varphi_{\bfT^*[-1]\bs{Y}}^p$ is a rank one local system shifted by $\vdim \bs{Y}$.
\end{cor}

\begin{proof}
  Since the statement is local, we may assume $\bs{Y}$ is a derived zero locus $\bs{Z}(s)$
  as in the previous lemma.
  The conclusion of the lemma implies that $\varphi_{\bfT^*[-1]\bs{Y}}^p$ is isomorphic $\varphi_{\bar{s}}^p$
  hence the statement follows from Theorem \ref{thm:locdim}.
\end{proof}

\section{Dimensional reduction for schemes}\label{section3}

In this section, we will prove that the local system appeared in
Corollary \ref{cor:ls} is in fact trivial,
by showing that the local dimensional reduction isomorphism \eqref{eq:locdim} is independent of the choice of
the Kuranishi chart.

Let $\bs{Y}$ be a quasi-smooth derived scheme and $\bs{\pi}_{\bs{Y}} \colon \bfT^*[-1]\bs{Y} \to \bs{Y}$ be the projection. We always equip $\bfT^*[-1]\bs{Y}$ with the $-1$-shifted symplectic form constructed in Example \ref{ex:cot} and the canonical orientation
\[
 o = o^{}_{\bfT^*[-1]\bs{Y}}.
 \]
Write ${\pi_{\bs{Y}}} = t_0(\bs{\pi}_{\bs{Y}} )$,
 $\widetilde{Y} = t_0(\bfT^*[-1]\bs{Y})$, and  $Y = t_0(\bs{Y})$.
Take a good Kuranishi chart
\[
\mathscr{Z} = (Z, U, E, s, \bs{\iota})
\] of $\bs{Y}$.
The map $\bs{\iota}$ induces an open immersion $\widetilde{\bs{\iota}} \colon \bfT^*[-1]\bs{Z}(s) \hookrightarrow \bfT^*[-1]\bs{Y}$
with the image $\widetilde{Z} \coloneqq \pi_{\bs{Y}}^{-1}(Z)$.
Lemma \ref{lem:cotcrit} shows that there exists a natural embedding $\tilde{i} \colon \widetilde{Z} \hookrightarrow \Tot_U(E^\vee)$
such that
\[
\widetilde{\mathscr{Z}} = (\widetilde{Z}, \Tot_U(E^\vee), \bar{s}, \tilde{i})
\] gives a d-critical chart on $\widetilde{Y}$.

Now we have an isomorphism
\begin{align}\label{eq:nat1}
  (\pi_{\bs{Y}})_! \varphi_{\bfT^*[-1]\bs{Y}}^p |_Z \cong i^* (\pi_{E^\vee})_! \tilde{i}_* (\varphi_{\bfT^*[-1]\bs{Y}}^p |_{\widetilde{Z}})
\end{align}
where $\pi_{E^\vee} \colon \Tot_U(E^\vee) \to U$ is the projection and $i \colon Z \hookrightarrow U$ is the natural embedding.
By the definition of $\varphi_{\bfT^*[-1]\bs{Y}}^p$ in Theorem \ref{thm:defvan} we also have an isomorphism
\begin{align}\label{eq:nat2}
  \omega_{\widetilde{\mathscr{Z}}} \colon \tilde{i}_* (\varphi_{\bfT^*[-1]\bs{Y}}^p |_{\widetilde{Z}} \otimes_{\mathbb{Z}/2\mathbb{Z}}
    (Q^o_{\widetilde{\mathscr{Z}}})^{-1}) \cong \varphi_{\bar{s}}^p .
\end{align}
By combining isomorphisms \eqref{eq:nat1} and \eqref{eq:nat2}, the trivialization of $Q^o_{\widetilde{\mathscr{Z}}}$ in \eqref{eq:rootchoice},
 and Theorem \ref{thm:locdim}, we obtain the following isomorphism:
\begin{align}
    \bar{\gamma}_{\mathscr{Z}} \colon (\pi_{\bs{Y}})_! \varphi_{\bfT^*[-1]\bs{Y}}^p |_Z \cong \mathbb{Q}_Z[\vdim \bs{Y}].
\end{align}

\begin{thm}\label{thm:dimred1}
   For $i=1, 2$, let $\mathscr{Z}_i = (Z_i, U_i, E_i, s_i, \bs{\iota}_i)$ be good Kuranishi charts on $\bs{Y}$.
   Then we have
   $\bar{\gamma}_{\mathscr{Z}_1} |_{Z_1 \cap Z_2} = \bar{\gamma}_{\mathscr{Z}_2} |_{Z_1 \cap Z_2}$.
   Therefore there exists a natural isomorphism
   \[
   \bar{\gamma}_{\bs{Y}} \colon (\pi_{\bs{Y}})_! \varphi_{\bfT^*[-1]\bs{Y}}^p \cong \mathbb{Q}_{Y}[\vdim \bs{Y}].
   \]
\end{thm}

We say that two Kuranishi charts $\mathscr{Z}_1 = (Z_1, U_1, E_1, s_1, \bs{\iota}_1)$
and $\mathscr{Z}_2 = (Z_2, U_2, E_2, s_2, \bs{\iota}_2)$ have
\textit{compatible dimensional reductions} at $p \in Z_1 \cap Z_2$ if there exists an analytic open neighborhood
$p \in W \subset Z_1 \cap Z_2$ such that
 $\bar{\gamma}_{\mathscr{Z}_1} |_{W} = \bar{\gamma}_{\mathscr{Z}_2} |_{W}$.

\begin{lem}\label{lem:mincomp}
  Let $\mathscr{Z}_1 = (Z_1, U_1, E_1, s_1, \bs{\iota}_1)$ and $\mathscr{Z}_2 = (Z_2, U_2, E_2, s_2, \bs{\iota}_2)$ be good Kuranishi charts on $\bs{Y}$.
  Assume that these Kuranishi charts are minimal at $p \in Z_1 \cap Z_2$.
  Then they have compatible dimensional reductions at $p$.
\end{lem}

\begin{proof}
  Denote by $i_1 \colon Z_1 \hookrightarrow U_1$ and $i_2 \colon Z_2 \hookrightarrow U_2$ the natural embeddings, and
  \[
  \widetilde{\mathscr{Z_1}} = (\widetilde{Z_1}, \Tot_{U_1}(E_1^\vee), \bar{s}_1, \tilde{i}_1), \,\,
  \widetilde{\mathscr{Z_2}} = (\widetilde{Z_2}, \Tot_{U_2}(E_2^\vee), \bar{s}_2, \tilde{i}_2)
  \]
  be d-critical charts associated with $\mathscr{Z}_1$ and $\mathscr{Z}_2$ respectively.
  Using Proposition \ref{prop:qslocal} (ii),
  we may assume that we have the following commutative diagram
\begin{align*}
\xymatrix{
 \Tot_{U_1}(E_1^\vee) \ar[r]_-{\tau^\vee}^-{\sim} \ar@/^20pt/[rrr]^{\bar{s}_1} \ar[rd]& \Tot_{U_1}(\eta^* E_2^\vee) \ar[r]_-{\tilde{\eta}} \pbcorner \ar[d] & \Tot_{U_2}(E_2^\vee) \ar[r]_-{\bar{s}_2} \ar[d] & \mathbb{A}^1 \\
 & U_1 \ar[r]^-{\eta} & U_2&
}
\end{align*}
such that $\eta$ is \'etale and $\eta(i_1(p)) = i_2(p)$.
The natural isomorphism
  \[
  (\tilde{\eta}\circ \tau^\vee)^* K_{\Tot_{U_2}(E_2^\vee)} \cong K_{\Tot_{U_1}(E_1^\vee)}
  \] induces an isomorphism
  \[
  Q^o_{\widetilde{\mathscr{Z}_1}} \cong  Q^o_{\widetilde{\mathscr{Z}_2}} |_{\widetilde{Z_1}}
  \] that identifies
  the trivializations \eqref{eq:rootchoice}.
  Then Threorem \ref{thm:defvan} (ii) implies that the following composition
  \begin{align}
    \begin{aligned}
    \tilde{i}_1^* \varphi_{\bar{s}_1}^p \cong \tilde{i}_1^*(\varphi_{\bar{s}_1}^p \otimes_{\mathbb{Z}/2\mathbb{Z}} Q^o_{\widetilde{\mathscr{Z}_1}})
    \cong \tilde{i}_2^*(\varphi_{\bar{s}_2}^p \otimes_{\mathbb{Z}/2\mathbb{Z}} Q^o_{\widetilde{\mathscr{Z}_2}}) |_{\widetilde{Z_1}}
    \cong (\tilde{i}_2^*\varphi_{\bar{s}_2}^p) |_{\widetilde{Z_1}}
  \end{aligned}
  \end{align}
  is the natural isomorphism
  \[
  \varphi^p_{\bar{s}_1} \cong \eta^* \varphi^p_{\bar{s}_2}
  \] given in Proposition \ref{prop:van} (iii)
  pulled back to $\widetilde{Z_1}$.
  Hence the commutativity of the diagram \eqref{eq:vanpullback} implies the lemma.
  \end{proof}

  \begin{prop}\label{prop:dimredred}
    Let $\mathscr{Z} = (Z, U, E, s, \bs{\iota})$ be a good Kuranishi chart on $\bs{Y}$, which is not minimal at $p \in Z$.
    Then there exists another good Kuranishi chart $\mathscr{Z}' = (Z', U', E', s', \bs{\iota}')$ with $p \in Z'$ and $\dim U' < \dim U$
    such that $\mathscr{Z}$ and $\mathscr{Z}'$ have compatible dimensional reductions at $p$.
  \end{prop}

\begin{proof}
  Take a trivialization
  \[
  E = \mathcal{O}_U \cdot e_1 \oplus \cdots \oplus \mathcal{O}_U \cdot e_r
  \]
  and write
  \[
  s = f_1 e_1 + \cdots f_r e_r.
  \]
  By the non-minimality assumption, we may assume that $f_1 \neq 0$ and
  the zero locus $Z(f_1)$ is smooth at $p$.
  Take a smooth affine open neighborhood $p \in U' \subset Z(f_1)$ and
  define a vector bundle $E'$ on $U'$ by
  \[
  E' \coloneqq (\mathcal{O}_U \cdot e_2 \oplus \cdots \oplus \mathcal{O}_U \cdot e_r )|_{U'}.
  \]
  Let $s' \in \Gamma(U', E')$ be the section induced by $s|_{U'}$.
  Then we obtain a natural open immersion of the derived zero loci
  $\bs{Z}(s') \hookrightarrow \bs{Z}(s)$.
  Define $\bs{\iota}' \colon \bs{Z}(s') \to \bs{Y}$ by the composition
  $\bs{Z}(s') \hookrightarrow \bs{Z}(s) \xrightarrow{\bs{\iota}} \bs{Y}$, and denote its image by $Z'$.
  By shrinking around $p$ if necessary, we may assume that
  \[
  \mathscr{Z}' = (Z', U', E', s', \bs{\iota}')
  \]
   is a good Kuranishi chart.
  We prove that this Kuranishi chart has the desired property.

  Firstly take a local coordinate $x_1, \ldots, x_n$ of $U$ around $p_U = \bs{\iota}^{-1}(p)$ with $x_1 = f_1$
  and an analytic open neighborhood $p_U \in V$ in $U$
  which maps biholomorphically to a polydisc $B_{\epsilon}^n \subset \setC^n$ under $(x_1, \ldots, x_n)$
  Write $V' = V \cap U'$.
  By shrinking $V$ if necessary,
  we can write
  \[
  f_i |_V = x_1 h_i + r_i \circ \bar{\alpha}
  \]
  for each $i \in \{ 2, \ldots, n\}$ where $h_i$ is a holomorphic function on $V$,
  $r_i$ is a holomorphic function on $V'$, and $\bar{\alpha} \colon V \to V'$ is the map
  identified with the projection $B_{\epsilon}^n \to B_{\epsilon}^{n-1}$.

  Let
  \[\widetilde{\mathscr{Z}} = (\widetilde{Z}, \Tot_{U}(E^\vee), \bar{s}, \tilde{i}), \,\,
  \widetilde{\mathscr{Z}'} = (\widetilde{Z'}, \Tot_{U'}((E')^\vee), \bar{s}', \tilde{i}')
  \]
  be d-critical charts on $\widetilde{Y}$ associated with $\mathscr{Z}$ and $\mathscr{Z}'$ respectively.
  Write
  \[
  \widetilde{\mathscr{Z}_V} = (\widetilde{W}, \widetilde{V}, \bar{s}|_{\widetilde{V}}, \tilde{i}|_{\widetilde{W}}), \,\,
  \widetilde{\mathscr{Z}'_{V'}} = (\widetilde{W}, \widetilde{V}', \bar{s}'|_{\widetilde{V}'}, \tilde{i}'|_{\widetilde{W}})
  \]
  the restrictions of $\widetilde{\mathscr{Z}}$ and $\widetilde{\mathscr{Z}'}$,
  where we define
  \begin{gather*}
    W \coloneqq i^{-1}(V), \,\, \widetilde{W} \coloneqq \pi_{\bs{Y}}^{-1}(W), \,\,
    \widetilde{V} \coloneqq \Tot_{V}(E |_V),\,\, \widetilde{V}' \coloneqq \Tot_{V'}(E' |_{V'}).
  \end{gather*}
  To simplify the notation, we write
  \[
  e_i = e_i |_V , \,\,\bar{s} = \bar{s}|_{\widetilde{V}}, \,\, \bar{s}' = \bar{s}'|_{\widetilde{V'}}, \,\, \tilde{i} = \tilde{i}|_{\widetilde{W}}, \,\, \tilde{i}' = \tilde{i}'|_{\widetilde{W}}.
  \]
  Define a closed immersion linear over $V' \hookrightarrow V$
  \[
   \Phi \colon \widetilde{V}' \hookrightarrow \widetilde{V}
  \]
  by
  \[
  e_i ^\vee |_{V'} \mapsto -h_i e_1 ^\vee + e_i^\vee
  \]
  where $e_1^\vee, \ldots, e_r^\vee$ is the dual basis of $e_1, \ldots, e_r$.
  A direct computation shows that $\Phi$ defines an embedding of d-critical charts
  \[
  \widetilde{\mathscr{Z}_V} \hookrightarrow \widetilde{\mathscr{Z}'_{V'}},
  \]
  in other words, we have the following commutative diagram
  \[
  \xymatrix@C=50pt{
  \widetilde{V}' \ar@/^20pt/[rr]^{\bar{s}'} \ar[r]_{\Phi} & \widetilde{V} \ar[r]_-{\bar{s}} & \mathbb{A}^1 \\
  \widetilde{W} \ar@{=}[r] \ar@{^{(}->}[u]^-{\tilde{i}'} & \widetilde{W} \ar@{^{(}->}[u]_-{\tilde{i}} &
  }
  \]
  such that
  \[\Crit(\bar{s}') = \mathrm{Im}(\tilde{i}'), \,\,
  \Crit(\bar{s}) = \mathrm{Im}(\tilde{i}).
  \]
  Now define a map linear over the projection $\bar{\alpha} \colon V \twoheadrightarrow V'$
  \[
  \alpha \colon \widetilde{V} \twoheadrightarrow \widetilde{V}'
  \]
   by
  \[
  e_i ^\vee \mapsto e_i^\vee |_{V'}
  \]
  and a map linear over the projection $x_1 \colon V \twoheadrightarrow B_{\epsilon}$
  \[
  \beta = (x_1, y) \colon \widetilde{V} \twoheadrightarrow B_{\epsilon} \times \setC
  \]
  by
  \[
  e_1^\vee \mapsto 1, \,\,\,\, e_i^\vee \mapsto h_i\,(i>1).
  \]
  It is clear by the construction that
  \[
  \alpha \circ \Phi = \id_{\widetilde{V}'}, \,\, \beta \circ \Phi = 0,
  \]
  the map $(\alpha, \beta) \colon \widetilde{V} \to \widetilde{V}' \times (B_{\epsilon} \times \setC)$
  is isomorphic, and the following diagram commutes:
  \[
  \xymatrix@C=50pt{
  \widetilde{W} \ar[d]_-{\pi_{\bs{Y} }|_{\widetilde{W}}}  \ar@{^{(}->}[r]_-{\tilde{i}} & \widetilde{V} \ar@/^20pt/[rr]^-{\bar{s}} \ar[r]_-{(\alpha, \beta)}^-{\sim} \ar[d]_-{\pi_V}
  & \widetilde{V}' \times (B_{\epsilon} \times \setC) \ar[r]_-{\bar{s}' \boxplus x_1 y} \ar[d]_-{\pi_{V'} \times x_1}  & \mathbb{A}^1 \\
  W \ar@{^{(}->}[r]_-{i } & V \ar[r]_-{(\bar{\alpha}, x_1)}^-{\sim}  & {V}' \times B_{\epsilon} &
  }
  \]
  Here $\pi_V$ and $\pi_{V'}$ are natural projections.

Consider the following composition of morphisms of perverse sheaves on $\widetilde{W}$
  \begin{align}\label{eq:vanred}
    \begin{split}
    \tilde{i}^* \varphi_{\bar{s}}^p
     \cong \tilde{i}^*\varphi_{\bar{s}}^p \otimes_{\mathbb{Z}/2\mathbb{Z}} Q^o_{\widetilde{\mathscr{Z}_V}}
    \cong \varphi_{\bfT^*[-1]\bs{Y}}
    \cong (\tilde{i}')^*\varphi_{\bar{s}'}^p \otimes_{\mathbb{Z}/2\mathbb{Z}} Q^o_{\widetilde{\mathscr{Z}'_{V'}}}
    \cong (\tilde{i}')^*\varphi_{\bar{s}'}^p
  \end{split}
  \end{align}
  where the first and final isomorphisms are induced by \eqref{eq:rootchoice},
  the second and third isomorphisms are $\omega^{-1}_{\widetilde{\mathscr{Z}}}$ and
  $\omega_{\widetilde{\mathscr{Z}'}}$ defined in Theorem \ref{thm:defvan} respectively.
  Now we show that this is equal to the following composition
  \begin{align}\label{eq:vanred2}
    \tilde {i}^* \varphi_{\bar{s}}^p
     \cong {i}^*(\alpha, \beta)^*(\varphi_{\bar{s}'}^p \boxtimes \varphi_{{x_1} y}^p)
     \cong (\tilde{i}')^*\varphi_{\bar{s}'}^p
  \end{align}
  where the first map is the Thom--Sebastiani isomorphism,
  and the second map is constructed by substituting $z_1 = (x_1 - y) / 2 \sqrt{-1}$ and $z_2 = (x_1 + y) / 2$ for $h_{2,(z_1, z_2)}$ in \eqref{eq:quadvann}.
  To do this, it suffices to prove the commutativity of the following diagram, thanks to Theorem \ref{thm:defvan} (iii):
  \begin{align*}
    \xymatrix@C=130pt{
    (\setZ / 2\setZ)_{\widetilde{W}} \ar[r]^-{1 \mapsto \epsilon_{n, r} (1/\sqrt{2})^{n+r} \cdot \nu} \ar@{=}[d]
    & Q^o_{\widetilde{\mathscr{Z}}}  \ar[d] \\
    (\setZ / 2\setZ)_{\widetilde{W}} \ar[r]^-{1 \mapsto \epsilon_{n-1, r-1} (1/\sqrt{2})^{n+r-2} \cdot {\nu'}}
    & Q^o_{\widetilde{\mathscr{Z}'}}.
    }
  \end{align*}
  Here the right vertical map is \eqref{eq:Qisom}
  and
  \begin{align*}
  \nu \colon \tilde{i}^* K_{\widetilde{V}} \cong (\pi_{\bs{Y}}^\red)^* \detr(\L_{\bs{Y}})|_{\widetilde{W}}, \,\,\,
  \nu' \colon (\tilde{i}') ^* K_{\widetilde{V}'} \cong (\pi_{\bs{Y}}^\red)^* \detr(\L_{\bs{Y}})|_{\widetilde{W}}
  \end{align*}
  are constructed in the same manner as \eqref{eq:natsq}.
  The commutativity of the diagram above is equivalent to the commutativity of the following diagram
  \begin{align*}
    \xymatrix@C=50pt{
   (\pi_{\bs{Y}}^\red)^* \detr(\L_{\bs{Y}})|_{\widetilde{W}} \ar[r]^-{\epsilon_{n, r} (1/\sqrt{2})^{n+r} \cdot \nu} \ar[d]_-{\epsilon_{n-1, r-1} (1/\sqrt{2})^{n+r-2} \cdot {\nu'}}
  &(\tilde{i} ) ^*  K_{\widetilde{V}}   \ar[d]^-{\simd}  \\
   (\tilde{i}') ^*  K_{\widetilde{V'}}
   \ar[r]^-{a \mapsto a \wedge \D z_1 \wedge \D z_2}
   &(\tilde{i} ) ^* (\alpha, \beta)^* K_{V' \times \setC^2}  .
     }
  \end{align*}
  The commutativity of this diagram follows by the definitions of $\nu$ and $\nu'$, and the equations
  $\epsilon_{n, r}/\epsilon_{n-1, r-1}=(-1)^r \sqrt{-1}$ and $\D z_1 \wedge \D z_2 = (-\sqrt{-1}/2) \D x_1 \wedge \D y$.
  Therefore we have obtained the equality of isomorphisms \eqref{eq:vanred} and \eqref{eq:vanred2}.

   Now consider the following commutative diagram
   \begin{align*}
     \xymatrix@C=50pt{
       \varphi_{\bar{s}}^p  \ar[rr]^-{\gamma_{\bar{s}} } \ar[d]_-{\simd}
       & & \mathbb{Q}_{\widetilde{V}}[n + r] \ar[d]^-{\simd} \\
      {\begin{subarray}{c}
        \ds (\alpha, \beta)^*
         (\varphi_{\bar{s}'}^p \boxtimes \varphi_{{x_1} y}^p)
      \end{subarray}}
      \ar[rr]^-{
            {\begin{subarray}{c}
      (\alpha, \beta)^*
      ({\gamma_{\bar{s}'}} \boxtimes \gamma_{x_1 y})
            \end{subarray}}
      }
       & &
       {\begin{subarray}{c}
      \ds (\alpha, \beta)^*
       (\mathbb{Q}_{\widetilde{V}'}[n + r -2] \boxtimes \mathbb{Q}_{ (0, 0) }[2])
      \end{subarray}}
     }
   \end{align*}
   where the left vertical map is induced by the Thom--Sebastiani isomorphism.
   The commutativity follows from the commutativity of the diagram \eqref{eq:tscomm}.
   By applying the functor $(\pi_{V})_!$, we obtain the following commutative diagram
   \begin{align*}
     \xymatrix{
       (\pi_{V})_! \varphi_{\bar{s}}^p  \ar[r]^-{\sim} \ar[d]_-{\simd}
        &       \mathbb{Q}_{V}[n - r] \ar[dd]^-{\simd}   \\
      (\pi_{V})_!(\alpha, \beta)^*  (\varphi_{\bar{s}'}^p \boxtimes \varphi_{{x_1} y} ^p)
      \ar[d]_-{\simd}
       & \\
      (\bar{\alpha}, x_1)^* ((\pi_{V'})_!\varphi_{\bar{s}'}^p \boxtimes (x_1)_!\varphi_{{x_1} y}^p)  \ar[r]^-{\sim} &
      (\bar{\alpha}, x_1)^*(\mathbb{Q}_{\widetilde{V}'}[n - r] \boxtimes \mathbb{Q}_{ (0, 0) })
     }
   \end{align*}
   By combining the commutativity of the diagram above and the equality of isomorphisms
   \eqref{eq:vanred} and \eqref{eq:vanred2},
   the proposition follows from the next lemma.
\end{proof}

\begin{lem}
  The following diagram commutes
  \begin{align}\label{eq:quaddimered}
    \begin{aligned}
      \xymatrix@C=130pt{
      (x_1)_!\varphi^p_{x_1 y} \ar[r]^-{(x_1)_! h_{2, (z_1, z_2)}} \ar[d]_-{\simu} ^-{\bar{\gamma}_{y}} & (x_1)_! \mathbb{Q}_{(0, 0)} \ar[d]^{\simu} \\
      \mathbb{Q}_0  \ar@{=}[r] & \mathbb{Q}_0
      }
    \end{aligned}
  \end{align}
  where we use $z_1 = (x_1 - y) / 2 \sqrt{-1}$ and $z_2 = (x_1 + y) / 2$.
\end{lem}

  \begin{proof}
    Since $x_1y$ is homogenous and $\Crit(x_1 y)$ has compact support,
    we have natural isomorphisms
    \begin{align}\label{eq:vanThom}
      \begin{aligned}
      \mathrm{H}^0(\mathbb{C}, (x_1)_!\varphi_{x_1y}^p) &\cong
       \mathrm{H}^2(\mathbb{C}^2, \{(x_1, y) \in \mathbb{C}^2 \mid \mathrm{Re}(x_1 y)>0\};\mathbb{Q}) \\
       &\cong \mathrm{H}^2(\mathbb{C}^2, \mathbb{C}^2 \setminus \mathbb{C} \times \{ 0 \};\mathbb{Q}).
     \end{aligned}
    \end{align}
    The left vertical map in \eqref{eq:quaddimered} is given by the Thom class of
    $\mathrm{H}^2(\mathbb{C}^2,\mathbb{C}^2 \setminus \mathbb{C} \times \{ 0 \};\mathbb{Q})$ and \eqref{eq:vanThom}.
    Now consider the following composition of isomorphisms
    \begin{align}\label{eq:vantriv}
      \begin{aligned}
        \mathrm{H}^0(\mathbb{C}, (x_1)_!\varphi_{x_1y}^p) &\cong
         \mathrm{H}^2(\mathbb{C}^2, \{(x_1, y) \in \mathbb{C}^2 \mid \mathrm{Re}(x_1 y)>0\};\mathbb{Q}) \\
         &=   \mathrm{H}^2(\mathbb{C}^2, \{(z_1, z_2) \in \mathbb{C}^2 \mid \mathrm{Re}(z_1^2 + z_2^2)>0\};\mathbb{Q}) \\
         &\cong  \mathrm{H}^2(\mathbb{C}^2,
         \{(z_1, z_2) \in \mathbb{C}^2 \mid \mathrm{Re}(z_1^2)>0,\text{or} \  \mathrm{Re}(z_2^2)>0\};\mathbb{Q}) \\
         &\cong \mathrm{H}^1(\mathbb{C}, \{z_1 \in \mathbb{C} \mid \mathrm{Re}(z_1 ^2)>0 \};\mathbb{Q}) \otimes
                \mathrm{H}^1(\mathbb{C}, \{z_2 \in \mathbb{C} \mid \mathrm{Re}(z_2 ^2)>0) \};\mathbb{Q}) \\
         &\cong \mathbb{Q} \otimes \mathbb{Q} \cong \mathbb{Q}.
      \end{aligned}
    \end{align}
    The third isomorphism is the relative Kunneth isomorphism
    and the fourth isomorphism is given by \eqref{eq:quadvan1}.
    Since the Thom--Sebastiani isomorphism is induced by the relative Kunneth isomorphism (see \cite[p.62]{Sch}),
    this composition corresponds to $h_{2, (z_1, z_2)}$.
    Therefore we only need to show that
    \[
    \mathbb{Q} \to \mathrm{H}^2(\mathbb{C}^2, \mathbb{C}^2 \setminus \mathbb{C} \times \{ 0 \};\mathbb{Q})
    \]
    constructed by combining \eqref{eq:vanThom} and \eqref{eq:vantriv} gives the Thom class.
    Consider the composition
    \begin{align}
      \mathbb{R}^2 \xrightarrow{(z_1, z_2)} \mathbb{C}^2 \xrightarrow{y} \mathbb{C}.
    \end{align}
    If we equip $\mathbb{R}^2$ with the product orientation of the positive directions,
     this composition preserves the orientation.
    This proves the claim.
  \end{proof}

By repeatedly using Proposition \ref{prop:dimredred}, we obtain the following corollary.

\begin{cor}\label{cor:dimredred}
  Under the assumption of Proposition \ref{prop:dimredred},
  there exists a good Kuranishi chart $\mathscr{Z'} = (Z', U', E', s', \bs{\iota}')$ containing and minimal at $p$,
  such that $\mathscr{Z}$ and $\mathscr{Z'}$ have compatible dimensional reductions at $p$.
\end{cor}

\begin{proof}[Proof of Theorem \ref{thm:dimred1}]
By the sheaf property, it suffices to show that $\mathscr{Z}_1$ and $\mathscr{Z}_2$
have compatible dimensional reductions at each $p \in Z_1 \cap Z_2$.
By Corollary \ref{cor:dimredred}, we may assume these Kuranishi charts are minimal at $p$,
and then the claim follows from Lemma \ref{lem:mincomp}.
\end{proof}

\begin{rem}
  For a d-critical scheme $(X, s)$, it is shown in \cite[\S 6.4]{BBDJS} that $\varphi_{X, s}$
  has a natural extension to a mixed Hodge module.
  We can extend Theorem \ref{thm:dimred1} to an isomorphism of mixed Hodge modules with the same proof as above.
\end{rem}

\section{Dimensional reduction for stacks}\label{section4}

The aim of this section is to extend Theorem \ref{thm:dimred1} to quasi-smooth derived Artin stacks.

\subsection{Lisse-analytic topology}

We briefly recall the theory of lisse-analytic topos introduced in \cite{Sun},
which is a complex analytic analogue of the lisse-\'etale topos.
All statements in this section can be deduced in the same manner as in \cite{Ols} or \cite{LO1},
so we do not give detailed proofs.

Let $\mathbf{AnSp}$ denotes the site of complex analytic spaces equipped with the analytic topology.
A stack in groupoid $\mathcal{X}$ over $\mathbf{AnSp}$ is called \emph{complex analytic stack}
if the following conditions hold:
\begin{enumerate}
  \item[(i)]
 The diagonal morphism $\mathcal{X} \to \mathcal{X} \times \mathcal{X}$ is representable by complex analytic spaces.

 \item[(ii)]
 There exists a smooth surjection $U \to \mathcal{X}$ from a complex analytic space $U$.
\end{enumerate}

\begin{dfn}
  Let $\mathcal{X}$ be a complex analytic stack.
  The \emph{lisse-analytic site} $\lisan (\mathcal{X})$ is the site defined as follows:
  \begin{itemize}
    \item The underlying category of $\lisan (\mathcal{X})$ is the full subcategory of complex analytic spaces over $\mathcal{X}$ spanned by ones smooth over $\mathcal{X}$.

    \item A family of morphisms $\{(U_i \to \mathcal{X}) \to (U \to \mathcal{X})\}_{i \in I}$ is a covering if
    $\{U_i \to U\}_{i \in I}$ is an open covering.
  \end{itemize}
  The topos $\mathcal{X}_{\mathrm{lis-an}}$ associated with $\lisan (\mathcal{X})$ is called the \emph{lisse-analytic topos}
  of $\mathcal{X}$.
\end{dfn}

It can be easily seen that a sheaf $F \in \mathcal{X}_{\mathrm{lis-an}}$ is given by the following data:
\begin{itemize}
  \item  a sheaf $F_{(U, u)}$ on $U$ for each $(u \colon U \to \mathcal{X}) \in \lisan (\mathcal{X})$ and
  \item   a morphism $c_f \colon f^{-1} F_{(V, v)} \to F_{(U, u)}$ for each $f \colon (u \colon U \to \mathcal{X}) \to (v \colon V \to \mathcal{X})$ in $\lisan (\mathcal{X})$
\end{itemize}
such that the following conditions hold:
\begin{enumerate}
  \item $c_f$ is an isomorphism if $f$ is an open immersion, and
  \item if we are given a composition
  \[ (u \colon U \to \mathcal{X}) \xrightarrow{f} (v \colon V \to \mathcal{X}) \xrightarrow{g} (w \colon W \to \mathcal{X}),
  \]
  we have $c_{g \circ f} = c_f \circ f^{-1}c_g$.
\end{enumerate}

Denote by $\mathrm{Mod}(\mathcal{X}_{\mathrm{lis-an}}, \mathbb{Q})$  the category of sheaves of $\mathbb{Q}$-vector spaces
over $\mathcal{X}$ and by $D(\mathcal{X}_{\mathrm{lis-an}}, \mathbb{Q})$ the derived category of $\mathrm{Mod}(\mathcal{X}_{\mathrm{lis-an}}, \mathbb{Q})$.

\begin{dfn}
A sheaf $F \in \mathrm{Mod}(\mathcal{X}_{\mathrm{lis-an}}, \mathbb{Q})$ is called \emph{Cartesian} if for any morphism
$f \colon (U \to \mathcal{X}) \to (V \to \mathcal{X})$ in $\lisan (\mathcal{X})$, $c_f$
is an isomorphism. A Cartesian sheaf $F \in \mathrm{Mod}(\mathcal{X}_{\mathrm{lis-an}}, \mathbb{Q})$ is called
(analytically) \emph{constructible} if for any $U \to \mathcal{X}$ in $\lisan (\mathcal{X})$
the restriction $F |_{U_{\mathrm{an}}}$ to the analytic topos of $U$ is (analytically) constructible.
\end{dfn}

Denote by $D_{\mathrm{cart}}(\mathcal{X}_{\mathrm{lis-an}}, \mathbb{Q})$
(resp. $D_c(\mathcal{X}_{\mathrm{lis-an}}, \mathbb{Q})$) the full subcategory of
$D(\mathcal{X}_{\mathrm{lis-an}}, \mathbb{Q})$ spanned by complexes whose cohomologies are
Cartesian sheaves (resp. constructible sheaves).

For an Artin stack $\mathfrak{X}$, one can define its associated complex analytic stack
 $\mathfrak{X}^{\mathrm{an}}$ as in \cite[3.2.2]{Sun}.
 By abuse of notation,
 we write $\lisan (\mathfrak{X})$ (resp. $\mathfrak{X}_{\mathrm{lis-an}}$) instead of $\lisan (\mathfrak{X}^{\mathrm{an}})$ (resp. ${\mathfrak{X}}^{\mathrm{an}}_{\mathrm{lis-an}}$).
 For $* \in \{b, +, -\}$, $D^{(*)}(\mathfrak{X}_{\mathrm{lis-an}}, \mathbb{Q})$ denotes the full subcategory of
 $D(\mathfrak{X}_{\mathrm{lis-an}}, \mathbb{Q})$ consists of complexes $K$ such that
 $K|_\mathfrak{U} \in D^{*}(\mathfrak{U}_{\mathrm{lis-an}}, \mathbb{Q})$  for any quasi-compact Zariski open subset
 $\mathfrak{U} \subset \mathfrak{X}$.
 Define $D^{(*)}_{\mathrm{cart}}(\mathfrak{X}_{\mathrm{lis-an}}, \mathbb{Q})$ and $D^{(*)}_c(\mathfrak{X}_{\mathrm{lis-an}}, \mathbb{Q})$ in a similar manner.

Arguing as in \cite{LO1}, if we are given a morphism $f \colon \mfrakX \to \mfrakY$ of finite type between Artin stacks,
one can construct six functors:
\begin{gather*}
 Rf_* \colon D^{(+)}_{c}(\mathfrak{X}_{\mathrm{lis-an}}, \mathbb{Q}) \to D^{(+)}_{c}(\mathfrak{Y}_{\mathrm{lis-an}}, \mathbb{Q}),
 f^* \colon D_{c}(\mathfrak{Y}_{\mathrm{lis-an}}, \mathbb{Q}) \to D_{c}(\mathfrak{X}_{\mathrm{lis-an}}, \mathbb{Q}),\\
 Rf_! \colon D^{(-)}_{c}(\mathfrak{X}_{\mathrm{lis-an}}, \mathbb{Q}) \to D^{(-)}_{c}(\mathfrak{Y}_{\mathrm{lis-an}}, \mathbb{Q}),
 f^! \colon D_{c}(\mathfrak{Y}_{\mathrm{lis-an}}, \mathbb{Q}) \to D_{c}(\mathfrak{X}_{\mathrm{lis-an}}, \mathbb{Q}),\\
 (-) \otimes (-) \colon  D^{(-)}_{c}(\mathfrak{X}_{\mathrm{lis-an}}, \mathbb{Q}) \times
 D^{(-)}_{c}(\mathfrak{X}_{\mathrm{lis-an}}, \mathbb{Q}) \to D^{(-)}_{c}(\mathfrak{X}_{\mathrm{lis-an}}, \mathbb{Q})\,\,\, \text{and} \\
\sRHom \colon D^{(-)}_{c}(\mathfrak{X}_{\mathrm{lis-an}}, \mathbb{Q})^{\op}
\times D^{(+)}_{c}(\mathfrak{X}_{\mathrm{lis-an}}, \mathbb{Q})
\to D^{(+)}_{c}(\mathfrak{X}_{\mathrm{lis-an}}, \mathbb{Q}).
\end{gather*}
We briefly recall the construction of $Rf_*$, $f^*$, $Rf_!$ and $f^!$.
Firstly define
\[f_* \colon \mathrm{Mod}(\mathfrak{X}_{\mathrm{lis-an}}, \setQ) \to \mathrm{Mod}(\mathfrak{Y}_{\mathrm{lis-an}}, \setQ)
\]
by the rule that $f_* F(U) = F(U \times_{\mfrakY} \mfrakX)$.
By taking the derived functor of $f_*$ we obtain $Rf_*$.
When $f$ is a smooth morphism, $f^*$ is nothing but the restriction functor.
In general $f^*$ is constructed by taking simplicial covers, but we use pullback functors only for
smooth morphisms in this paper, so we do not need this general construction.
To define $Rf_!$ and $f^!$ we use the Verdier duality functor.
Arguing as in \cite[\textsection 3]{LO1}, we can construct the dualizing complex
\[
\omega_{\mfrakX} \in D^{(b)}_{c}(\mathfrak{X}_{\mathrm{lis-an}}, \mathbb{Q})
\]
and define the Verdier duality functor
\[
\mathbb{D}_{\mfrakX} \coloneqq \sRHom(-, \omega_{\mfrakX}) \colon
D^{(-)}_{c}(\mathfrak{X}_{\mathrm{lis-an}}, \mathbb{Q})^{\op} \to D^{(+)}_{c}(\mathfrak{X}_{\mathrm{lis-an}}, \mathbb{Q}).
\]
Now define $Rf_! \coloneqq \mathbb{D}_{\mfrakY} \circ f_* \circ \mathbb{D}_{\mfrakX}$ and
$f^! \coloneqq \mathbb{D}_{\mfrakX} \circ f^* \circ \mathbb{D}_{\mfrakY}$.
If $f$ is a smooth morphism of relative dimension $d$, we have natural isomorphisms
\begin{align}\label{eq:smup}
  \begin{split}
f^*\mathbb{D}_{\mathfrak{X}}(F) = f^*\sRHom(F, \omega_{\mathfrak{Y}}) \xrightarrow{\sim}
\sRHom(f^*F, f^*\omega_{\mathfrak{Y}}) \\
 \xrightarrow{\sim} \sRHom(f^*F, \omega_{\mathfrak{X}}[-2d]) \cong
\mathbb{D}_{\mathfrak{Y}}(f^*F)[-2d]
\end{split}
\end{align}
for $F \in D_c(\mathfrak{Y}, \mathbb{Q})$.
Therefore we have $f^! \cong f^*[2d]$.

If we are given a $2$-morphism $\xi \colon f \Rightarrow g$ between morphisms of finite type of Artin stacks,
we have a natural isomorphism $\xi_* \colon Rf_* \Rightarrow Rg_*$ compatible with the vertical and horizontal compositions.
The same statement also holds for $Rf_!$, $f^*$ and $f^!$.

Now we discuss the base change isomorphisms.
Consider the following $2$-Cartesian diagram of Artin stacks:
\[
  \xymatrix@R=4pt@C=8pt{
     \pbcorner \mathfrak{X}' \ar[rrr]^-{g'}\ar[ddd]_-{f'} & & & \mathfrak{X} \ar[ddd]^-{f} \\
     & & \ar@{=>}[ld]^{\eta} & \\
     & & & \\
     \mathfrak{Y}' \ar[rrr]^-{g} & & & \mathfrak{Y}.
    }
\]
By adjunction, we have the base change map
\begin{align}\label{eq:bamap}
    \mathfrak{bc}_{\eta} \colon g^* Rf_* \to Rf'_*g'^*
\end{align}
which is isomorphic if $g$ is smooth.
Now assume that $f$ is of finite type and $g$ is smooth with relative dimension $d$, and take $F \in D^{(-)}_c(\mathfrak{X}, \mathbb{Q})$.
Then we can construct the proper base change map
\[
\pbc_{\eta} \colon g^* Rf_! F \xrightarrow{\sim} Rf'_!g'^* F
\] by the following composition
\begin{align*}
  g^* Rf_! F &=  g^* \mathbb{D}_{\mathfrak{Y}} Rf_* \mathbb{D}_{\mathfrak{X}} F \\
             &\cong   \mathbb{D}_{\mathfrak{Y}'} g^* Rf_* \mathbb{D}_{\mathfrak{X}} F [2d] \\
             &\cong   \mathbb{D}_{\mathfrak{Y}'} Rf'_* g'^* \mathbb{D}_{\mathfrak{X}} F [2d] \\
             &\cong   \mathbb{D}_{\mathfrak{Y}'} Rf'_* \mathbb{D}_{\mathfrak{X}'} g'^* F = Rf'_! g'^* F
\end{align*}
where the first and third isomorphism is defined by using \eqref{eq:smup} and the second isomorphism is the base change map \eqref{eq:bamap}.
Now consider the following composition of $2$-Cartesian diagrams
\[
  \xymatrix@R=2pt@C=8pt{
  &&&\ar@{=>}[dd]_{\eta_1}&&&\\
  &&&&&&\\
  \pbcorner \mathfrak{X}'' \ar[rrr]^{h'} \ar[ddd]_{f''} \ar@/^23pt/[rrrrrr]^{k'}  & & & \pbcorner \mathfrak{X}' \ar[rrr]^-{g'}\ar[ddd]_-{f'} &&& \mathfrak{X} \ar[ddd]_-{f} \\
  &&\ar@{=>}[ld]^(0.7){\eta_2} &&&\ar@{=>}[ld]^(0.7){\eta_3}&\\
  &&&&&&\\
  \mathfrak{Y}''  \ar@/_25pt/[rrrrrr]_{k} \ar[rrr]_{h} &&& \ar@{=>}[dd]_{\eta_4} \mathfrak{Y}' \ar[rrr]_-{g} &&& \mathfrak{Y}\\
  &&&&&&\\
  &&&&&&
    }
\]
where $f$ is of finite type, and $g$ and $h$ are smooth.
We define $\eta \colon f \circ k' \Rightarrow k \circ f''$ by composing $2$-morphisms in the diagram.
Arguing as \cite[Lemma 5.1.2]{LO1}, we can show the commutativity of the following diagram:
\begin{align}\label{eq:pbcass}
  \begin{split}
    \xymatrix{
    k^* f_! \ar[rr]^{\pbc_\eta} \ar[d]_-{\eta_4 ^*} & & (f'')_! (k')^* \ar[d]^-{{\eta_1} _*} \\
    h^*g^*f_! \ar[r]^-{h^* \pbc_{\eta_2}} & h^* (f')_! (g')^* \ar[r]^(.5){\pbc_{\eta_3}(g')^*} &  (f'')_! (h')^* (g')^*.
    }
     \end{split}
\end{align}

For an Artin stack $\mfrakX$, we define a full subcategory $\mathrm{Perv}(\mfrakX) \hookrightarrow D_c(\mfrakX, \setQ)$ consists of objects $K$ such that for any smooth morphism $f \colon U \to \mfrakX$ from a scheme,
$f^* K [\dim f]$ is a perverse sheaf on $U$. An object in $\mathrm{Perv}(\mfrakX)$ is called a perverse sheaf on $\mfrakX$.
Arguing as \cite[Proposition 7.1]{LO2},
we see that $U \mapsto \mathrm{Perv}(U)$
defines a stack on $\lisan(\mfrakX)$ whose global section category is $\mathrm{Perv}(\mfrakX)$.

\subsection{D-critical stacks}

In this section we first recall the notion of d-critical stacks introduced in \cite[\S 2.8]{Joy}, which is a stacky generalization of d-critical schemes.
And then we prove that the canonical d-critical structures and the canonical orientations for $-1$-shifted cotangent stacks are preserved under smooth base changes.
This is the key ingredient in the proof of the dimensional reduction theorem for quasi-smooth derived Artin stacks.

We first explain the functorial behavior of the d-critical structure:
\begin{prop}\cite[Proposition 2.3, 2.8]{Joy}
  Let $f \colon X \to Y$ be a morphism of complex analytic spaces, and
  $\mathcal{S}_X$ (resp. $\mathcal{S}_Y$) be the sheaf on $X$ (resp. $Y$) defined in \eqref{eq:shS}.
  Then there exists a natural map $\theta_f \colon f^{-1}\mathcal{S}_Y \to \mathcal{S}_X$ with the following property:
  If $R \subset X$ and $S \subset Y$ are open subsets with $f(R) \subset S$, $i \colon R \hookrightarrow U$ and
  $j \colon S \hookrightarrow V$ are closed embeddings into complex manifolds, and $\tilde{f} \colon U \to V$ is a holomorphic map with $j \circ f |_R = \tilde{f} \circ i$, then the following diagram commutes:
  \[
    \xymatrix{
      f^{-1}\mathcal{S}_Y|_R \ar[r]^{} \ar[d]_-{\theta_f |_R} & (f|_R)^{-1}((j^{-1}\mathcal{O}_V)/ I_{S, V}^2) \ar[d]\\
       \mathcal{S}_X |_R \ar[r] & (i^{-1}\mathcal{O}_U)/I_{R, U}^2.
      }
  \]
  Here horizontal maps are induced by the natural inclusions, and right vertical map is induced by
  $\tilde{f}^\sharp \colon \tilde{f}^{-1}\mathcal{O}_V \to \mathcal{O}_U$.
  The map $\theta_f$ induces natural map $f^{-1}\mathcal{S}_Y ^0 \to \mathcal{S}_X ^0$ (also written as $\theta_f$).
  If $f$ is smooth and $s \in \Gamma(Y, \mathcal{S}_Y^ 0)$ is a d-critical structure,
  $f^\star s \coloneqq \theta_f(f^{-1}s)$ is also a d-critical structure.
\end{prop}

Now we explain that the definition of the sheaf  $\mathcal{S}_X ^0$ can be extended to complex analytic stacks:

\begin{prop}\cite[Corollary 2.52]{Joy}
  Let $\mathcal{X}$ be a complex analytic stack.
  Then there exists a sheaf of complex vector spaces $\mathcal{S}_\mathcal{X} ^0$ in $\text{\normalfont{Lis-an}}(\mathcal{X})$
  with the following properties.
  \begin{itemize}
    \item For $(u \colon U \to \mathcal{X}) \in \text{\normalfont{Lis-an}}(\mathcal{X})$,
          we have an isomorphism
          \[
          \theta_u \colon \mathcal{S}_{\mathcal{X}}^0 |_{U_{\mathrm{an}}} \cong \mathcal{S}_U ^0.
          \]
    \item For a morphism $f \colon (u \colon U \to \mathcal{X}) \to (v \colon V \to \mathcal{X})$ in          ${\text{\normalfont{Lis-an}}}(\mathcal{X})$, the following diagram commutes
          \[
            \xymatrix@C=50pt{
            f^{-1}(\mathcal{S}_\mathcal{X}^0 |_{V_\mathrm{an}}) \ar[r] ^-{f^{-1}(\theta_v)} \ar[d]_-{c_f} & f^{-1}\mathcal{S}_{V} ^0 \ar[d]^-{\theta_f} \\
            \mathcal{S}_{\mathcal{X}}^0 |_{U_\mathrm{an}} \ar[r]^-{\theta_u} & \mathcal{S}_U ^0.
          }
          \]
  \end{itemize}
\end{prop}

\begin{dfn}
  Let $\mathcal{X}$ be a complex analytic stack.
  A section $s \in \Gamma(\mathcal{X_{\text{lis-an}}}, \mathcal{S}_\mathcal{X}^0)$ is called a \emph{d-critical structure}
  if for any $(u\colon U \to \mathcal{X}) \in \lisan(\mathcal{X})$, $u^\star s \coloneqq \theta_u(s |_{U_\mathrm{an}})$ is a d-critical structure on $U$.
  A \emph{d-critical Artin stack} is an Artin stack $\mathfrak{X}$ with a d-critical structure on its analytification.
\end{dfn}

We have a stacky version of Theorem \ref{thm:shiftdcrit}:

\begin{thm}\cite[Theorem 3.18(a)]{BBBBJ}\label{thm:dcritstack}
  Let $(\bs{\mathfrak{X}}, \omega_{\bs{\mathfrak{X}}})$ be a $-1$-shifted symplectic derived Artin stack.
  Then there exists a natural d-critical structure $s_{\mathfrak{X}}$ on $\mathfrak{X} \coloneqq t_0(\bs{\mathfrak{X}})$ uniquely characterized by
  the following property:

  Assume we are given derived schemes $\bs{X}$ and $\widehat{\bs{X}}$,
  morphisms $\bs{g} \colon \bs{X} \to \bs{\mathfrak{X}}$
  and $\bs{\tau} \colon \bs{X} \to \widehat{\bs{X}}$ such that $\bs{g}$ is smooth.
  Further assume that there exists a $-1$-shifted symplectic structure $\omega_{\widehat{\bs{X}}}$
  and an equivalence $\bs{g}^\star \omega_{\bs{\mathfrak{X}}} \sim \bs{\tau}^\star \omega_{\widehat{\bs{X}}}$.
  If we write $g = t_0(\bs{g})$, $\tau = t_0 (\bs{\tau})$, and $s_{\widehat{X}}$ the
  d-critical structure on $\widehat{X} = t_0(\widehat{\bs{X}})$ associated with $\omega_{\widehat{\bs{X}}}$,
  we have $g^\star s_{\mathfrak{X}} = \tau^\star s_{\widehat{X}}$.
\end{thm}

\begin{proof}
  The uniqueness part is proved in \cite[Theorem 3.18(a)]{BBBBJ}.
  We now verify that the d-critical structure constructed in loc. cit. satisfies the property as above.
  Using \cite[Theorem 2.10]{BBBBJ}, we have derived schemes $\bs{U}$ and $\widehat{\bs{U}}$,
   a smooth surjection $\bs{u} \colon \bs{U} \to \bs{\mathfrak{X}}$, a morphism $\bs{\tau}_{\bs{U}} \colon \bs{U} \to \widehat{\bs{U}}$,
   and a $-1$-shifted symplectic form $\omega_{\widehat{\bs{U}}}$ on $\widehat{\bs{U}}$ such that
   $\bs{u}^\star \omega_{\bs{X}} \sim \bs{\tau}_{\bs{U}}^\star \omega_{\widehat{\bs{U}}}$.
   Further, if we write $s_{\widehat{U}}$ the d-critical structure associated with $\omega_{\widehat{\bs{U}}}$,
   we may assume $t_0(\bs{u})^\star s_{\mathfrak{X}} = t_0(\bs{\tau}_{\bs{U}})^\star s_{\widehat{U}}$.
   We have the following diagram of derived stacks:
   \[
   \xymatrix{
  & \pbcorner \bs{U} \times_{\bs{\mathfrak{X}}} \bs{X} \ar[r]^-{\bs{u}'}\ar[d]_-{\bs{g}'} & \bs{X} \ar[d]^-{\bs{g}} \ar[r]^-{\bs{\tau}} & \widehat{\bs{X}} \\
  \widehat{\bs{U}}  & \bs{U} \ar[l]_-{\bs{\tau}_{\bs{U}}} \ar[r]^-{\bs{u}} & \bs{\mfrakX}. &
   }
   \]
   Now take any point $x \in t_0(\bs{X})$ and an \'etale morphism from a derived scheme
   $\bs{\eta} \colon \bs{W} \to \bs{U} \times_{\bs{\mfrakX}} \bs{X}$
   such that the image of $t_0(\bs{u}' \circ \bs{\eta})$ contains $x$.
   Since $t_0(\bs{u}' \circ \bs{\eta})$ is a smooth morphism, it suffices to show that
   \[
   t_0(\bs{\tau}_{\bs{U}} \circ \bs{g}' \circ \bs{\eta})^\star s_{\widehat{U}} = t_0(\bs{\tau} \circ \bs{u}' \circ \bs{\eta})^\star s_{\widehat{X}}.
   \]
   This follows by arguing as \cite[Example 5.22]{BBJ} since we have
   $(\bs{\tau}_{\bs{U}} \circ \bs{g}' \circ \bs{\eta})^\star \omega_{\widehat{\bs{U}}} \sim (\bs{\tau} \circ \bs{u}' \circ \bs{\eta}) ^\star \omega_{\widehat{\bs{X}}}$.
\end{proof}

Now we discuss the behavior of the d-critical structure associated with the canonical $-1$-shifted symplectic structures
on $-1$-shifted cotangent stacks constructed in Example \ref{ex:cot} under smooth pullbacks.
Let $\bs{f}\colon \bs{Y} \to \bs{\mfrakY}$ be a smooth morphism from a derived scheme $\bs{Y}$ to a quasi-smooth derived Artin stack $\bs{\mfrakY}$.
Consider the following diagram:
\begin{align}\label{eq:maindiag}
  \begin{split}
  \xymatrix@C=40pt{
\bs{f}^*\bfT^*[-1]\bs{\mfrakY} \ar[r]^-{\bs{\tau}} \ar[rd]_-{\widetilde{\bs{f}}} \ar@/^20pt/[rr]^-{\pi_{\bs{\mfrakY}, \bs{f}}}
& \bfT^*[-1]\bs{Y} \ar[r]^-{\pi_{\bs{Y}}} & \bs{Y} \ar[d]^{\bs{f}} \\
  & \bfT^*[-1]\bs{\mfrakY} \ar[r]^-{\pi_{\bs{\mfrakY}}} & \bs{\mfrakY}.
  }
\end{split}
\end{align}
Here $\bs{f}^*\bfT^*[-1]\bs{\mfrakY}$ is the total space $\Tot_{\bs{Y}}(\mathbb{L}_{\bs{\bs{f}^* \mathfrak{Y}}}[-1])$,
$\pi_{\bs{Y}}$, $\pi_{\bs{\mfrakY}}$, and $\pi_{\bs{\mfrakY}, \bs{f}}$ are the projections,
$\bs{\tau}$ is induced by the canonical map $\bs{f}^* \mathbb{L}_{\bs{\mathfrak{Y}}}[-1] \to \mathbb{L}_{\bs{Y}}[-1]$,
and $\widetilde{\bs{f}} \colon \bs{f}^*\bfT^*[-1]\bs{\mfrakY} \to \bfT^*[-1]\bs{\mfrakY}$ is the base change of $\bs{f}$.
The smoothness of $\bs{f}$ implies that $\bs{\tau}$ induces an isomorphism on underlying schemes,
so we use the identification
\begin{align}\label{eq:ident}
t_0(\bs{f}^*\bfT^*[-1]\bs{\mfrakY}) = t_0(\bfT^*[-1]\bs{Y})
\end{align}
throughout the paper.

\begin{prop}\label{prop:dcritcot}
  Consider the situation as above.
  Denote by $s_{\bfT^*[-1]\bs{Y}}$ (resp. $s_{\bfT^*[-1]\bs{\mfrakY}}$) the d-critical structure
  associated with the canonical $-1$-shifted symplectic form $\omega_{\bfT^*[-1]\bs{Y}}$ (resp. $\omega_{\bfT^*[-1]\bs{\mfrakY}}$) constructed in Example \ref{ex:cot}.
  Then we have $ s_{\bfT^*[-1]\bs{Y}}  = \tilde{f}^\star s_{\bfT^*[-1]\bs{\mfrakY}}$
  where we write $\tilde{f} = t_0(\widetilde{\bs{f}})$ and use the identification \eqref{eq:ident}.
\end{prop}

\begin{proof}
By Theorem \ref{thm:dcritstack}, we only need to show that
$\bs{\tau}^\star \omega_{\bfT^*[-1]\bs{Y}} \sim \widetilde{\bs{f}}^\star \omega_{\bfT^*[-1]\bs{\mfrakY}}$.
If we write $\lambda_{\bfT^*[-1]\bs{Y}}$ and $\lambda_{\bfT^*[-1]\bs{\mfrakY}}$ the tautological $1$-forms on $\bfT^*[-1]\bs{Y}$ and $\bfT^*[-1]\bs{\mfrakY}$ respectively,
we have $\ddrc \lambda_{\bfT^*[-1]\bs{Y}} = \omega_{\bfT^*[-1]\bs{Y}}$ and
$\ddr \lambda_{\bfT^*[-1]\bs{\mfrakY}} = \omega_{\bfT^*[-1]\bs{\mfrakY}}$ by definition.
Therefore we only need to prove
\[
\bs{\tau}^\star \lambda_{\bfT^*[-1]\bs{Y}} \sim \widetilde{\bs{f}}^\star \lambda_{\bfT^*[-1]\bs{\mfrakY}}.
\]
By the functoriality of the cotangent complex, we have the following homotopy commutative diagram:
\[
\xymatrix@C=40pt{
  (\bs{\pi}_{\bs{\mfrakY}, \bs{f}})^* \bs{f}^* \L_{\bs{\mathfrak{Y}}}[-1] \ar[r]_-{\sim} ^-{a}
  \ar[dd]_-{(\bs{\pi}_{\bs{\mfrakY}, \bs{f}})^* \theta_{\bs{f}}}
  &\widetilde{\bs{f}}^* \bs{\pi}_{\bs{\mfrakY}}^* \L_{\bs{\mathfrak{Y}}}[-1]
  \ar[r]^-{\widetilde{\bs{f}}^* \theta_{\bs{\pi}_{\bs{\mfrakY}}}}
  &\widetilde{\bs{f}}^* \L_{\bfT^*[-1]\bs{\mfrakY}}[-1]
  \ar[d]^-{\theta_{\widetilde{\bs{f}}^*}} \\
   & & \L_{\bs{f}^*\bfT^*[-1]\bs{\mfrakY}}[-1] \\
 (\bs{\pi}_{\bs{\mfrakY}, \bs{f}})^* \L_{\bs{Y}}[-1] \ar[r]_-{\sim} ^-{b}
 &\bs{\tau}^* \bs{\pi}_{\bs{Y}}^* \L_{\bs{Y}}[-1]
 \ar[r]^-{\bs{\tau}^* \theta_{\bs{\pi}_{\bs{Y}}}}
 &\bs{\tau}^* \L_{\bfT^*[-1]\bs{Y}}[-1] \ar[u]_-{\theta_{\bs{\tau}}}.
}
\]
Here $a$ and $b$ are defined by using
$\bs{f} \circ \bs{\pi}_{\bs{\mfrakY}, \bs{f}} \simeq
\bs{\pi}_{\bs{\mfrakY}} \circ \widetilde{\bs{f}}$ and
$\bs{\pi}_{\bs{\mfrakY}, \bs{f}} \simeq
\bs{\pi}_{\bs{Y}} \circ \bs{\tau}$ respectively, and other morphisms are induced by the functoriality of the cotangent complex.
Now write $\gamma_{\bs{f}^*\bfT^*[-1]\bs{\mfrakY}}$, $\gamma_{\bfT^*[-1]\bs{Y}}$ and $\gamma_{\bfT^*[-1]\bs{\mfrakY}}$ the tautological sections of
$(\bs{\pi}_{\bs{\mfrakY}, \bs{f}})^* \bs{f}^* \L_{\bs{\mathfrak{Y}}}[-1]$,
$\bs{\pi}_{\bs{\mfrakY}}^* \L_{\bs{\mathfrak{Y}}}[-1]$ and
$\bs{\pi}_{\bs{Y}}^* \L_{\bs{Y}}[-1]$ respectively.
By definition, we have
\begin{align*}
  \widetilde{\bs{f}}^\star \lambda_{\bfT^*[-1]\bs{\mfrakY}} \sim
   \theta_{\widetilde{\bs{f}}}(\widetilde{\bs{f}}^* \lambda_{\bfT^*[-1]\bs{\mfrakY}}) \sim
   \theta_{\widetilde{\bs{f}}} \circ \widetilde{\bs{f}}^* \theta_{\bs{\pi}_{\bs{\mfrakY}}}(\widetilde{\bs{f}}^* \gamma_{\bfT^*[-1]\bs{\mfrakY}}),
   \\
  \bs{\tau}^\star \lambda_{\bfT^*[-1]\bs{Y}} \sim
  \theta_{\bs{\tau}}(\bs{\tau}^* \lambda_{\bfT^*[-1]\bs{Y}}) \sim
  \theta_{\bs{\tau}} \circ \bs{\tau}^* \theta_{\bs{\pi}_{\bs{Y}}}
  (\bs{\tau}^*  \gamma_{\bfT^*[-1]\bs{Y}}).
\end{align*}
Since we have the following tautological relations
\begin{align*}
  \widetilde{\bs{f}}^* \gamma_{\bfT^*[-1]\bs{\mfrakY}} \sim
  a( \gamma_{\bs{f}^*\bfT^*[-1]\bs{\mfrakY}}), \,
  \bs{\tau}^* \gamma_{\bfT^*[-1]\bs{Y}} \sim
  b \circ  (\bs{\pi}_{\bs{\mfrakY}, \bs{f}})^* \theta_{\bs{f}}
  (\gamma_{\bs{f}^*\bfT^*[-1]\bs{\mfrakY}}),
\end{align*}
the proposition follows.
\end{proof}

We now discuss the notion of the virtual canonical bundles and the orientations for d-critical stacks.
Before doing this we recall a property of the virtual canonical bundle of d-critical schemes.
For a d-critical chart $(R, U, f, i)$ of a d-critical scheme $(X, s)$ and a point $x \in R$,
consider the following complex concentrated in degree $-1$ and $0$:
\[
L_x \coloneqq (\mathrm{T}_U |_x \xrightarrow{\mathrm{Hess}(f) |_x} \Omega_U |_x).
\]
Since $\mH^0(L_x) \cong \Omega_X |_x$ and $\mH^{-1}(L_x) \cong (\Omega_X |_x)^\vee$,
we can define an isomorphism
\begin{align}
  \kappa_x \colon K_{X, s} |_x \cong \det(L_x) \cong \det(\Omega_X|_x) \otimes \det((\Omega_X|_x)^\vee)^{-1} \cong  \det(\Omega_X|_x)^{\otimes^2}.
\end{align}
Here the final isomorphism is defined in the same manner as \eqref{eq:dualconv}.
It is proved in \cite[Theorem 2.28]{Joy} that $\kappa_x$ does not depend on the choice of a d-critical chart.
Now the virtual canonical bundle for a d-critical stack is defined by the following proposition:
\begin{prop}\cite[Theorem 2.56]{Joy}\label{prop:dcritstkcan}
  Let $(\mathfrak{X}, s)$ be a d-critical stack.
  Then there exists a line bundle $K_{\mathfrak{X}, s}$ on $\mathfrak{X}^\red$, which we call the virtual canonical bundle of $(\mathfrak{X}, s)$, characterized uniquely up to unique isomorphism by the following properties:
  \begin{itemize}
    \item[(i)] For $x\in \mathfrak{X}$, there exists an isomorphism
    \begin{align}
      \kappa_x \colon K_{\mathfrak{X}, s} |_x \cong \det(\tau^{\geq 0}(\L_{\mathfrak{X}})|_x)^{\otimes ^2} .
    \end{align}
    \item[(ii)]\label{prop:dcritstkcan2}
      For a smooth morphism $u \colon U \to \mathfrak{X}$ from a scheme $U$,
      there exists an isomorphism
      \begin{align}\label{eq:dcritstkcan2}
        \Gamma_{U, u} \colon (u^\red)^* K_{\mathfrak{X}, s} \cong K_{U, u^\star s} \otimes \detr(\Omega_{U/\mathfrak{X}}) ^{\otimes ^{-2}}.
      \end{align}
      \item[(iii)]
      In the situation of (ii), take any $p \in U$.
      The following distinguished triangle
      \[
      \Delta \colon u^* \tau^{\geq 0} (\L_{\mathfrak{X}}) \to \Omega_U \to \Omega_{U/\mathfrak{X}} \to u^* \tau^{\geq 0} (\L_{\mathfrak{X}})[1]
      \]
      induces an isomorphism
      $\hi(\Delta)_p \colon \det(\tau^{\geq 0} (\L_{\mathfrak{X}})|_{u(p)}) \otimes \det(\Omega_{U/\mathfrak{X}}|_p) \cong \det(\Omega_U |_p)$
      where $\hi$ is defined in Lemma \ref{lem:det}.
      Then the following diagram commutes:
      \begin{align*}
        \xymatrix{
        K_{\mathfrak{X}, s} |_{u(p)} \ar[r]^-{\Gamma_{U, u} |_p} \ar[d]_-{\kappa_{u(p)}}
        &  K_{U, u^\star s} |_p  \otimes \det(\Omega_{U/\mathfrak{X}}|_p)   ^{\otimes ^{-2}} \ar[d]^-{\kappa_p \otimes \id}\\
        \det(\tau^{\geq 0}(\L_{\mathfrak{X}})|_{u(p)})^{\otimes ^2} \ar[r]
        & \det(\Omega_U |_p)^{\otimes ^2}  \otimes \det(\Omega_{U/\mathfrak{X}}|_p)   ^{\otimes ^{-2}}
        }
      \end{align*}
      Here the bottom horizontal map is defined by using $\hi(\Delta)_p$.
  \end{itemize}
\end{prop}

An \emph{orientation} $o$ of a d-critical stack $(\mathfrak{X}, s)$ is the choice of a line bundle $L$ on $\mathfrak{X}^\red$ and an isomorphism $o \colon L^{\otimes^2} \cong K_{\mathfrak{X}, s}$.
An isomorphism between orientations $o_1\colon L_1^{\otimes^2} \cong K_{\mathfrak{X}, s}$ and $o_2 \colon L_2^{\otimes^2} \cong K_{\mathfrak{X}, s}$
is defined by an isomorphism $\phi \colon L_1 \cong L_2$ such that $o_1 = o_2 \circ {\phi}^{\otimes ^2}$.
If there exists a smooth morphism $u \colon U \to \mfrakX$,
we define an orientation $u^\star o$ for $(U, u^\star s)$ by the following composition:
\[
u^\star o \colon ((u^\red)^*L \otimes \detr(\Omega_{U/X}))^{\otimes ^2} \cong (u^\red)^* K_{\mathfrak{X}, s} \otimes \detr(\Omega_{U/X})^{\otimes ^2} \xrightarrow{\Gamma_{U, u}} K_{U ,u^\star s}.
\]
If we are given a smooth morphism $q \colon (u \colon U \to \mfrakX) \to (v \colon V \to \mfrakX)$ in $\lisan(\mfrakX)$,
define an isomorphism
\begin{align}\label{eq:oripullcomp}
  u^\star o \cong q^\star v^\star o
\end{align} by using the natural isomorphism
\[
\detr(\Omega_{U/\mathfrak{X}}) \cong (f^\red)^*\detr(\Omega_{V/\mathfrak{X}}) \otimes \detr(\Omega_{U/V}).
\]

We now discuss the relation of the cotangent complex of a $-1$-shifted symplectic derived Artin stack and the virtual canonical bundle of the associated d-critical stack:

\begin{thm}\cite[Theorem 3.18(b)]{BBBBJ}\label{thm:cotstk}
  Let $(\bs{\mathfrak{X}}, \omega_{\bs{\mathfrak{X}}})$ be a $-1$-shifted symplectic derived Artin stack, and $(\mathfrak{X}, s_{\mfrakX})$ the associated d-critical stack.
  Then there exists a natural isomorphism
  \begin{align}
    \Lambda_{\bs{\mfrakX}} \colon \detr(\L_{\bs{\mathfrak{X}}})  \cong K_{\mathfrak{X}, s_{\mfrakX}}
  \end{align}
  characterized by the following property:

  Assume we are given derived schemes $\bs{X}$ and $\widehat{\bs{X}}$,
  morphisms $\bs{g} \colon \bs{X} \to \bs{\mathfrak{X}}$ and
  $\bs{\tau} \colon \bs{X} \to \widehat{\bs{X}}$ such that $\bs{g}$ is smooth and $\L_{\bs{\tau}} |_{x}$ is concentrated in degree $-2$ for each $x \in \bs{X}$. Note that it automatically follows that $\tau = t_0(\bs{\tau})$ is \'etale.
  Further assume that there exist a $-1$-shifted symplectic structure $\omega_{\widehat{\bs{X}}}$ on $\widehat{\bs{X}}$ with associated d-critical locus $(\widehat{X}, s_{\widehat{X}})$ and an equivalence
  $\bs{g}^\star \omega_{\bs{\mathfrak{X}}} \sim \bs{\tau}^\star \omega_{\widehat{\bs{X}}}$.
  This equivalence induces a homotopy between the composition
  \[
  \mathbb{T}_{\bs{g}} \to \mathbb{T}_{\bs{X}} \to \bs{\tau}^*\mathbb{T}_{\widehat{\bs{X}}}
   \xrightarrow{\bs{\tau}^\star \omega_{\widehat{\bs{X}}}} \bs{\tau}^*\mathbb{L}_{\widehat{\bs{X}}}[-1] \to \L_{\bs{X}}[-1]
  \]
  and $0$, hence an isomorphism
  \begin{align}\label{eq:ell}
  \ell \colon \mathbb{T}_{\bs{g}} \xrightarrow{\sim} \mathbb{L}_{\bs{\tau}}[-2].
\end{align}
  Then the composition
  \begin{align*}
    \detr(\bs{g}^* \L_{\bs{\mathfrak{X}}})  &\cong \detr(\L_{\bs{X}}) \otimes \detr(\L_{\bs{g}})^{-1} \\
                                    &\cong \detr(\bs{\tau}^* \L_{\widehat{\bs{X}}}) \otimes \detr(\L_{\bs{\tau}}) \otimes \detr(\L_{\bs{g}})^{-1} \\
                                    &\cong \detr(\bs{\tau}^* \L_{\widehat{\bs{X}}}) \otimes \detr(\L_{\bs{g}})^{\otimes ^{-2}}\\
                                    &\cong (\tau^\red)^* K_{\widehat{X}, s_{\widehat{X}}} \otimes \detr(\L_{\bs{g}})^{\otimes ^{-2}} \\
                                    &\cong K_{X,\tau^\star s_{\widehat{X}}} \otimes \detr(\L_{\bs{g}})^{\otimes ^{-2}}\\
                                    &\cong (g^\red)^*K_{\mathfrak{X}, s_{\mfrakX}}
  \end{align*}
  is equal to $(-1)^{\rank(\Omega_{g})} {(g^\red)}^*\Lambda_{\bs{\mfrakX}}$
  where we write $g = t_0(\bs{g})$.
  Here the first and second isomorphisms are defined by using $\hi(\Delta_{\bs{g}})$ and $\hi(\Delta_{\bs{\tau}})$ respectively, where
  $\Delta_{\bs{g}}$ and $\Delta_{\bs{\tau}}$ are distinguished triangles
  \begin{align*}
     \Delta_{\bs{g}} \colon \bs{g}^* \L_{\bs{\mathfrak{X}}} \to \L_{\bs{X}} \to \L_{\bs{g}} \to \bs{g}^* \L_{\bs{\mathfrak{X}}}[1],\,
      \Delta_{\bs{\tau}} \colon \bs{\tau}^* \L_{\widehat{\bs{X}}} \to \L_{\bs{X}} \to \L_{\bs{\tau}} \to \bs{\tau}^* \L_{\widehat{\bs{X}}}[1].
\end{align*}
   The third isomorphism is defined in the same manner as \eqref{eq:dualconv} using $\ell$ (without any sign intervention
   used in Appendix \ref{ap:A}),
  the fourth isomorphism is $\Lambda_{\widehat{\bs{X}}}$ in Theorem \ref{thm:shiftdcrit},
  and the fifth isomorphism is $\Gamma_{X, \tau}$ defined in Proposition \ref{prop:dcritstkcan}. The final isomorphism is
  $\Gamma_{X, g}$, where we use the fact that $\tau^\star s_{\widehat{X}} = g^\star s_{\mfrakX}$ proved in
  Theorem \ref{thm:dcritstack}.
\end{thm}

\begin{proof}
  The proof is essentially same as one in \cite{BBBBJ}, but we include this for reader's convenience and to fix the sign.
  Suppose we are given $\bs{X}$, $\widehat{\bs{X}}$, $\bs{g}$ and $\bs{\tau}$ as above.
  Define
  \[
  \Lambda_{\bs{X}, \widehat{\bs{X}}, \bs{g}, \bs{\tau}} \colon \detr(\bs{g}^* \L_{\bs{\mathfrak{X}}})  \to (g^\red)^*K_{\mathfrak{X}, s_{\mfrakX}}
  \]
  by the composition as above multiplied by $(-1)^{\rank(\Omega_{g})}$.
  Write $\mathrm{pr}_1, \mathrm{pr}_2 \colon R = X \times_{\mathfrak{X}} X \rightrightarrows X$ the first and second projections.
  We have a natural $2$-morphism $\xi \colon g \circ \mathrm{pr}_1 \Rightarrow g \circ \mathrm{pr}_2$.
  Now we prove the commutativity of the following diagram:
  \begin{align}\label{eq:vcandsc}
    \begin{split}
  \xymatrix@C=60pt{
  (\mathrm{pr}_1^\red)^* (\detr(\bs{g}^* \L_{\bs{\mathfrak{X}}}) ) \ar[r]^-{(\mathrm{pr}_1^\red)^* \Lambda_{\bs{X}, \widehat{\bs{X}}, \bs{g}, \bs{\tau}}} \ar[d]_-{\xi_*} & (\mathrm{pr}_1^\red)^* (g^\red)^* K_{\mfrakX, s_{\mfrakX}} \ar[d]^-{\xi_*} \\
  (\mathrm{pr}_2^\red)^* (\detr(\bs{g}^* \L_{\bs{\mathfrak{X}}})) \ar[r]^-{(\mathrm{pr}_2^\red)^* \Lambda_{\bs{X}, \widehat{\bs{X}}, \bs{g}, \bs{\tau}}} & (\mathrm{pr}_2^\red)^* (g^\red)^* K_{\mfrakX, s_{\mfrakX}}.
  }
\end{split}
\end{align}
  By the reducedness of $R^\red$, we only need to prove the commutativity at each point $r \in R$.
  Write $\mathrm{pr}_1(r) = x_1$ and $\mathrm{pr}_2(r) = x_2$. Now consider the following diagram:
  \[
  \xymatrix@C=15pt{
   \detr(\bs{g}^* \L_{\bs{\mfrakX}}) |_{x_1} \ar[rrr]^-{\Lambda_{\bs{X}, \widehat{\bs{X}}, \bs{g}, \bs{\tau}}|_{x_1}} \ar[ddd]_-{\xi_* |_r} \ar[rd]^-{\text{(A)}_1} & & & (g^\red)^* K_{\mfrakX, s_{\mfrakX}} |_{x_1} \ar[ddd]^-{\xi_* |_r} \ar[ld]_-{\text{(B)}_1} \\
  & \det(\mH^*(\bs{g}^*\L_{\bs{\mfrakX}} |_{x_1} )) \ar[r]^-{\text{(C)}_1} \ar[d]_-{\xi_* |_r} &\det(\mH^*(\tau^{\geq 0}(g^*\L_{\mfrakX})|_{x_1}))^{\otimes^2} \ar[d]^-{\xi_* |_r} & \\
  & \detr(\mH^*(\bs{g}^*\L_{\bs{\mfrakX}} |_{x_2} )) \ar[r]^-{\text{(C)}_2} & \det(\mH^*(\tau^{\geq 0}(g^*\L_{\mfrakX})|_{x_2}))^{\otimes^2} & \\
  \det(\bs{g}^*\L_{\bs{\mfrakX}}) |_{x_2} \ar[rrr]^-{\Lambda_{\bs{X}, \widehat{\bs{X}}, \bs{g}, \bs{\tau}}|_{x_2}} \ar[ru]^-{\text{(A)}_2} & & & \ar[lu]_-{\text{(B)}_2} (g^\red)^* K_{\mfrakX, s_{\mfrakX}} |_{x_2}. \\
  }
  \]
  Here (A)$_i$ is defined by the quasi-isomorphism $\bs{g}^* \L_{\bs{\mfrakX}} |_{x_i} \simeq \mH^*(\bs{g}^*\L_{\bs{\mfrakX}} |_{x_i} )$ and (B)$_i$ is defined by
  $\kappa_{g(x_i)}$ in Proposition \ref{prop:dcritstkcan} and the quasi-isomorphism
  $\tau^{\geq 0}(g^*\L_{\mfrakX})|_{x_i} \simeq \mH^*(\tau^{\geq 0}(g^*\L_{\mfrakX})|_{x_i})$.
  The map (C)$_i$ is defined in the same manner as \eqref{eq:dualconv} using the isomorphisms
  \[
  \mH^n(\bs{g}^*\L_{\bs{\mfrakX}} |_{x_i}) \cong
  \begin{cases} \mH^n(\tau^{\geq 0}(g^*\L_{\mfrakX})|_{x_i}) & n = 0,1 \\
                \mH^{-n-1}(\tau^{\geq 0}(g^*\L_{\mfrakX})|_{x_i})^\vee & n =-2, -1.
  \end{cases}
  \]
  The commutativity of the left trapezoid and middle square is obvious, and the commutativity of the right trapezoid
  follows from the proof of \cite[Theorem 2.56]{Joy}.
  It is easy to see that the upper and lower trapezoids commute up to the sign
  $(-1)^{\rank(\mH^1(\L_{\mfrakX} |_{g(x_i)}))}$ by using the equality \eqref{eq:extsign}.
  These commutativity properties imply the commutativity of the outer square,
  and hence the commutativity of the diagram \eqref{eq:vcandsc}.

  By Darboux theorem \cite[Theorem 2.10]{BBBBJ}, we can take $\bs{X}$, $\widehat{\bs{X}}$, $\bs{g}$ and $\bs{\tau}$ in the proposition
  so that $g$ is surjective.
  By the commutativity of the diagram \eqref{eq:vcandsc}, $\Lambda_{\bs{X}, \widehat{\bs{X}}, \bs{g}, \bs{\tau}}$ descends to
  $\Lambda_{\bs{\mfrakX}} \colon \detr(\L_{\bs{\mathfrak{X}}}) \cong K_{\mathfrak{X}, s_{\mfrakX}}$ satisfying the property in the proposition.
  The uniqueness of $\Lambda_{\bs{\mfrakX}}$ as in the theorem is clear from the construction.
\end{proof}
  The notion of orientation for $-1$-shifted symplectic derived Artin stacks is defined by
  that of the associated d-critical stacks.
  Let $\bs{\mfrakY}$ be a quasi-smooth derived Artin stack.
  The argument in Example \ref{ex:canori} works also for the stacky case
  and defines a natural isomorphism
  \begin{align}\label{eq:canoristk}
  o_{\bfT^*[-1]\bs{\mfrakY}}' \colon \detr(\bs{\pi}_{\bs{\mfrakY}}^*\mathbb{L}_{\bs{\mfrakY}})^{\otimes ^2}
  \cong \detr(\mathbb{L}_{\bfT^*[-1]\bs{\mfrakY}}).
  \end{align}
  We define the \emph{canonical orientation} $o^{}_{\bfT^*[-1]\bs{\mfrakY}}$ for $\bfT^*[-1]\bs{\mfrakY}$ by the composition
  $o^{}_{\bfT^*[-1]\bs{\mfrakY}} \coloneqq \Lambda_{\bs{\mfrakY}} \circ o'_{\bfT^*[-1]\bs{\mfrakY}}$.

  \begin{prop}\label{prop:oripull}
    Under the notation as in Proposition \ref{prop:dcritcot},
    we have an isomorphism
    \begin{align}\label{eq:oripullcan}
    \Xi_{\bs{f}} \colon \tilde{f}^\star o^{}_{\bfT^*[-1]\bs{\mfrakY}} \cong o^{}_{\bfT^*[-1]\bs{Y}}.
  \end{align}
  \end{prop}

  \begin{proof}
    Throughout the proof we use the following notation:
    for a morphism $\bs{h} \colon {\bs{\mfrakZ}} \to \bs{\mfrakW}$ of derived Artin stacks, we write
    \[
    \Delta_{\bs{h}} \colon \bs{h}^* \L_{\bs{\mfrakW}} \xrightarrow{\theta_{\bs{h}}} \L_{\bs{\mfrakZ}}
     \xrightarrow{\zeta_{\bs{h}}} \L_{\bs{h}} \xrightarrow{\delta_{\bs{h}}} \bs{h}^* \L_{\bs{\mfrakW}}[1]
    \]
    the natural distinguished triangle of cotangent complexes.

    Define
    \begin{align*}
    \Xi_{\bs{f}}' \colon&  (\tilde{f}^\red)^* \detr(\bs{\pi}_{\bs{\mfrakY}}^*\mathbb{L}_{\bs{\mfrakY}}) \otimes \detr(\L_{\widetilde{\bs{f}}})
    \cong  \detr(\bs{\pi}_{\bs{Y}}^*\mathbb{L}_{\bs{Y}})
   \end{align*}
   by using $\hi(\Delta_{\bs{f}})$ and the identification
   $(\bs{\pi}_{\bs{\mfrakY}, \bs{f}})^* \L_{\bs{f}} \cong \L_{\widetilde{\bs{f}}}$.
   Write
   \begin{align}\label{eq:deforip}
   \Xi_{\bs{f}} \coloneqq \sqrt{-1}^{{\vdim \bs{f} \cdot (\vdim \bs{f} - 1) / 2} + \vdim{\bs{\mfrakY}} \cdot \vdim{\bs{f}}} \cdot \Xi_{\bs{f}}'.
 \end{align}
   Now it is enough to prove the commutativity of the following diagram of line bundles on $\bfT^*[-1]\bs{Y}^{\red}$:
   \begin{equation}\label{eq:commori}
     \begin{split}
   \xymatrix{
   {\begin{subarray}{l}\ts ((\tilde{f}^\red)^* \detr(\bs{\pi}_{\bs{\mfrakY}}^*\mathbb{L}_{\bs{\mfrakY}})
    \ts \otimes \detr(\L_{\widetilde{\bs{f}}})) ^{\otimes ^2} \end{subarray}}
    \ar[r]^-{\Xi_{\bs{f}}^{\otimes ^2}} \ar[d]_-{o_{\bfT^*[-1]\bs{\mfrakY}}' \otimes \id}
      & \detr(\bs{\pi}_{\bs{Y}}^*\mathbb{L}_{\bs{Y}})^{\otimes ^2}
    \ar[d]^-{o_{\bfT^*[-1]\bs{Y}}'} \\
    {\begin{subarray}{l}
    \ts(\tilde{f}^\red)^* \detr(\mathbb{L}_{\bfT^*[-1]\bs{\mfrakY}})
     \ts \otimes \detr(\L_{\widetilde{\bs{f}}}) ^{\otimes ^2} \end{subarray}}
     \ar[d]_-{(\tilde{f}^\red)^*\Lambda_{\bfT^*[-1]\bs{\mfrakY}} \otimes
    \id}
    & \detr(\mathbb{L}_{\bfT^*[-1]\bs{Y}}) \ar[d]^-{\Lambda_{\bfT^*[-1]\bs{Y}}} \\
    {\begin{subarray}{l}\ts(\tilde{f}^\red)^*  K_{t_0(\bfT^*[-1]\bs{\mfrakY}), s_{\bfT^*[-1]\bs{\mfrakY}}}
    \ts \otimes \detr(\L_{\widetilde{\bs{f}}}) ^{\otimes ^2} \end{subarray}}
     \ar[r]_-{}
    & K_{t_0(\bfT^*[-1]\bs{Y}), s_{\bfT^*[-1]\bs{Y}}}.
    }
  \end{split}
 \end{equation}
 Here $\Lambda_{\bfT^*[-1]\bs{Y}}$ and $\Lambda_{\bfT^*[-1]\bs{\mfrakY}}$ are defined in Theorem \ref{thm:shiftdcrit} and Theorem \ref{thm:cotstk} respectively, and the bottom arrow is defined by using
 $\Gamma_{t_0(\bfT^*[-1]\bs{\mfrakY}) ,\tilde{f}}$ in $\eqref{eq:dcritstkcan2}$ and the identification
 \[
 \L_{\widetilde{\bs{f}}}\vert_{\bfT^*[-1]\bs{Y}^{\red}}
 \cong \Omega_{\tilde{f}} \vert_{\bfT^*[-1]\bs{Y}^{\red}}.
 \]

 Consider the following diagram in $\Perf(\bs{f}^* \bfT^*[-1]\bs{\mfrakY})$
 \begin{align}\label{eq:commdist1}
   \begin{split}
   \xymatrix@C=40pt{
    (\bs{\pi}_{\bs{\mfrakY}, \bs{f}})^* \mathbb{L}_{\bs{Y}}  \ar[r]^-{\bs{\tau}^*\theta_{\bs{\pi}_{\bs{Y}}}} \ar@{=}[d]
   & \bs{\tau}^* \L_{\bfT^*[-1]\bs{Y}} \ar[r]^-{\bs{\tau}^*\zeta_{\bs{\pi}_{\bs{Y}}}} \ar[d]^-{\theta_{\bs{\tau}}}
   & (\bs{\pi}_{\bs{\mfrakY}, \bs{f}})^*  \mathbb{T}_{\bs{Y}}[1] \ar[d]^-{(\bs{\pi}_{\bs{\mfrakY}, \bs{f}})^* \theta_{\bs{f}}^\vee[1]}\\
    (\bs{\pi}_{\bs{\mfrakY}, \bs{f}})^* \mathbb{L}_{\bs{Y}}  \ar[r]^-{\theta_{\bs{\pi}_{\bs{\mfrakY}, \bs{f}}}}
   & \L_{\bs{f}^* \bfT^*[-1]\bs{\mfrakY}} \ar[r]^-{\zeta_{\bs{\pi}_{\bs{\mfrakY}, \bs{f}}}} \ar[d]^-{\zeta_{\bs{\tau}}}
   & (\bs{\pi}_{\bs{\mfrakY}, \bs{f}})^* \bs{f}^* \mathbb{T}_{\bs{\mfrakY}}[1] \ar[d]^-{-(\bs{\pi}_{\bs{\mfrakY}, \bs{f}})^* \delta_{\bs{f}}^\vee[2]}\\
   &\L_{\bs{\tau}} \ar@{-->}[r]^-{k}_-{\sim} \ar[d]^-{\delta_{\bs{\tau}}}
   &  (\bs{\pi}_{\bs{\mfrakY}, \bs{f}})^* \mathbb{T}_{\bs{f}}[2] \ar[d]^-{(\bs{\pi}_{\bs{\mfrakY}, \bs{f}})^* \zeta_{\bs{f}}^\vee[2]}  \\
   & \bs{\tau}^* \L_{\bfT^*[-1]\bs{Y}}[1]  \ar[r]^-{\bs{\tau}^*\zeta_{\bs{\pi}_{\bs{Y}}}[1]}
   & (\bs{\pi}_{\bs{\mfrakY}, \bs{f}})^* \mathbb{T}_{\bs{Y}}[2]
   }
 \end{split}
\end{align}
where the top vertical arrows are identified with a part of the natural morphism between distinguished triangles
\[
\bs{\tau}^* \Delta_{\bs{\pi}_{\bs{Y}}} \to \Delta_{\bs{\pi}_{\bs{\mfrakY}, \bs{f}}},
\]
by the natural isomorphisms
\[
\bs{\tau}^* \bs{\pi}_{\bs{Y}}^* \L_{\bs{Y}} \cong (\bs{\pi}_{\bs{\mfrakY}, \bs{f}})^* \mathbb{L}_{\bs{Y}},
\,\, \bs{\tau}^* \L_{\bs{\pi}_{\bs{Y}}} \cong (\bs{\pi}_{\bs{\mfrakY}, \bs{f}})^*  \mathbb{T}_{\bs{Y}}[1],
\,\, \L_{\bs{\pi}_{\bs{\mfrakY}, \bs{f}}} \cong (\bs{\pi}_{\bs{\mfrakY}, \bs{f}})^* \bs{f}^* \mathbb{T}_{\bs{\mfrakY}}[1]
\]
 and $k$ is taken so that the right horizontal arrows define a morphism between distinguished triangles
\[
\Delta_{\bs{\tau}} \to \Delta_{\bs{f}, \rot}^\vee.
\]
Here $\Delta_{\bs{f}, \rot}^\vee$ denotes the right vertical distinguished triangle in the diagram above.
Now we claim that
\begin{align}\label{eq:detequal}
\detr(k) \circ \detr(\ell[2]) = \det(\phi[2])
\end{align}
where $\ell \colon \T_{\widetilde{\bs{f}}} \xrightarrow{\sim} \L_{\bs{\tau}}[-2]$ is defined in \eqref{eq:ell},
and $\phi \colon (\bs{\pi}_{\bs{\mfrakY}, \bs{f}})^* \T_{\widetilde{\bs{f}}} \xrightarrow{\sim} \mathbb{T}_{\bs{f}}$
is the natural isomorphism.
To see this consider the following commutative diagram:
\begin{align*}
  \xymatrix@C=40pt{
  \T_{\bs{f}^*\bfT^*[-1]\bs{\mfrakY}}[1] \ar[rr]^-{(\cdot \bs{\tau}^* \omega_{\bfT^*[-1]\bs{Y}})[1] \circ \theta_{\bs{\tau}}^\vee[1]} \ar[d]_{\theta_{\widetilde{\bs{f}}}^\vee[1]}
  & & \bs{\tau}^* \L_{\bfT^*[-1]\bs{Y}} \ar[r]^-{\bs{\tau}^* \zeta_{\bs{\pi}_{\bs{Y}}}}
  \ar[d]_{\theta_{\bs{\tau}}}
  & (\bs{\pi}_{\bs{\mfrakY}, \bs{f}})^* \T_{\bs{Y}}[1] \ar[d]^-{(\bs{\pi}_{\bs{\mfrakY}, \bs{f}})^*\theta_{\bs{f}}^\vee[1]} \\
  \widetilde{\bs{f}}^* \T_{\bfT^*[-1]\bs{\mfrakY}}[1] \ar[d]_-{-\delta_{\widetilde{\bs{f}}}^\vee[2]}
  \ar[rr]^-{\theta_{\widetilde{\bs{f}}} \circ (\cdot \widetilde{\bs{f}}^* \omega_{\bfT^*[-1]\bs{\mfrakY}})[1]}
  & &\L_{\bs{f}^*\bfT^*[-1]\bs{\mfrakY}} \ar[d]_-{\zeta_{\bs{\tau}}} \ar[r]^-{\zeta_{\bs{\pi}_{\bs{\mfrakY}, \bs{f}}}}
  &(\bs{\pi}_{\bs{\mfrakY}, \bs{f}})^* \bs{f}^* \T_{\bs{\mfrakY}}[1] \ar[d]^-{-(\bs{\pi}_{\bs{\mfrakY}, \bs{f}})^* \delta_{\bs{f}}^\vee[2]} \\
  \T_{\widetilde{\bs{f}}}[2] \ar[rr]^-{\ell[2]}
  &&\L_{\bs{\tau}} \ar[r]^-{k}
  &(\bs{\pi}_{\bs{\mfrakY}, \bs{f}})^* \T_{\bs{f}}[2]. \\
  }
\end{align*}
By using \cite[Proposition 6]{KM}, it is enough to prove the following equalities
\begin{align*}
   \bs{\tau}^* \zeta_{\bs{\pi}_{\bs{Y}}} \circ (\cdot \bs{\tau}^* \omega_{\bfT^*[-1]\bs{Y}})[1] \circ \theta_{\bs{\tau}}^\vee[1]
   = \theta_{\bs{\pi}_{\bs{\mfrakY}, \bs{f}}}^\vee [1] \\
   \zeta_{\bs{\pi}_{\bs{\mfrakY}, \bs{f}}} \circ \theta_{\widetilde{\bs{f}}} \circ (\cdot \widetilde{\bs{f}}^* \omega_{\bfT^*[-1]\bs{\mfrakY}})[1]
   = \widetilde{\bs{f}}^* \theta_{\bs{\pi}_{\bs{\mfrakY}}}^\vee [1]
\end{align*}
but these are consequences of \cite[Remark 2.5]{Cal}.

Now consider the following diagram of line bundles on $\bs{f}^* \bfT^*[-1]\bs{\mfrakY}^{\red}$, in which we omit the
pullback functors $\bs{\tau}^*$, $\bs{\pi}_{\bs{\mfrakY}, \bs{f}}^*$, and $(\bs{f} \circ \bs{\pi}_{\bs{\mfrakY}, \bs{f}} )^*$
to simplify the notation:
\begin{align}\label{eq:commdet1}
  \begin{split}
\xymatrix@C=40pt@R=30pt{
\detr(\L_{\bs{f}^* \bfT^*[-1]\bs{\mfrakY}}) \ar[d]_-{\hi(\Delta_{\bs{\pi}_{\bs{\mfrakY}, \bs{f}}})^{-1}} \ar[r]^-{\hi(\Delta_{\bs{\tau}})^{-1}}
& \detr( \L_{\bfT^*[-1]\bs{Y}}) \otimes \detr(\L_{\bs{\tau}}) \ar[d]^-{{\begin{subarray}{c}\hi(\Delta_{\bs{\pi}_{\bs{Y}}})^{-1} \otimes  \detr(\ell^{-1}[2])\end{subarray}}} \\
{\begin{subarray}{c} \ts \detr(\mathbb{L}_{\bs{Y}}) \otimes  \detr(\mathbb{T}_{\bs{\mfrakY}}[1]) \end{subarray}}
 \ar[r]^-{{\begin{subarray}{c} \id \otimes  \hi(\Delta_{\bs{f}, \rot}^\vee)^{-1}\end{subarray}}} \ar[d]_{{\begin{subarray}{c} \id \otimes \hat{\chi}_{ \mathbb{T}_{\bs{\mfrakY}}}\end{subarray}}}
&{\begin{subarray}{c}\ts \detr( \mathbb{L}_{\bs{Y}})
\otimes  \detr( \mathbb{T}_{\bs{Y}}[1])
 \otimes \detr(\mathbb{T}_{\widetilde{\bs{f}}}[2])  \end{subarray}}
 \ar[d]^-{{\begin{subarray}{c}\id \otimes  \hat{\chi}_{\mathbb{T}_{\bs{Y}}} \otimes  \hat{\chi}^{(2)}_{\mathbb{T}_{\widetilde{\bs{f}}}}\end{subarray}}}\\
 {\begin{subarray}{c} \ts \detr( \mathbb{L}_{\bs{Y}}) \otimes
  \detr( \mathbb{T}_{\bs{\mfrakY}})^{\otimes^{-1}}\end{subarray}}
   \ar[d]_-{{\begin{subarray}{c}\id \otimes (\hat{\eta}_{ \mathbb{L}_{\bs{\mfrakY}}}^{\otimes^{-1}})^{-1}\end{subarray}}}
&{\begin{subarray}{c}\ts \detr(\mathbb{L}_{\bs{Y}})
\otimes   \detr(\mathbb{T}_{\bs{Y}})^{\otimes^{-1}}
 \otimes \detr(\mathbb{T}_{\widetilde{\bs{f}}})  \end{subarray}}
 \ar[d]^-
 {{\begin{subarray}{c} \id \otimes
  (\hat{\eta}_{ \mathbb{L}_{\bs{Y}}}^{\otimes^{-1}})^{-1}
  \otimes  (\hat{\eta}_{\mathbb{L}_{\widetilde{\bs{f}}}})^{-1}\end{subarray}}} \\
 {\begin{subarray}{c} \ts\detr( \mathbb{L}_{\bs{Y}}) \otimes \detr( \mathbb{L}_{\bs{\mfrakY}})\end{subarray}}
  \ar[r]_(.5){{\begin{subarray}{c} \id \otimes  \hi(\Delta_{\bs{f}})\otimes \id_{\detr(\mathbb{L}_{\widetilde{\bs{f}}})^{\otimes^{-1}}} \end{subarray}}}
&{\begin{subarray}{c}\ts \detr( \mathbb{L}_{\bs{Y}}) ^{\otimes^2}
 \otimes \detr(\mathbb{L}_{\widetilde{\bs{f}}})^{\otimes^{-1}}. \end{subarray}}
}
\end{split}
\end{align}
Here $\hat{\eta}, \hat{\chi}$ are defined in Lemma \ref{lem:det}.
The commutativity of the diagram \eqref{eq:commdist1}, the equality \eqref{eq:detequal}, and \cite[Theorem 1]{KM} implies the commutativity of the upper square.
By applying Proposition \ref{prop:detdual} and Proposition \ref{prop:detrot} we see that the lower square also commutes.
Next consider the following commutative diagram in $\Perf(\bs{f}^* \bfT^*[-1]\bs{\mfrakY})$:
\begin{align}\label{eq:commdist2}
  \begin{split}
  \xymatrix@C=40pt{
   \widetilde{\bs{f}}^* \bs{\pi}_{\bs{\mfrakY}}^* \L_{\bs{\mfrakY}} \ar[r]^-{\widetilde{\bs{f}}^* \theta_{\bs{\pi}_{\bs{\mfrakY}}}}
   \ar[d]_-{(\bs{\pi}_{\bs{\mfrakY}, \bs{f}})^* \theta_{\bs{f}}}
  & \widetilde{\bs{f}}^* \L_{\bfT^*[-1]\bs{\mfrakY}} \ar[r]^-{\widetilde{\bs{f}}^* \zeta_{\bs{\pi}_{\bs{\mfrakY}}}} \ar[d]_-{\theta_{\widetilde{\bs{f}}}}
  & \widetilde{\bs{f}}^* \bs{\pi}_{\bs{\mfrakY}}^* \mathbb{T}_{\bs{\mfrakY}}[1] \ar@{=}[d] \\
  (\bs{\pi}_{\bs{\mfrakY}, \bs{f}})^* \L_{\bs{Y}} \ar[r]^-{ \theta_{\bs{\pi}_{\bs{\mfrakY}, \bs{f}}}}
  \ar[d]_-{(\bs{\pi}_{\bs{\mfrakY}, \bs{f}})^* \zeta_{\bs{f}}}
  & \L_{\bs{f}^* \bfT^*[-1]\bs{\mfrakY}} \ar[r]^-{\zeta_{\bs{\pi}_{\bs{\mfrakY}, \bs{f}}}}
  \ar[d]_-{\zeta_{\widetilde{\bs{f}}}}
  & \widetilde{\bs{f}}^* \bs{\pi}_{\bs{\mfrakY}}^* \mathbb{T}_{\bs{\mfrakY}}[1] \\
  (\bs{\pi}_{\bs{\mfrakY}, \bs{f}})^* \L_{\bs{f}} \ar[r]^-{\sim}
  & \L_{\widetilde{\bs{f}}}.
  &
  }
\end{split}
\end{align}
The upper vertical arrows are identified with a part of the natural morphism of distinguished triangles
\[
\widetilde{\bs{f}}^* \Delta_{\bs{\pi}_{\bs{\mfrakY}}} \to \Delta_{\bs{\pi}_{\bs{\mfrakY}, \bs{f}}}
\]
by the natural isomorphisms
\[
\widetilde{\bs{f}}^* \L_{\bs{\pi}_{\bs{\mfrakY}}} \cong \widetilde{\bs{f}}^* \bs{\pi}_{\bs{\mfrakY}}^*  \mathbb{T}_{\bs{\mfrakY}}[1],
\,\, \L_{\bs{\pi}_{\bs{\mfrakY}, \bs{f}}} \cong \widetilde{\bs{f}}^* \bs{\pi}_{\bs{\mfrakY}}^* \mathbb{T}_{\bs{\mfrakY}}[1]
\]
and the left horizontal arrows are identified with a part of the natural morphism
\[
(\bs{\pi}_{\bs{\mfrakY}, \bs{f}})^* \Delta_{\bs{f}} \to \Delta_{\widetilde{\bs{f}}}.
\]
by the natural isomorphism
\[
(\bs{\pi}_{\bs{\mfrakY}, \bs{f}})^* \L_{\bs{\mfrakY}}  \cong  \widetilde{\bs{f}}^* \bs{\pi}_{\bs{\mfrakY}}^* \L_{\bs{\mfrakY}}.
\]
Now consider the following diagram of line bundles on $\bs{f}^* \bfT^*[-1]\bs{\mfrakY}^{\red}$, in which we omit pullback functors as previous:
\begin{align}\label{eq:commdet2}
  \begin{split}
  \xymatrix@C=40pt@R=30pt{
  \detr(\L_{\bs{f}^* \bfT^*[-1]\bs{\mfrakY}}) \ar[d]_-{\hi(\Delta_{\bs{\pi}_{\bs{\mfrakY}, \bs{f}}})^{-1}}
   \ar[r]^-{\hi(\Delta_{\widetilde{\bs{f}}})^{-1}}
  & {\begin{subarray}{c} \ts
   \detr(\L_{{\bs{f}}}) \otimes \detr( \L_{\bfT^*[-1]\bs{\mfrakY}}) \end{subarray}} \ar[d]^-{{\begin{subarray}{c}
   (-1)^{\vdim \bs{\mfrakY} \cdot \vdim {\bs{f}}} \cdot \\
   \id \otimes \hi(\Delta_{\bs{\pi}_{\bs{\mfrakY}}})^{-1}
   \end{subarray}}} \\
  {\begin{subarray}{c} \ts \detr( \mathbb{L}_{\bs{Y}}) \otimes  \detr( \mathbb{T}_{\bs{\mfrakY}}[1])\end{subarray}}
   \ar[r]_-{{\begin{subarray}{c} \hi( \Delta_{\bs{f}})^{-1}
  \otimes \id \end{subarray}}}
 \ar[d]_{{\begin{subarray}{c}\id \otimes  \hat{\chi}_{ \mathbb{T}_{\bs{\mfrakY}}}\end{subarray}}}
  &{\begin{subarray}{c} \ts \detr ( \L_{\bs{f}}) \otimes
  \detr( \L_{\bs{\mfrakY}}) \otimes
       \detr( \mathbb{T}_{\bs{\mfrakY}}[1]) \end{subarray}}
    \ar[d]^-{\id \otimes \id \otimes \hat{\chi}_{\mathbb{T}_{\bs{\mfrakY}}}}\\
   {\begin{subarray}{c} \ts \detr( \mathbb{L}_{\bs{Y}}) \otimes  \detr(\mathbb{T}_{\bs{\mfrakY}})^{\otimes^{-1}}\end{subarray}}
     \ar[d]_-{{\begin{subarray}{c}\id \otimes (\hat{\eta}_{\mathbb{L}_{\bs{\mfrakY}}}^{\otimes^{-1}})^{-1}\end{subarray}}}
  & {\begin{subarray}{c} \ts   \detr ( \L_{\bs{f}}) \otimes
  \detr( \L_{\bs{\mfrakY}}) \otimes
      \detr( \mathbb{T}_{\bs{\mfrakY}})^{\otimes^{-1}} \end{subarray}}
    \ar[d]^-{\id \otimes \id \otimes (\hat{\eta}_{ \mathbb{L}_{\mathbb{\mfrakY}}}^{\otimes^{-1}})^{-1}} \\
   {\begin{subarray}{c} \ts \detr( \mathbb{L}_{\bs{Y}}) \otimes  \detr( \mathbb{L}_{\bs{\mfrakY}})\end{subarray}}
   \ar[r]_-{{\begin{subarray}{c} \hi( \Delta_{\bs{f}})^{-1}  \otimes \id\end{subarray}}}
  & {\begin{subarray}{c} \ts  \detr ( \L_{\bs{f}}) \otimes \detr(\L_{\bs{\mfrakY}})^{\otimes^2}.
      \end{subarray}}
  }
\end{split}
\end{align}
The commutativity of the diagram \eqref{eq:commdist2} and \cite[Theorem 1]{KM} implies the commutativity of the upper square, and
the commutativity of the lower square is obvious.
By combining the commutativity of the diagrams \eqref{eq:commdet1} and \eqref{eq:commdet2}, we obtain the commutativity of
the diagram \eqref{eq:commori}
(the sign $(-1)^{{\vdim \bs{f} \cdot (\vdim \bs{f} - 1) / 2}}$ appears due to the difference of the maps
\eqref{eq:dualconv} and \eqref{eq:dualconv2}).
  \end{proof}

  \begin{rem}\label{rem:oripullass}
  Under the situation of the proposition above, assume further that there exists a smooth morphism
  $\bs{q} \colon \bs{Y}' \to \bs{Y}$, and write
  $\tilde{q} \colon t_0(\bfT^*[-1]\bs{Y}') \to t_0(\bfT^*[-1]\bs{Y})$ the base change of $q = t_0(\bs{q})$.
  Then it is clear that the following composition
  \[
  \xymatrix{
  (\tilde{f} \circ \tilde{q})^\star o^{}_{\bfT^*[-1]\bs{\mfrakY}} \ar[r]_-{\sim}^-{\text{\eqref{eq:oripullcomp}}}
  & \tilde{q}^\star \tilde{f}^\star o^{}_{\bfT^*[-1]\bs{\mfrakY}} \ar[r]_-{\sim}^-{\tilde{q}^* \Xi_{\bs{f}}}
  & \tilde{q}^\star o^{}_{\bfT^*[-1]\bs{Y}} \ar[r]_-{\sim}^-{\Xi_{\bs{q}}} & o^{}_{\bfT^*[-1]\bs{Y}'}
  }
  \]
  is equal to $\Xi_{\bs{f} \circ \bs{q}}$.
  \end{rem}

\subsection{Dimensional reduction for Artin stacks}
We first recall the definition of the vanishing cycle complexes associated with d-critical stacks.
To do this, we discuss the functorial behavior of the vanishing cycle complexes associated with d-critical schemes
with respect to smooth morphisms.
\begin{prop}\cite[Proposition 4.5]{BBBBJ}
  Let $(Y, s, o)$ be an oriented d-critical scheme, and $q \colon X \to Y$ be a smooth morphism.
  Then there exists a natural isomorphism
  \[
  \Theta_q = \Theta_{q, s, o} \colon  \varphi^p _{X, q^\star s, q^\star o} \cong q^* \varphi^p _{Y, s, o}[\dim q]
  \]
   characterized by the following property:
  for a d-critical chart $\mathscr{R} = (R, U, f, i)$ of $(X, s)$, a d-critical chart $\mathscr{S} = (S, V, g, j)$ of $(Y, s)$ such that $q(R) \subset S$,
  and a smooth morphism $\tilde{q} \colon U \to V$ such that $f = g \circ \tilde{q}$ and $j \circ q = \tilde{q} \circ i$,
  the following diagram commutes
  \[
  \xymatrix@C=70pt{
  \varphi_{X, q^\star s, q^\star o}^p |_R \ar[r]^-{\omega_{\mathscr{R}}} \ar[d]_{\Theta_{q}|_R}
  & i^* \varphi^p_f \otimes_{\setZ/2\setZ} Q_{\mathscr{R}}^{q^\star o}|_R
  \ar[d]^-{\Theta_{\tilde{q}, f} \otimes \rho_{q}}
  \\
  q^* \varphi_{Y, s, o}^p [\dim q] |_R \ar[r]^-{q^*\omega_{\mathscr{S}}[\dim q]}
  & j^* \tilde{q}^* \varphi^p_g [\dim q]\otimes_{\setZ/2\setZ} (q|_R)^* (Q_{\mathscr{S}}^o|_S)
  }
  \]
  where $\Theta_{\tilde{q}, f}$ is defined in Proposition \ref{prop:van}(iii), and $\rho_q$ is defined by using the
  natural isomorphism $K_U \cong \tilde{q}^* K_V \otimes \det(\Omega_{U/V})$.
\end{prop}

\begin{thm}\cite[Theorem 4.8]{BBBBJ}
Let $(\mfrakX, s, o)$ be an oriented d-critical stack.
Then there exists a natural perverse sheaf $\varphi_{\mfrakX, s, o}$ with the following property:
for each $(u \colon U \to \mfrakX) \in \lisan(\mfrakX)$ there exists an isomorphism
\[
 \Theta_u = \Theta_{u, s, o} \colon \varphi^p_{U, u^\star s, u^\star o} \cong u^* \varphi^p_{\mfrakX, s, o}[\dim u]
\]
satisfying $\Theta_{u, s, o} = q^*\Theta_{v, s, o}[\dim q] \circ  \Theta_{q, v^\star s, v^\star o}$ for any smooth morphism $q \colon (u\colon U \to \mfrakX) \to (v\colon V \to \mfrakX)$ in $\lisan(\mfrakX)$.
Here we identify $q^\star v^\star o$ and $u^\star o$ by using \eqref{eq:oripullcomp}.
\end{thm}

Let $(\mfrakX, s)$ be a d-critical stack, and $\Xi \colon o_1 \cong o_2$ be an isomorphism between orientations on $(\mfrakX,s)$.
We write
\[
a_{\Xi} \colon \varphi_{\mfrakX, s, o_1} \cong \varphi_{\mfrakX, s, o_2}
\]
the isomorphism induced by $\Xi$.

Now we state our main theorem.

\begin{thm}\label{thm:dimred2}
Let $\bs{\mfrakY}$ be a quasi-smooth derived Artin stack, and equip $\bfT^*[-1]\bs{\mfrakY}$
with the canonical $-1$-shifted sympelectic structure and the canonical orientation.
Then we have a natural isomorphism
\[
\bar{\gamma}_{\bs{\mfrakY}} \colon ({\pi_{\bs{\mfrakY}}})_! \varphi^p_{\bfT^*[-1]\bs{\mfrakY}} \cong \setQ_{\mfrakY}[\vdim \bs{\mfrakY}]
\]
where we write $\mfrakY = t_0(\bs{\mfrakY})$ and ${\pi_{\bs{\mfrakY}}} = t_0(\bs{\pi}_{\bs{\mfrakY}})$.
\end{thm}

\begin{proof}
Take a smooth surjective morphism $\bs{v} \colon \bs{V} \to \bs{\mfrakY}$ and an \'etale morphism
$\bs{\eta} \colon \bs{U} \to \bs{V} \times_{\bs{\mfrakY}} \bs{V}$ where $\bs{V}$ and $\bs{U}$ are derived schemes.
 $\bs{q}_1, \bs{q}_2 \colon \bs{U} \to \bs{V}$
denote the composition of $\bs{\eta}$ and the first and second projections respectively. Write $U = t_0(\bs{U})$,
$\widetilde{U} = t_0(\bfT^*[-1]\bs{U})$, $V = t_0(\bs{V})$, $\widetilde{V} = t_0(\bfT^*[-1]\bs{V})$,
$\widetilde{\mfrakY} = t_0(\bfT^*[-1]\bs{\mfrakY})$, $v = t_0(\bs{v})$, and $q_i = t_0(\bs{q}_i)$ for $i =1, 2$.
Denote by $\pi_{\mfrakY} \colon \widetilde{\mfrakY} \to \mfrakY$,
$\pi_{U} \colon \widetilde{U} \to U$, and $\pi_{V} \colon \widetilde{V} \to V$ the projections, and by
$\tilde{v} \colon \widetilde{V} \to \widetilde{\mfrakY}$ (resp. $\tilde{q}_i \colon \widetilde{U} \to \widetilde{V}$)
 the base change of $v$ (resp. $q_i$).
Denote by $s$ the d-critical structure on $\widetilde{\mfrakY}$ associated with the canonical $-1$-shifted symplectic structure $\omega_{\bfT^*[-1]\bs{\mfrakY}}$.

Define
\[
\bar{\gamma}_{\bs{\mfrakY}, \bs{v}} \colon
v^*(\pi_{\mfrakY})_! \varphi^p_{\bfT^*[-1]\bs{\mfrakY}} \cong v^* \setQ_{\mfrakY}[\vdim \bs{\mfrakY}]
\]
by the following composition:
\begin{align*}
  v^*({\pi_{\mfrakY}})_! \varphi^p_{\bfT^*[-1]\bs{\mfrakY}}
  &\cong (\pi_{{V}})_! \tilde{v}^* \varphi^p_{\bfT^*[-1]\bs{\mfrakY}} \\
  &\cong (\pi_{{V}})_! \varphi^p_{\bfT^*[-1]\bs{V}}[-\dim v] \\
  &\cong v^* \setQ_{\mfrakY}[\vdim \bs{\mfrakY}].
\end{align*}
where the first isomorphism is the proper base change map,
the second isomorphism is $(\pi_{V})_! ( a_{\Xi_{\bs{v}}} \circ  \Theta_{\tilde{v}}^{-1}) [-\dim v]$,
and the third isomorphism is $\bar{\gamma}_{\bs{V}} [-\dim v]$.
By the sheaf property, it is enough to prove the commutativity of the following diagram:
\begin{align}\label{eq:commwant}
  \begin{split}
    \xymatrix{
    q_1 ^* v^*(\pi_{{\mfrakY}})_! \varphi^p_{\bfT^*[-1]\bs{\mfrakY}} \ar[d]_-{\xi_*} \ar[r]^-{q_1^*\bar{\gamma}_{\bs{\mfrakY}, \bs{v}}}
    & q_1^* v^* \setQ_{\mfrakY}[\vdim \bs{\mfrakY}] \ar[d]^-{\xi_*} \\
    q_2 ^* v^*(\pi_{{\mfrakY}})_! \varphi^p_{\bfT^*[-1]\bs{\mfrakY}} \ar[r]^-{q_2^* \bar{\gamma}_{\bs{\mfrakY}, \bs{v}} }
    & q_2^* v^* \setQ_{\mfrakY}[\vdim \bs{\mfrakY}]
    }
  \end{split}
\end{align}
where $\xi \colon v \circ q_1 \Rightarrow v \circ q_2$ is the natural $2$-morphisim.
We define
\[
\bar{\gamma}_{\bs{\mfrakY}, {\bs{v}} \circ \bs{q}_i} \colon
(v \circ q_i)^*(\pi_{{\mfrakY}})_! \varphi^p_{\bfT^*[-1]\bs{\mfrakY}} \cong (v \circ q_i)^* \setQ_{\mfrakY}[\vdim \bs{\mfrakY}]
\]
for $i =1, 2$ in the same manner as $\bar{\gamma}_{\bs{\mfrakY}, \bs{v}}$.
The commutativity of the diagram \eqref{eq:pbcass} implies the commutativity of the following diagram:
\[
\xymatrix{
 (v \circ q_1)^*(\pi_{{\mfrakY}})_! \varphi^p_{\bfT^*[-1]\bs{\mfrakY}} \ar[d]_-{\xi_*} \ar[r]^-{\bar{\gamma}_{\bs{\mfrakY}, {\bs{v}} \circ \bs{q}_1} }
 & (v \circ q_1)^* \setQ_{\mfrakY}[\vdim \bs{\mfrakY}] \ar[d]^-{\xi_*} \\
  (v \circ q_2)^*(\pi_{{\mfrakY}})_! \varphi^p_{\bfT^*[-1]\bs{\mfrakY}} \ar[r]^-{\bar{\gamma}_{\bs{\mfrakY}, {\bs{v}} \circ \bs{q}_2} } & (v \circ q_2)^* \setQ_{\mfrakY}[\vdim \bs{\mfrakY}].
}
\]
Therefore the commutativity of the diagram \eqref{eq:commwant} follows once we prove the commutativity of the following diagram
\begin{align}\label{eq:commwant2}
  \begin{split}
\xymatrix{
(v \circ q_i)^*(\pi_{{\mfrakY}})_! \varphi^p_{\bfT^*[-1]\bs{\mfrakY}} \ar[d]_{\simd} \ar[r]^-{\bar{\gamma}_{\bs{\mfrakY}, {\bs{v}} \circ \bs{q}_i} }
& (v \circ q_i)^* \setQ_{\mfrakY}[\vdim \bs{\mfrakY}] \ar[d]^-{\simd} \\
q_i ^* v^*(\pi_{{\mfrakY}})_! \varphi^p_{\bfT^*[-1]\bs{\mfrakY}} \ar[r]^-{q_i^*\bar{\gamma}_{\bs{\mfrakY}, \bs{v}}}
& q_i^* v^* \setQ_{\mfrakY}[\vdim \bs{\mfrakY}]
}
\end{split}
\end{align}
for each $i = 1, 2$. We drop $i$ from the notation, and write $q = q_i$ and $\tilde{q} = \tilde{q}_i$.
By Remark \ref{rem:oripullass}, the following diagram commutes:
\[
\xymatrix{
 &\varphi^p_{\bfT^*[-1]\bs{U}}  \ar[dl]_-{a_{\Xi_{\bs{v}}}^{-1} \circ \Theta_{\tilde{v}}}
 \ar[dr]^-{a_{\Xi_{\bs{v} \circ \bs{q}}}^{-1} \circ \Theta_{\tilde{q} \circ \tilde{v}}} &   \\
\tilde{q}^* \varphi^p_{\bfT^*[-1]\bs{V}}[\dim q]
\ar[rr]^-{a_{\Xi_{\bs{q}}}^{-1} \circ \Theta_{\tilde{q}}} &  & \tilde{q}^* \tilde{v}^* \varphi^p_{\bfT^*[-1]\bs{\mfrakY}}[\dim v \circ q].
}
\]
Using this and the commutativity of the diagram \eqref{eq:pbcass} again, the commutativity of the diagram \eqref{eq:commwant2} is implied by the commutativity of the following diagram
\begin{align}\label{eq:commlast}
  \begin{split}
\xymatrix@C=40pt{
(\pi_{{U}})_! \varphi^p_{\bfT^*[-1]\bs{U}} \ar[r]^-{\bar{\gamma}_{\bs{U}}}
\ar[d]_-{\pbc_q \circ (\pi_{{V}})_! (a_{\Xi_{\bs{q}}}^{-1} \circ \Theta_{\tilde{q}})} & \setQ_{U} [\vdim \bs{U}] \ar[d]^-{\simd}  \\
q^* (\pi_{{V}})_! \varphi^p_{\bfT^*[-1]\bs{V}}[\dim q]  \ar[r]^-{q^* \bar{\gamma}_{\bs{V}}[\dim q]}
 & q^*\setQ_{V} [\vdim \bs{U}]
}
\end{split}
\end{align}
where $\pbc_q$ denotes the base change map.
Arguing as the proof of \cite[Theorem 2.9]{BBBBJ} and by shrinking if necessary,
we may assume that there exist a smooth morphism $F \colon M \to N$ with a constant relative dimension between smooth schemes, a vector bundle $E$ of rank $r$ on $N$, and its section $e \in \Gamma(N, E)$
such that $\bs{V} = \bs{Z}(e)$, $\bs{U} = \bs{Z}(F^* e)$, and $\bs{q} \colon \bs{Z}(F^* e) \to \bs{Z}(e)$ is the base change of $F$.
Write $\widetilde{F} \colon \Tot_M(F^*E^\vee) \to \Tot_N(E^\vee)$ the base change of $F$, and $\bar{e} \colon \Tot_N(E^\vee) \to \mathbb{A}^1$ denotes the regular
function corresponding to $e$.
Then
\[\mathscr{U} = (\widetilde{U}, \Tot_M(F^* E^\vee), \bar{e} \circ \widetilde{F}, i), \, \, \mathscr{V} = (\widetilde{V}, \Tot_N(E^\vee), \bar{e}, j)
\] define d-critical
charts on $\widetilde{U}$ and $\widetilde{V}$ respectively, where $i$ and $j$ denote the natural embeddings.
Consider the following composition
\[
\rho_{\tilde{q}}' \colon
Q_{\mathscr{U}}^{o^{}_{\bfT^*[-1]\bs{U}}} \cong
Q_{\mathscr{U}}^{q^\star o^{}_{ \bfT^*[-1]\bs{V}}} \xrightarrow{\rho_{\tilde{q}}}
 \tilde{q}^* Q_{\mathscr{V}}^{o^{}_{\bfT^*[-1]\bs{V}}}
\]
where the first map is induced by $\Xi_{\bs{q}}$.
Recall that we have chosen trivializations of $Q_{\mathscr{U}}^{o^{}_{\bfT^*[-1]\bs{U}}}$ and $Q_{\mathscr{V}}^{o^{}_{\bfT^*[-1]\bs{V}}}$
in \eqref{eq:rootchoice}.
Since we have
\[
\epsilon_{\dim N, r} / \epsilon_{\dim M, r} = \sqrt{-1}^{{\vdim \bs{q} \cdot (\vdim \bs{q} - 1) /2} + \vdim{\bs{V}} \cdot \vdim{\bs{q}}},
\]
these trivializations are identified by $\rho_{\tilde{q}}'$.
This shows the commutativity of the following diagram
\begin{align*}
  \xymatrix@C=40pt{
  \varphi^p_{\bfT^*[-1]\bs{U}} \ar[r]^-{\omega_{\mathscr{U}}} \ar[d]_-{a_{\Xi_{\bs{q}}}^{-1} \circ \Theta_{\tilde{q}}}
  &  i^* \varphi^p_{\bar{e} \circ \widetilde{F}} \otimes _{\setZ/2\setZ}
  Q_{\mathscr{U}}^{o^{}_{\bfT^*[-1]\bs{U}}}  \ar[r]^-{\id \otimes \text{triv}}
  & i^* \varphi^p_{\bar{e} \circ \widetilde{F}} \ar[d]^-{i^* \Theta_{\widetilde{F}, \bar{e}}} \\
  \tilde{q}^* \varphi^p_{\bfT^*[-1]\bs{V}}  \ar[r]^-{\tilde{q}^* \omega_{\mathscr{U}}}
  & \tilde{q}^*(j^* \varphi^p_{\bar{e}}[\dim q] \otimes _{\setZ/2\setZ} Q_{\mathscr{V}}^{o^{}_{\bfT^*[-1]\bs{V}}})  \ar[r]^-{\tilde{q}^*(\id \otimes \text{triv})} &  \tilde{q}^*j^* \varphi^p_{\bar{e}}[\dim q]
  }
\end{align*}
where two $\text{triv}$ in the right horizontal arrows denote the trivialization as above.
Then the commutativity of the diagram \eqref{eq:commlast} follows from the commutativity of
the diagram \eqref{eq:vanpullback}.
\end{proof}

\section{Applications}\label{section5}

In this section, we will discuss two applications of Theorem \ref{thm:dimred1} and its stacky generalization Theorem \ref{thm:dimred2}.
Firstly, we will apply it to prove the dimensional reduction theorem for the vanishing cycle cohomology
of the moduli stacks of sheaves on local surfaces.
Secondly, we will propose a sheaf theoretic construction of virtual fundamental classes of quasi-smooth derived schemes
by regarding Theorem \ref{thm:dimred1} as a version of Thom isomorphism for $-1$-shifted cotangent stacks.

\subsection{Cohomological Donaldson--Thomas theory for local surfaces}

Let $S$ be a smooth quasi-projective surface and denote by $p \colon X = \Tot_S(\omega_S) \to S$ the projection
from the total space of the canonical bundle.
Denote by $\bs{\mfrakM}_{S}$ (resp. $\bs{\mfrakM}_{X}$) the derived moduli stack of coherent sheaves on $S$ (resp. $X$) with proper supports, and $\bs{\pi}_{p} \colon \bs{\mfrakM}_{X} \to \bs{\mfrakM}_{S}$ the projection defined by $p_*$.
By applying the main theorem of \cite{BD1}, $\bs{\mfrakM}_{X}$ carries a canonical $-1$-shifted symplectic strucure $\omega_{\bs{\mfrakM}_{S}}$.
\begin{thm}\label{thm:compsymp}
  There exists an equivalence of $-1$-shifted symplectic derived Artin stacks
  \begin{align}
    \bs{\Psi} \colon (\bs{\mfrakM}_{X}, \omega_{\bs{\mfrakM}_{X}}) \simeq
    (\bfT^*[-1] \bs{\mfrakM}_{S}, \omega_{\bfT^*[-1] \bs{\mfrakM}_{S}})
  \end{align}
  such that $\bs{\pi}_{p} \simeq \bs{\pi}_{\bs{\mfrakM}_{S}} \circ \Phi$.
\end{thm}
\begin{proof}
Let $G$ be a compact generator of $D(\Qcoh(S))$, and $A = \RHom(G, G)$ and $B = \RHom(p^*G, p^*G)$ be the derived endomorphism algebras.
It is clear that $p^*G$ is a compact generator of $D(\Qcoh(X))$, and we have quasi-equivalences:
\begin{align*}
  \RHom(G, -) &\colon L_{\mathrm{coh}}(S) \xrightarrow{\sim} \per_{\dg} A, \\
  \RHom(p^*G, -) &\colon L_{\mathrm{coh}}(X) \xrightarrow{\sim} \per_{\dg} B
\end{align*}
where $L_{\mathrm{coh}}(S)$ (resp. $L_{\mathrm{coh}}(X)$) denotes the derived dg-category of $\Coh(S)$ (resp. $\Coh(X)$),
and $\per_{\dg} A$ (resp. $\per_{\dg} B$) denotes the derived dg-category of perfect $A$-modules (resp. $B$-modules).
It is proved in \cite[\textsection 2.5]{IQ} that $B$ is equivalent to the $3$-Calabi--Yau completion of $A$.
Therefore, by applying \cite[Theorem 6.17]{BCS}, we only need to prove the coincidence of two left Calabi--Yau structure
\[c_1, c_2 \in \mathrm{HC}^{-}_{3}(L_{\mathrm{coh}}(X)) \cong \mathrm{H}^0(X, \omega_X)
\]
where $c_1$ is induced by the Calabi--Yau completion description
and $c_2$ corresponds to the canonical Calabi--Yau form on $X$.
Since the statement is local on $S$, we may assume $S$ is affine.
By the discussion after \cite[Theorem 5.8]{BCS}, we see that $c_1 = \delta c$ where $\delta$ denotes the mixed differential
and
\[
c \in \mH\mH_2(L_{\mathrm{coh}}(X)) \cong \mathrm{H}^0(X, \wedge^2 \Omega_X)
\] corresponds to the tautological $2$-form on $X$ under the
Hochschild-Kostant-Rosenberg isomoprhism.
Since the
Hochschild-Kostant-Rosenberg isomorphism identifies the mixed differential on the Hochschild homology and the de Rham differential (see \cite{TV}),
the theorem is proved.
\end{proof}
We always equip $\bs{\mfrakM}_{X}$ with the canonical $-1$-shifted symplectic structure as above and the orientation $t_0(\bs{\Psi})^\star o^{}_{\bfT^*[-1] \bs{\mfrakM}_{S}}$.
The following statement is a direct consequence of Theorem \ref{thm:dimred2}:
\begin{cor}\label{cor:dimred3}
  We have an isomorphism
  \[
  (\pi_{p})_! \varphi^p _{\bs{\mfrakM}_{X}} \cong \setQ_{\mfrakM_{S}}[\vdim \bs{\mfrakM}_{S}]
  \]
  where we write $\mfrakM_{X} = t_0(\bs{\mfrakM}_{X})$, $\mfrakM_{S} = t_0(\bs{\mfrakM}_{S})$, and $\pi_{p} = t_0(\bs{\pi}_{p})$.
\end{cor}
Now assume $S$ is quasi-projective and $\omega_S$ is trivial.
Take an ample divisor $H$ on $S$ and denote by $\mfrakM_{S}^{H\ss} \subset \bs{\mfrakM}_{S}$
(resp. $\mfrakM_{X}^{p^*H\ss} \subset \mfrakM_{X}$)
 the moduli stack of $H$-semistable sheaves on $S$ (resp. $p^*H$-semistable sheaves on $X$) with proper supports.
 By the triviality of $\omega_S$, we have an equality
 $\pi_{p}^{-1} (\mfrakM_{S}^{H\ss}) = \mfrakM_{X}^{p^*H\ss}$.
 This observation and the Verdier self-duality of $\varphi^p _{\bs{\mfrakM}_{X}}$ implies the following corollary:
 \begin{cor}
   Write $\varphi^p_{\bs{\mfrakM}_{X}^{p^*H\ss}} \coloneqq \varphi^p_{\bs{\mfrakM}_{X}} |_{\mfrakM_{X}^{p^*H\ss}}$.
   Then we have following isomorphisms:
   \begin{align*}
     \mH^*_c (\mfrakM_{X}^{p^*H\ss}; \varphi^p_{\mfrakM_{X}^{p^*H\ss}}) &\cong \mH_c^{* + \vdim \bs{\mfrakM}_{S}^{H\ss}}
     (\mfrakM_{S}^{H\ss}), \\
     \mH^* (\mfrakM_{X}^{p^*H\ss}; \varphi^p_{\mfrakM_{X}^{p^*H\ss}})
     &\cong \mH_{\vdim \bs{\mfrakM}_{S}^{H\ss}-*}^{\mathrm{BM}}(\mfrakM_{S}^{H\ss}).
   \end{align*}
   Here $\mH_c$ denotes the cohomology with compact support, and $\mH^{\mathrm{BM}}$ denotes the Borel--Moore homology.
 \end{cor}

\subsection{Thom isomorphism}

Let $\bs{Y}$ be a quasi-smooth derived scheme,
and write $Y = t_0(\bs{Y})$ and $\widetilde{Y} = t_0(\bfT^*[-1]\bs{Y})$.
Thanks to Theorem \ref{thm:dimred1}, we have the following isomorphism:
\begin{align}\label{eq:virtualhomot}
  \mH^*(\widetilde{Y} ; \varphi^p_{\bfT^*[-1]\bs{Y}})
  \cong \mH^{\mathrm{BM}}_{\vdim \bs{Y} - *}(Y).
\end{align}
Since $\varphi_{\bfT^*[-1]\bs{Y}}^p$ is conical, by using Theorem \ref{thm:dimred1} and \cite[Proposition 3.7.5]{KaSc},
we also have the following isomorphism:
\begin{align}\label{eq:virtualThom}
  \mH^*(\widetilde{Y}, \widetilde{Y} \setminus Y ; \varphi^p_{\bfT^*[-1]\bs{Y}})
  \cong \mH^{* + \vdim \bs{Y}}(Y).
\end{align}
This isomorphism can be regarded as a version of the Thom isomorphism.
Indeed, if $\bs{Y} = M \times^{\mathbf{R}}_{0_E, E, 0_E} M$ where $M$ is a smooth scheme, $E$ is a vector bundle on $M$,
and $0_E$ is the zero section of $E$, the isomorphism \eqref{eq:virtualThom} is the usual Thom isomorphism.
By imitating the construction of the Euler class, we construct a class
\[
e(\bfT^*[-1] \bs{Y}) \in \mH^{\mathrm{BM}}_{2\vdim \bs{Y}}(Y)
\] by the image of $1 \in \mathrm{H}^{0}(Y)$ under the following composition:
\[
\xymatrix@C=15pt{
\mathrm{H}^{0}(Y) \ar[r]^-{\eqref{eq:virtualThom}}_-{\sim} &
 \mH^{-\vdim \bs{Y}}(\widetilde{Y}, \widetilde{Y} \setminus Y ; \varphi^p_{\bfT^*[-1]\bs{Y}}) \ar[r]
 & \mH^{-\vdim \bs{Y}}(\widetilde{Y}; \varphi^p_{\bfT^*[-1]\bs{Y}})\ar[r]_-{\sim}^-{\eqref{eq:virtualhomot}}  & \mathrm{H}^\mathrm{BM} _{2\mathrm{vdim \bs{Y}}}(Y).
}
\]
Denote by $[\bs{Y}]^{\vir} \in \mH^{\mathrm{BM}}_{2\vdim \bs{Y}}(Y)$ the virtual fundamental class of $\bs{Y}$
constructed by Behrend--Fantechi in \cite{BF}.
We have the following conjecture:
\begin{conj}\label{conj:VFC}
\[
  e(\bfT^*[-1] \bs{Y}) = (-1)^{\vdim \bs{Y} \cdot (\vdim \bs{Y} - 1) /2}[\bs{Y}]^{\vir}.
\]
\end{conj}
In other words, we expect that this gives a new construction of the virtual fundamental class.

\begin{ex}
Assume $\bs{Y} = M \times^{\mathbf{R}}_{0_E, E, 0_E} M$.
In this case, we have $\widetilde{Y} \cong \Tot_{N}(E^\vee)$ and
$\varphi^p_{\bfT^*[-1]\bs{Y}} \cong \setQ_{\widetilde{Y}}[\dim \widetilde{Y}]$, and the
construction of the Verdier duality isomorphism $\sigma_{\widetilde{Y}}$ in \eqref{eq:dualchoice} implies that the following diagram commutes:
\[
\xymatrix@C=100pt{
 \mathbb{D}_{\widetilde{Y}} (\varphi^p_{\bfT^*[-1]\bs{Y}})\ar[r]^-{(-1)^{\dim \widetilde{Y}\cdot (\dim \widetilde{Y}-1)/2} \sigma_{\widetilde{Y}}} \ar[d]_-{\simd}
 & \varphi^p_{\bfT^*[-1]\bs{Y}} \ar[d]^-{\simd} \\
 \mathbb{D}_{\widetilde{Y}}(\setQ_{\widetilde{Y}}[\dim \widetilde{Y}]) \ar[r]^{\sim}
 & \setQ_{\widetilde{Y}}[\dim \widetilde{Y}].
}
\]
Therefore we have
\begin{align*}
e(\bfT^*[-1] \bs{Y}) &= (-1)^{\dim \widetilde{Y}\cdot (\dim \widetilde{Y}-1)/2} e(E^\vee) \cap [M] \\
          &= (-1)^{(\dim \widetilde{Y}\cdot (\dim \widetilde{Y}-1)/2) + \rank E} e(E) \cap [M] \\
          &= (-1)^{\vdim \bs{Y} \cdot (\vdim \bs{Y} - 1) /2} [\bs{Y}]^{\vir}.
\end{align*}
\end{ex}

The author has verified Conjecture \ref{conj:VFC} under a certain assumption:
\begin{thm}\cite{Kin}
  If $\L_{\bs{Y}} |_Y$ is represented by a two-term complex of vector bundles, then Conjecture \ref{conj:VFC} is true.
  In particular, Conjecture \ref{conj:VFC} holds when $Y$ is quasi-projective.
\end{thm}

\begin{rem}
  We can extend the above construction for stacky cases as follows.
  Let $\bs{\mfrakY}$ be a quasi-smooth derived Artin stack and write $\mfrakY = t_0(\bs{\mfrakY})$
  and $\widetilde{\mfrakY} = t_0(\bfT^*[-1]\bs{\mfrakY})$.
  Denote by $\pi \colon \widetilde{\mfrakY} \to \mfrakY$ the projection and by $i \colon \mfrakY \to \widetilde{\mfrakY}$
  the zero section.
  Then \cite[Proposition 3.7.5]{KaSc} and the smooth base change theorem implies isomorphisms of functors
  \[
  \pi_! \cong i^!, \,\, \pi_* \cong i^*.
  \]
  Therefore Theorem \ref{thm:dimred2} implies isomorphisms
  \begin{align*}
  i^! \varphi_{\bfT^*[-1]\bs{\mfrakY}} &\cong \setQ_{\mfrakY}[\vdim \bs{\mfrakY}], \\
  i^* \varphi_{\bfT^*[-1]\bs{\mfrakY}} &\cong \omega_{\mfrakY}[- \vdim \bs{\mfrakY}].
\end{align*}
Consider the following composition of natural transforms:
\[
i^! \varphi_{\bfT^*[-1]\bs{\mfrakY}}^p \to i^! i_* i^* \varphi_{\bfT^*[-1]\bs{\mfrakY}}^p \cong i^! i_! i^* \varphi_{\bfT^*[-1]\bs{\mfrakY}}^p \cong i^* \varphi_{\bfT^*[-1]\bs{\mfrakY}}^p
\]
where the first map is the $*$-unit, the second map defined in the same manner as \cite[Proposition 4.6.2]{LO1},
and the third map is the inverse of the $!$-unit.
This composition defines an element
\[
e(\bfT^*[-1]\bs{\mfrakY}) \in \mH^{\mathrm{BM}}_{2 \vdim \bs{\mfrakY}}(\mfrakY).
\]
We conjecture that this is equal (up to the sign $(-1)^{\vdim \bs{\mfrakY} \cdot (\vdim \bs{\mfrakY} - 1) /2}$) to the stacky virtual fundamental class recently constructed by \cite{AP} and \cite{Kha}.
\end{rem}

\begin{rem}
  By \cite[Theorem 2.2]{Cal}, the zero section $\bs{Y} \hookrightarrow \bfT^*[-1]\bs{Y}$
  carries a canonical Lagrangian structure.
  Further, arguing as Example \ref{ex:canori}, we see that this Lagrangian structure admits a canonical orientation with respect to
  $o^{}_{\bfT^*[-1] \bs{Y}}$.
  The isomorphism \eqref{eq:virtualThom} can be regarded as \cite[Conjecture 5.18]{ABB}
  for this oriented Lagrangian structure.
\end{rem}

\appendix

\section{Remarks on the determinant functor}\label{ap:A}

In this appendix, we prove some results on the determinant of perfect complexes.
All results follow easily from \cite{KM},
but we include this for completeness and to fix the sign conventions.

Let $X$ be a scheme.
Denote by $\mathscr{P}is_X$ the category of invertible sheaves on $X$ with the isomorphisms,
and $\mathscr{P}_{\mathrm{gr}}is_X$ by the category of locally $\mathbb{Z}/2\mathbb{Z}$-graded invertible sheaves with the isomorphisms defined as follows:
\begin{itemize}
  \item Objects of $\mathscr{P}_{\mathrm{gr}}is_X$ are pairs $(L, \alpha)$
        where $L$ is an invertible sheaf on $X$ and $\alpha$ is a locally constant $\mathbb{Z}/2\mathbb{Z}$-valued function.
  \item A morphism from $(L, \alpha)$ to $(M, \beta)$ is an isomorphism form $L$ to $M$ when $\alpha = \beta$,
        and otherwise there is no morphism between them.
\end{itemize}
If there is no confusion, we omit the local grading.
We define a monoidal structure on $\mathscr{P}_{\mathrm{gr}}is_X$ by
\[(L, \alpha) \otimes (M, \beta) \coloneqq (L \otimes M, \alpha + \beta),
\]
with the monoidal unit $(\mathcal{O}_X, 0)$ and the obvious associator.
By the Koszul sign rule with respect to the local grading,
we define the symmetrizer
\begin{align}\label{eq:symmon}
s^\flat_{(L, \alpha), (M, \beta)} \colon (L, \alpha) \otimes (M, \beta) \cong (M, \beta) \otimes (L, \alpha).
\end{align}
This makes $\mathscr{P}_{\mathrm{gr}}is_X$ a symmetric monoidal category.
In this paper we do not equip $\mathscr{P}_{\mathrm{gr}}is_X$ with any other symmetric monoidal structure.
Note that the forgetful functor $\mathscr{P}_{\mathrm{gr}}is_X \to \mathscr{P}is_X$ is monoidal but not symmetric
monoidal with respect to the standard symmetric monoidal structure on $\mathscr{P}is_X$.
For $(L, \alpha) \in \mathscr{P}_{\mathrm{gr}}is_X$, define its (right) inverse by $(L, \alpha)^{-1} \coloneqq (L^{-1}, - \alpha)$,
and define morphisms $\delta^\flat_{(L, \alpha)}, (\delta_{(L, \alpha)}')^\flat$ as follows:
\begin{align*}
  \delta^\flat_{(L, \alpha)} &\colon  (L, \alpha) \otimes (L, \alpha)^{-1} \cong (L \otimes L^{-1}, 0) \cong (\mathcal{O}_X, 0)\\
  (\delta_{(L, \alpha)}')^\flat &\colon (L, \alpha)^{-1} \otimes (L, \alpha) \xrightarrow{s^\flat_{(L, \alpha)^{-1}, (L, \alpha)}} (L, \alpha) \otimes (L, \alpha)^{-1}
  \xrightarrow{\delta^\flat_{(L, \alpha)}} (\mathcal{O}_X, 0).
\end{align*}
  Define
  \[
  \mu^\flat_{(L, \alpha)} \colon ((L, \alpha)^{-1})^{-1} \to (L, \alpha)
  \] so that the following diagram commutes:
  \[
    \xymatrix{
    ((L, \alpha)^{-1})^{-1} \otimes (L, \alpha)^{-1}  \ar[rr]^-{\mu^\flat_{(L, \alpha)} \otimes \id} \ar[dr]_{(\delta_{{(L, \alpha)}^{-1}}')^\flat} & & (L, \alpha) \otimes (L, \alpha)^{-1} \ar[dl]^{\delta^\flat_{(L, \alpha)}} \\
          & (\mathcal{O}_X, 0). &
    }
  \]
Note that the map $\mu^\flat_{(L, \alpha)}$ differs from the natural isomorphism of the ungraded line bundles $(L^{-1})^{-1} \to L$ by the sign $(-1)^\alpha$.
For $L, M \in \mathscr{P}_{\mathrm{gr}}is_X$, define
\[
\theta^\flat_{L, M}\colon (L \otimes M)^{-1} \to L^{-1} \otimes M^{-1}
\] so
that the following diagram commutes:
\[
  \xymatrix@C=70pt{
  (L\otimes M) \otimes (L \otimes M)^{-1} \ar[r]^-{\delta^\flat_{L \otimes M}} \ar[d]_{\id_{L \otimes M} \otimes \theta^\flat_{L, M}} & \mathcal{O}_X \\
  (L\otimes M) \otimes (L^{-1} \otimes M^{-1}) \ar[r]^{\id_L \otimes s^\flat_{M, L^{-1}} \otimes \id_{M^{-1}}} & (L \otimes L^{-1}) \otimes (M \otimes M^{-1})
  \ar[u]_{\delta^\flat_L \otimes \delta^\flat_M}.
  }
\]
Note that the map $\theta^\flat_{(L, \alpha), (M, \beta)}$ differs from tha natural isomorphism
$(L \otimes M)^{-1} \to L^{-1} \otimes M^{-1}$ defined by using the standard symmetric monoidal structure on
the category of ungraded line bundles by $(-1)^{\alpha \cdot \beta}$.

Write $\mathscr{C}_X^\bullet$ the category of bounded complexes of finite locally free $\mathcal{O}_X$-modules,
and $\mathscr{C}is_X^\bullet$ the subcategory of $\mathscr{C}_X^\bullet$ with the same objects and the morphisms are the quasi-isomorphisms.
For a locally free $\mathcal{O}_X$-module $F$, define  a graded line bundle $\det^\flat(F) \in \mathscr{P}_{\mathrm{gr}}is_X$ by
\[
  \det^\flat(F) \coloneqq (\wedge^{\mathrm{rank}(F)} F, \mathrm{rank}(F) \ \mathrm{mod}\ 2).
\]
Clearly, $\det^\flat$ is functorial with respect to isomorphisms.
For $F^\bullet \in \mathscr{C}is_X^\bullet$, define a graded line bundle $\det^\flat (F^\bullet) \in \mathscr{P}_{\mathrm{gr}}is_X$ by
\[
  \det^\flat(F^\bullet) \coloneqq (\cdots \otimes \det^\flat(F^i)^{(-1)^i} \otimes \det^\flat(F^{i-1})^{(-1)^{i-1}} \otimes \cdots) .
\]
In \cite[Theorem~1]{KM}, it is shown that $\det^\flat$ extends naturally to a functor $\mathscr{C}is_X^\bullet \to \mathscr{P}_{\mathrm{gr}}is_X$, which we also write as $\det^\flat$.
Define a functor $\det \colon \mathscr{C}is_X^\bullet \to \mathscr{P}is_X$
by the composition
\[
\mathscr{C}is_X^\bullet \xrightarrow{\det ^\flat} \mathscr{P}_{\mathrm{gr}}is_X \to \mathscr{P}is_X
\]
where the latter functor is the forgetful one.
For a short exact sequence $0 \to E^\bullet \xrightarrow{u^\bullet} F^\bullet \xrightarrow{v^\bullet} G^\bullet \to 0$ in $\mathscr{C}_X^\bullet$,
define
\[
i^\flat(u^\bullet, v^\bullet) \colon \det^\flat(E^\bullet) \otimes \det^\flat(G^\bullet) \cong \det^\flat(F^\bullet)
\] by the following composition:
\begin{align*}
    \det^\flat(E^\bullet) \otimes \det^\flat(G^\bullet) &= (\bigotimes_{i} \det^\flat(E^i)^{(-1)^i}) \otimes (\bigotimes_i \det^\flat(G^i)^{(-1)^i}) \\
    &\overset{(\mathrm{i})}{\cong} \bigotimes_i(\det^\flat(E^i)^{(-1)^i} \otimes \det^\flat(G^i)^{(-1)^i}) \\
    &\overset{(\mathrm{ii})}{\cong} \bigotimes_i(\det^\flat(E^i) \otimes \det^\flat(G^i))^{(-1)^i} \\
    &\overset{(\mathrm{iii})}{\cong}\bigotimes_i\det^\flat(F^i)^{(-1)^i} = \det^\flat(F^\bullet).
\end{align*}
Here (i) is defined by the symmetric monoidal structure on $\mathscr{P}_{\mathrm{gr}}is_X$,
(ii) is defined using $\theta^\flat_{\det(E^i), \det(G^i)}$, and (iii) is defined by the natural isomorphisms
$\det(E_i) \otimes \det(G_i) \cong \det(F_i) $.
We define $i(u^\bullet, v^\bullet) \colon \det(E^\bullet) \otimes \det(G^\bullet) \cong \det(F^\bullet)$
by forgetting the local grading from $i^\flat(u^\bullet, v^\bullet)$.

\subsection{Compatibility with the derived dual functor}\label{subsection:dual}

For a free $\mathcal{O}_X$-module $E$ with a fixed basis $e_1,\ldots,e_n$, define
$\eta_E \colon \det(E^\vee) \xrightarrow{\cong} \det(E)^{-1}$ by the rule
\begin{align}\label{eq:dualconv2}
  e_1^\vee \wedge \cdots \wedge e_n^\vee \mapsto (e_n \wedge \cdots \wedge e_1)^\vee
\end{align}
 where $e_1^\vee,\ldots,e_n^\vee$ denotes the dual basis of $e_1,\ldots,e_n$ and $(e_n \wedge \cdots \wedge e_1)^\vee$
 denotes the dual of $e_n \wedge \cdots \wedge e_1$.
 Clearly $\eta_E$ is independent of the choice of the basis, and we can define $\eta_E$ for any locally free $\mathcal{O}_X$-module.
 For $E^\bullet \in \mathscr{C}_X^\bullet$, define
 \[
 (\eta_{E^\bullet}^{\flat})' \colon \det^\flat({E^\bullet}^\vee) \cong \det^\flat(E^\bullet)^{-1}
 \]
  by the following composition:
\begin{align*}
    \det^\flat({E^\bullet}^\vee) = \bigotimes_i \det^\flat((E^{-i})^\vee)^{(-1)^i}
    \overset{(\mathrm{i})}{\cong} \bigotimes_i (\det^\flat(E^{-i})^{(-1)^i})^{-1}
     \overset{(\mathrm{ii})}{\cong} (\det^\flat(E^\bullet))^{-1}.
\end{align*}
Here (i) is defined by $\eta_{E_i}$ and (ii) is defined by iterating $\theta^\flat$.
Write $\epsilon(E^\bullet) \coloneqq \sum_{i \equiv 1, 2 \mathrm{mod}4} \mathrm{rank}(E^i)$
and define
\[
\eta_{E^\bullet}^\flat \coloneqq (-1)^{\epsilon(E^\bullet)}(\eta_{E^\bullet}^\flat)'.
\]

\begin{prop}\label{prop:detdual}
  \begin{itemize}
    \item[(i)] For a short exact sequence $0 \to E^\bullet \xrightarrow{u^\bullet} F^\bullet \xrightarrow{v^\bullet} G^\bullet \to 0$ in $\mathscr{C}_X^\bullet$, the following diagram commutes:
    \[
      \xymatrix@C=25pt{
      \det^\flat({G^\bullet} ^\vee) \otimes \det^\flat({E^\bullet} ^\vee) \ar[r]_-{\eta^\flat_{G^\bullet}\otimes \eta^\flat_{E^\bullet}} \ar[d]_-{i((v^\bullet)^\vee,(u^\bullet)^\vee)}
      & \det^\flat({G^\bullet})^{-1} \otimes \det^\flat({E^\bullet})^{-1} \ar[r]_(.6){s^\flat_{\det^\flat({G^\bullet})^{-1}, \det^\flat({E^\bullet})^{-1}}}
      & \det^\flat({E^\bullet})^{-1} \otimes \det^\flat({G^\bullet})^{-1}  \\
      \det^\flat({F^\bullet} ^\vee) \ar[r]_-{\eta^\flat_{F^\bullet}}
      & \det^\flat(F^\bullet)^{-1} \ar[r]_(0.37){(i^\flat(u^\bullet, v^\bullet))^{\otimes^{-1}}}
      & (\det^\flat({E^\bullet}) \otimes \det^\flat({G^\bullet}))^{-1} \ar[u]_-{\theta^\flat_{\det^\flat({E^\bullet}), \det^\flat({E^\bullet})}}.
      }
  \]

  \item[(ii)] For a quasi-isomorphism $u^\bullet \colon E^\bullet \to F^\bullet$ in $\mathscr{C}is_X$,
    the following diagram commutes:
    \[
      \xymatrix{
      \det^\flat({F^\bullet} ^\vee) \ar[r]^-{\eta^\flat_{F^\bullet}} \ar[d]_-{\det^\flat((u^\bullet)^\vee)}
      & \det^\flat({F^\bullet})^{-1} \ar[d]^-{\det^\flat(u^\bullet)^{\otimes^{-1}}} \\
      \det^\flat({E^\bullet} ^\vee) \ar[r]^-{\eta^\flat_{E^\bullet}}
      & \det^\flat({E^\bullet})^{-1}.
      }
    \]
  \end{itemize}
\end{prop}

\begin{proof}
  \begin{itemize}
    \item[(i)] Clearly we may assume that these three complexes are concentrated in a single degree $i$.
    Then the claim follows from a direct computation.
    \item[(ii)] Arguing as the proof of \cite[Lemma~2]{KM} and using (i),
    we may assume that $E^\bullet$ is an acyclic complex and $F^\bullet = 0$.
    Further, by localizing $X$ if necessary and using (i),
    we may assume that $E^\bullet$ has length two,
    but then the claim follows from a direct computation (note the sign convention of the dual complex).
  \end{itemize}
\end{proof}

In \cite[Theorem~2]{KM}, the determinant functor is defined for the category of perfect complexes with quasi-isomorphisms.
By using the proposition above, we can define
$\eta^\flat_{E} \colon \det^\flat({E}^\vee) \xrightarrow{\cong} \det^\flat(E)^{-1}$
for any perfect complex $E$.
We define
\[
\eta_{E} \colon \det({E}^\vee) \xrightarrow{\cong} \det(E)^{-1}
\] by forgetting the local grading from $\eta^\flat_{E}$.
\subsection{Compatibility with the shift functor}\label{subsection:shift}

For $E^\bullet \in \mathscr{C}is_X$, define
\[
\chi^\flat_{E^\bullet} \colon \det^\flat(E^\bullet[1]) \cong \det^\flat(E^\bullet)^{-1}
\] by the following composition:
\begin{align*}
  \det^\flat(E^\bullet[1]) &=  \bigotimes_i \det^\flat(E^{i+1})^{(-1)^i}
  \overset{(\mathrm{i})}{\cong}\bigotimes_i (\det^\flat(E^{i+1})^{(-1)^{i+1}})^{-1}
 \overset{(\mathrm{ii})}{\cong} \det^\flat(E^\bullet)^{-1}.
\end{align*}
Here (i) is defined by using $\mu^\flat_{\det^\flat(E_i)}$ and (ii) is defined by iterating $\theta^\flat$.
The following proposition can be shown similarly to Proposition \ref{prop:detdual}:

\begin{prop}
  \begin{itemize}
    \item[(i)] For a short exact sequence
    $0 \to E^\bullet \xrightarrow{u^\bullet} F^\bullet \xrightarrow{v^\bullet} G^\bullet \to 0$ in $\mathscr{C}_X^\bullet$,
     the following diagram commutes:
    \[
      \xymatrix@C=40pt{
      \det^\flat({E^\bullet}[1]) \otimes \det^\flat({G^\bullet}[1]) \ar[rr]^{\chi^\flat_{E^\bullet}\otimes \chi^\flat_{G^\bullet}} \ar[d]_-{i^\flat(u^\bullet[1],v^\bullet[1])}
      & {}
      & \det^\flat({E^\bullet})^{-1} \otimes \det^\flat({G^\bullet})^{-1}  \\
      \det^\flat(F^\bullet[1]) \ar[r]^-{\chi^\flat_{F^\bullet}}
      & \det^\flat(F^\bullet)^{-1} \ar[r]^(.4){i^\flat(u^\bullet, v^\bullet)^{\otimes^{-1}}}
      & (\det^\flat({E^\bullet}) \otimes \det^\flat({G^\bullet}))^{-1} \ar[u]_-{\theta^\flat_{\det^\flat({E^\bullet}), \det^\flat({E^\bullet})}}.
      }
  \]

  \item[(ii)] For a quasi-isomorphism $u^\bullet \colon E^\bullet \to F^\bullet$ in $\mathscr{C}is_X$,
    the following diagram commutes:
    \[
      \xymatrix{
      \det^\flat({E^\bullet}[1]) \ar[r]^-{\chi^\flat_{E^\bullet}} \ar[d]_-{\det^\flat(u^\bullet[1])}
      & \det^\flat({E^\bullet})^{-1} \ar[d]^-{(\det^\flat(u^\bullet)^{\otimes^{-1}})^{-1}} \\
      \det^\flat({F^\bullet}[1]) \ar[r]^-{\chi^\flat_{F^\bullet}}
      & \det^\flat({F^\bullet})^{-1}.
      }
    \]
  \end{itemize}
\end{prop}

This proposition implies that we can define $\chi^\flat_{E } \colon \det^\flat(E [1]) \cong \det^\flat(E )^{-1}$ for any perfect complex $E $.
We define
\[
\chi_{E } \colon \det(E [1]) \cong \det(E )^{-1}
\] by forgetting the local grading from
$\chi^\flat_{E}$.
We also define
\[
\chi_{E }^{(n)} \colon \det(E [n]) \cong \det(E )^{(-1)^n}
\] for each $n \in \mathbb{Z}$
so that $\chi_{E }^{(1)} =\chi_{E }$ and $\chi_{E }^{(n+m)} = (\chi_{E[n]}^{(m)})^{(-1)^{^{\otimes ^n}}} \circ \chi_{E }^{(n)}$ holds for each $n,  m \in \mathbb{Z}$,
where we identify $(\det(E)^{(-1)^n})^{(-1)^m}$ and $\det(E)^{(-1)^{n + m}}$ by using $\mu^{\flat}_{\det(E)^\flat}$
if both $n$ and $m$ are odd.

\subsection{Compatibility with distinguished triangles}

Consider a distinguished triangle $E \to F \to G \to E[1]$ of perfect complexes on $X$.
In \cite{KM} it is observed that there exists an isomorphism $\det^\flat (E) \otimes \det^\flat(G) \cong \det^\flat(F)$,
though there is no natural choice in general\footnote{The essential reason is the non-functoriality of the mapping cone.
See \cite[\textsection 5]{BS} for an approach via the $\infty$-categorical enhancement.}.
However, it is also observed in \cite{KM} that there is a canonical choice when $X$ is reduced:

\begin{prop}\cite[Proposition 7]{KM}
  For each distinguished triangle of perfect complexes $\Delta \colon E \xrightarrow{u} F \xrightarrow{v} G \xrightarrow{w} E[1]$ on a reduced scheme $X$,
  there exists a unique isomorphism
  \[
  i^\flat(\Delta) = i^\flat(u, v, w) \colon \det^\flat (E) \otimes \det^\flat(G) \cong \det^\flat(F)
  \] characterized by the following properties:
  \begin{itemize}
  \item If $E \xrightarrow{u} F \xrightarrow{v} G \xrightarrow{w} E[1]$ is represented by a short exact sequence of complexes of locally free sheaves $0 \to E^\bullet \xrightarrow{u^\bullet} F^\bullet \xrightarrow{v^\bullet} G^\bullet  \to 0$, then $i^\flat(u, v, w) = i^\flat(u^\bullet, v^\bullet)$.
  \item If there exists a morphism of reduced schemes $f \colon Y \to X$, then $f^* i^\flat(u, v, w) = i(f^*u, f^*v, f^*w)$.
\end{itemize}
\end{prop}
We define $i(\Delta) = i(u, v, w) \colon \det(E) \otimes \det(G) \cong \det(F)$ by forgetting the local grading from $i^\flat(u, v, w)$.
The following statement follows by a direct computation:

\begin{prop}\label{prop:detrot}
Let $X$ be a reduced scheme, $\Delta \colon E \xrightarrow{u} F \xrightarrow{v} G \xrightarrow{w} E[1]$ a distinguished triangle of perfect complexes over $X$, and $\Delta' \colon  F \xrightarrow{v} G \xrightarrow{w} E[1] \xrightarrow{-u[1]} F[1]$ the rotated triangle.
Then the following diagram commutes:
\[
\xymatrix@C=8pt{
\det(E) \otimes \det(F) \otimes \det(E[1]) \ar[rrr]^-{\id \otimes i(\Delta')} \ar[d]_{\id \otimes \chi_E} & & & \det(E) \otimes \det(G) \ar[r]^-{i(\Delta)} & \det(F) \ar[d] \\
\det(E) \otimes \det(F) \otimes \det(E)^{-1} \ar[rrrr]^-{\cdot (-1)^{\rank(E)\rank(F)}} & & & & \det(E) \otimes \det(F) \otimes \det(E)^{-1}.
}
\]
Here the right vertical map is defined by unit map for the standard symmetric monoidal structure on $\mathscr{P}is_X$.
\end{prop}

Let $k$ be a field, and $E$ be a perfect complex over $\Spec k$.
Write $\mH^*(E) \coloneqq \bigoplus_i \mH^i(E)[-i]$, considered as a complex with zero differential.
The natural isomorphism $\mH^*(E) \simeq E$ in $D(k)$ induces an isomorphism
\[
  j_E \colon \det(\mH^*(E)) \cong \det(E).
\]
Let $E \xrightarrow{u} F \xrightarrow{v} G \xrightarrow{w} E[1]$ be a distinguished triangle of perfect complexes over $\Spec k$.
By decomposing the long exact sequence induced by the above distinguished triangle into short exact sequences,
we can construct
\[
  i'_{\mH}(u, v, w) \colon \det(\mH^*(E)) \otimes \det(\mH^*(G)) \to \det(\mH^*(F)).
\]
Define
\[
  i'(u, v, w) \coloneqq j_F \circ  i'_{\mH}(u, v, w) \circ (j_E^{-1} \otimes j_G^{-1}) \colon \det(E) \otimes \det(G) \cong \det(F).
\]
The maps $i(u, v, w)$ and $i'(u, v, w)$ do not coincide in general because we used the symmetric monoidal structure on $\mathscr{P}_{\mathrm{gr}}is_X$ to construct $i(u, v, w)$.
However we can explicitly write down the difference between $i(u, v, w)$ and $i'(u, v, w)$ as follows:
write $a_i \coloneqq \rank \ker \mH^i(u)$, $b_i \coloneqq \rank \ker \mH^i(v)$, and $c_i \coloneqq \rank \ker \mH^i(w)$.
Then we have
\begin{align}\label{eq:extsign}
  \begin{split}
    i(u, v, w) &= (-1)^T i'(u, v, w), \text{where} \\
  T = \sum_{i \colon \text{even}}(\sum_{j \leq i-1}a_i a_j &+ \sum_{j \leq i} a_i b_j  +   \sum_{j \leq i+1} c_i a_j +\sum_{j \leq i-1} c_i b_j ) \\
   + \sum_{i \colon \text{odd}}(\sum_{j \leq i}a_i a_j &+ \sum_{j \leq i-1} a_i b_j + \sum_{j \leq i} c_i a_j + \sum_{j \leq i} c_i b_j).
 \end{split}
\end{align}

\subsection{Extension to Artin stacks}\label{subsection:stkdet}

Since the determinant functor commutes with base change,
we can extend the determinant functor to the category of perfect complexes on Artin stacks with quasi-isomorphisms.
Clearly, $\eta_E$ and $\chi_E$ commutes with base change,
we can define $\eta_{E} \colon \det(E^\vee) \cong \det(E)^{-1}$ and $\chi_{E} \colon \det(E[1]) \cong \det(E)^{-1}$
for any perfect complex $E$ over an Artin stack $\mathfrak{X}$.
By a similar reasoning, we can define $i(u, v, w) \colon \det(E) \otimes \det(G) \cong \det(F)$ for a distinguished triangle of perfect complexes $E \xrightarrow{u} F \xrightarrow{v} G \xrightarrow{w} E[1]$ over a reduced Artin stack $\mathfrak{X}$.

\end{document}